\documentclass[a4paper, 10pt, twoside, notitlepage]{amsart}

\usepackage{amsmath,amscd}
\usepackage{amssymb}
\usepackage{amsthm}
\usepackage{comment}
\usepackage{graphicx, xcolor}

\usepackage{mathrsfs}
\usepackage[ocgcolorlinks, linkcolor=blue]{hyperref}

\usepackage{bm}
\usepackage{bbm}
\usepackage{url}

\usepackage[utf8]{inputenc}
\usepackage{mathtools}
\usepackage{esint}
\usepackage{tikz}
\usepackage{dsfont}
\usepackage{relsize}
\usepackage[shortlabels]{enumitem}
\usepackage{lineno}
\usepackage{verbatim}
\numberwithin{equation}{section}

\allowdisplaybreaks

\mathtoolsset{showonlyrefs}

\graphicspath{{images/}}

\newtheorem{theorem}{Theorem}[section]
\newtheorem{proposition}[theorem]{Proposition}
\newtheorem{lemma}[theorem]{Lemma}
\newtheorem{corollary}[theorem]{Corollary}
\newtheorem{example}[theorem]{Example}

\theoremstyle{definition}
\newtheorem{definition}[theorem]{Definition}

\theoremstyle{remark}
\newtheorem{remark}[theorem]{Remark}

\title[Recovery of rough sources in the inverse MA problem]
{Recovery of Rough Sources in the Inverse Monge--Amp\`ere Problem}

\author[G. Uhlmann]{Gunther Uhlmann}
\address{Department of Mathematics, University of Washington, Seattle, WA 98195-4350, USA}
\email{gunther@math.washington.edu}

\author[P. Zimmermann]{Philipp Zimmermann}
\address{Institute of Mathematics, École Polytechnique Fédérale de Lausanne (EPFL), Lausanne, Switzerland}
\email{philipp.zimmermann@epfl.ch}

\date{\today}


\newcommand{\R}{{\mathbb R}}

\newcommand{\N}{{\mathbb N}}

\newcommand{\eps}{\varepsilon}

\newcommand{\cL}{\mathcal L}
\newcommand{\cC}{\mathcal C}

\DeclareMathOperator*{\argmax}{arg\,max}

\DeclareMathOperator{\Div}{div} 
\DeclareMathOperator{\dist}{dist} 

\newcommand{\weakstar}{\overset{\ast}{\rightharpoonup}}

\newcommand{\res}{\mathbin{\vrule height 1.6ex depth 0pt width
0.13ex\vrule height 0.13ex depth 0pt width 1.3ex}}

\begin{document}

\begin{abstract}

We prove that every nonnegative $L^1$ source in the Monge--Amp\`ere equation
\[
   \det D^2u=f\quad\text{in }\Omega,
\]
on a bounded uniformly convex $C^{1,1}$ domain $\Omega\subset\mathbb R^n$, $n\geq2$, is uniquely determined by generalized boundary measurements of Alexandrov solutions. The measurement is the extremal boundary slope, which extends the classical outward normal derivative and may take the value $+\infty$. For each direction $\omega\in\mathbb S^{n-1}$, it suffices to know these slopes for cylindrical Dirichlet data $\frac{t_k^\omega}{2}|P_\omega x|^2$ along any sequence $t_k^\omega\to\infty$, where $P_\omega$ projects onto $\omega^\perp$. The boundary value and extremal slope determine the interior subgradient image up to a Lebesgue null set. An exact change of variables formula then converts equality of these images into integral identities for the source, and a large-amplitude limit recovers its Euclidean X-ray transform. This yields uniqueness without a positive lower bound on the source. For nonnegative $L^\infty$ sources, the extremal slopes are bounded and agree with the interior normal derivatives, giving a Dirichlet-to-interior-Neumann formulation. We also extend the inverse argument to the gradient-weighted equation
\[
   \vartheta(\nabla u)\det D^2u=f,
\]
where the known weight $\vartheta$ is locally integrable, Borel measurable, and bounded below by a positive constant. We prove existence of the corresponding Alexandrov solutions and show that the same cylindrical boundary-slope data uniquely determine every nonnegative $L^1$ source.

	\end{abstract}
	\maketitle

	\tableofcontents

\section{Introduction}
\label{sec:intro}

We study the recovery of an unknown nonnegative source $f$ in the Monge--Amp\`ere problem
\begin{equation}
\label{eq:dirichlet}
\begin{cases}
\det D^2u=f & \text{in }\Omega,\\
u=\varphi & \text{on }\partial\Omega,
\end{cases}
\end{equation}
from boundary measurements obtained by varying $\varphi$. Our aim is to establish uniqueness for sources $f\in L^1(\Omega)$, allowing both roughness and degeneracy. In this regime, the equation is understood in the Alexandrov sense, and the classical normal derivative is replaced by an extremal boundary slope. We prove that these generalized measurements determine $f$ in every dimension $n\geq 2$. Moreover, only cylindrical boundary values along an unbounded sequence of amplitudes in each direction are needed. We further show that the inverse argument is robust under the introduction of a known positive gradient weight: it also applies to
\begin{equation}
\label{eq:intro-gradient-weighted-MA}
\begin{cases}
\vartheta(\nabla u)\det D^2u=f & \text{in }\Omega,\\
u=\varphi & \text{on }\partial\Omega,
\end{cases}
\end{equation}
provided that $\vartheta$ is a locally integrable Borel function bounded below by a positive constant. We establish the required forward existence result in Section~\ref{sec:weighted-MA-extension}.

For linear elliptic equations, varying the Dirichlet value generally does not remove the nonuniqueness of an inverse source problem; see \cite{isakov1990inverse}. For example, if $u$ solves $-\Delta u=f$ and $v\in C_c^\infty(\Omega)$, then $u+v$ solves the equation with source $f-\Delta v$ and has the same Dirichlet and Neumann traces as $u$. This ambiguity persists for every boundary excitation. For the Monge--Amp\`ere equation, however, the dependence of the solution on the boundary value is nonlinear. Different boundary values produce different gradient geometries through which the source can be probed. Recent works exploit this mechanism to establish source uniqueness in the smooth setting \cite{LiimatainenLin,CarsteaGhosh,UhlmannZimmermann}. Here we develop a convex-analytic formulation that extends the large-amplitude approach to nonnegative $L^1$ sources and, through a weighted subgradient-image measure, to the gradient-dependent equation \eqref{eq:intro-gradient-weighted-MA}.

\subsection{Alexandrov solutions and boundary measurements}

For the inverse problem, we assume that $\Omega\subset\mathbb R^n$, $n\geq2$, is a bounded uniformly convex domain with $C^{1,1}$ boundary. For a smooth positive source and smooth boundary data on a smooth uniformly convex domain, classical Dirichlet theory gives a unique smooth convex solution; see \cite{CNS}. Under the weaker assumption $f\in L^1(\Omega)$, $f\geq0,$ we instead use Alexandrov solutions \cite{AL4}. For a convex function $u\colon\Omega\to\mathbb R$, its Monge--Amp\`ere measure is defined by
\[
\mu_u(A)\vcentcolon=\mathcal L^n(\partial u(A)),\qquad A\subset\Omega,
\]
where $\partial u(A)$ is the union of the subdifferentials at points of $A$. A convex $u\in C(\overline\Omega)$ solves \eqref{eq:dirichlet} in the Alexandrov sense if $\mu_u=f\,dx$ as Borel measures on $\Omega$ and $u|_{\partial\Omega}=\varphi$. This formulation requires neither pointwise second derivatives nor strict positivity of $f$. The Dirichlet problem has a unique solution for every continuous boundary value; see Section~\ref{sec:Alexandrov-sol}.

To define the boundary measurement, let $\varphi\in C^1(\partial\Omega)$ and extend $u$ to the proper lower semicontinuous convex function $\bar u$ that equals $u$ on $\overline\Omega$ and $+\infty$ outside $\overline\Omega$. Write $\bar\partial u(x)\vcentcolon=\partial\bar u(x)$. At a boundary point $x$, the tangential component of any supporting vector is fixed by $\varphi$, while its normal component ranges over a ray. We therefore define the \emph{extremal boundary slope} by
\[
q_u(x)\vcentcolon=
\inf\{a\in\mathbb R:\nabla_\tau\varphi(x)+a\nu(x)\in\bar\partial u(x)\},
\qquad \inf\varnothing=+\infty,
\]
where $\nu$ is the outward unit normal and $\nabla_\tau$ is the tangential gradient; see Proposition~\ref{def:slope}. If $q_u(x)<\infty$, then
\[
\bar\partial u(x)=\nabla_\tau\varphi(x)+[q_u(x),\infty)\nu(x).
\]
If $q_u(x)=+\infty$, the boundary subdifferential is empty. These results are established in Lemma~\ref{lem:ray}. Moreover, the function $q_u$ is Borel measurable and agrees with $\partial_\nu u$ whenever $u\in C^1(\overline\Omega)$; see Section~\ref{sec:one-side-slope}.

For a nonnegative $f\in L^1(\Omega)$, the \emph{generalized global Cauchy data set} is
\begin{equation}
\label{eq:Global-Cauchy}
\mathcal C_f^{\mathrm{glob}}
\vcentcolon=
\{(\varphi,q_{u_{f,\varphi}}):\varphi\in C^1(\partial\Omega)\},
\end{equation}
where $u_{f,\varphi}$ is the unique Alexandrov solution of \eqref{eq:dirichlet}; see Section~\ref{subsec:Def-Cauchy-data}. The measurements are pointwise boundary functions, possibly taking the value $+\infty$. The term \emph{global} refers to the range of boundary values: they are not restricted to a fixed neighborhood of a background datum. Our uniqueness proof uses only a particular family of smooth boundary values from this data set.

\subsection{Main results}
\label{sec:main-results}

Our first main result shows that the nonnegative $L^1$ source in the Monge--Ampère problem \eqref{eq:dirichlet} is uniquely determined by the extremal boundary slopes $q_u$ corresponding to cylindrical boundary data of arbitrarily large amplitude.

\begin{definition}
For each unit vector $\omega\in\mathbb S^{n-1}$, define $\Phi_\omega\colon\mathbb R^n\to\mathbb R$ by
\begin{equation}
\label{eq:quadratic-tests}
\Phi_\omega(x)\vcentcolon=\frac12|P_\omega x|^2,
\quad\text{with}\quad
P_\omega y\vcentcolon=y-(\omega\cdot y)\omega.
\end{equation}
We call $t\Phi_\omega|_{\partial\Omega}$, $t>0$,
\emph{cylindrical boundary data}.
\end{definition}

The following theorem shows that, for each direction $\omega$, it suffices to know the boundary slopes for a sequence of cylindrical boundary data $t_k^\omega\Phi_\omega|_{\partial\Omega}$ with $t_k^\omega\to\infty$.

\begin{theorem}[Cylindrical-data uniqueness]
\label{thm:cylindrical}
    Let $\Omega\subset\R^n$, $n\geq 2$, be a bounded uniformly convex $C^{1,1}$ domain and let $f_1,f_2\in L^1(\Omega)$ be nonnegative. For $j=1,2$, $\omega\in \mathbb S^{n-1}$, and $t>0$, let $u_{j,t}^{\omega}\in C(\overline{\Omega})$ be the unique Alexandrov solution of
\begin{equation}
\label{eq:MA-main-thm}
\begin{cases}
    \det D^2u=f_j\quad&\text{in }\Omega,\\
     u=t\Phi_\omega
 \quad&\text{on }\partial\Omega,
\end{cases}
\end{equation}
provided by Proposition~\ref{prop:forward}. Suppose that, for every $\omega\in \mathbb{S}^{n-1}$, there is a sequence $(t_k^\omega)_{k\in\mathbb N}\subset(0,\infty)$ with $t_k^\omega\to\infty$ such that the boundary slopes corresponding to the cylindrical data $t_k^\omega \Phi_\omega$ agree:
\begin{equation}
\label{eq:countably-equal-slopes}
 q_{u_{1,t_k^\omega}^{\omega}}  =q_{u_{2,t_k^\omega}^{\omega}}
 \quad\text{on }\partial\Omega,
 \qquad\text{for every }k\in\N;
\end{equation}
see Proposition~\ref{def:slope}. Then $f_1=f_2$ almost everywhere in $\Omega$.
\end{theorem}

Thus the global formulation is an immediate consequence of the restricted measurement theorem.

\begin{corollary}[Global uniqueness]
\label{thm:main}
Let $\Omega\subset\R^n$, $n\geq 2$, be a bounded uniformly convex $C^{1,1}$ domain and let $f_1,f_2\in L^1(\Omega)$ be nonnegative. Assume that the global Cauchy data sets agree:
\[
   \cC_{f_1}^{\mathrm{glob}}=\cC_{f_2}^{\mathrm{glob}};
\]
see Proposition~\ref{def:gen-cauchy-data}. Then $f_1=f_2$ almost everywhere in $\Omega$.
\end{corollary}

For nonnegative $L^1$ sources, the supporting slope may be infinite. For nonnegative $L^\infty$ sources and $\varphi\in C^{1,1}(\partial\Omega)$, we prove that $u_{f,\varphi}$ is Lipschitz on $\overline\Omega$ and that $q_{u_{f,\varphi}}$ is bounded. It also agrees pointwise with the interior normal derivative
\begin{equation}
\label{eq:intro-interior-normal}
\partial_\nu^-u_{f,\varphi}(x)
\vcentcolon=
\lim_{r\downarrow0}
\frac{u_{f,\varphi}(x)-u_{f,\varphi}(x-r\nu(x))}{r};
\end{equation}
see Proposition~\ref{prop:boundary-Lipschitz} and Lemma~\ref{lem:radial-slope}. Consequently, the map
\begin{equation}
\label{eq:DT-one-sided-slope}
\Lambda_f^-\colon C^{1,1}(\partial\Omega)\to\mathcal B_b(\partial\Omega),
\qquad
\Lambda_f^-(\varphi)(x)\vcentcolon=\partial_\nu^-u_{f,\varphi}(x),
\end{equation}
is well-defined, where $\mathcal B_b(\partial\Omega)$ denotes the space of bounded Borel functions. This gives the following formulation of uniqueness.

\begin{corollary}[Dirichlet-to-interior-Neumann formulation]
\label{cor:Linfty}
 Let $\Omega\subset\R^n$, $n\geq 2$, be a bounded uniformly convex $C^{1,1}$ domain and let $f_1,f_2\in L^\infty (\Omega)$ be nonnegative.  If $\Lambda_{f_1}^-=\Lambda_{f_2}^-$ pointwise, that is,
\[
\Lambda_{f_1}^-(\varphi)(x)
=
\Lambda_{f_2}^-(\varphi)(x),\quad x\in\partial\Omega,
\]
for every $\varphi\in C^{1,1}(\partial\Omega)$, then $f_1=f_2$ almost everywhere in $\Omega$.
\end{corollary}

The $L^\infty$ assumption is needed here to identify the extremal boundary slope with the interior normal derivative. The uniqueness result itself already holds for nonnegative $L^1$ sources. We do not assert this identification at arbitrary $L^1$ regularity.

Our second main result shows that the same restricted boundary
measurements determine the source in the gradient-weighted problem
\eqref{eq:intro-gradient-weighted-MA}. Let
$\vartheta\colon\mathbb R^n\to(0,\infty)$ be a locally integrable
Borel function. Denote by $\lambda_\vartheta$ the outer measure
induced by the weighted Lebesgue measure
$\vartheta(p)\,dp$, and define the
\emph{$\vartheta$-weighted Monge--Ampère measure} of a convex
function $u$ by
\[
\mu_u^\vartheta(A)
\vcentcolon=
\lambda_\vartheta(\partial u(A)),
\qquad A\subset\Omega.
\]
A convex function $u\in C(\overline\Omega)$ is an Alexandrov
solution of the weighted Dirichlet problem \eqref{eq:intro-gradient-weighted-MA} if $\mu_u^\vartheta=f\,dx$ as Borel measures on $\Omega$ and it assumes the prescribed boundary value pointwise. The extremal
boundary slope is defined exactly as in the unweighted problem,
because its definition depends only on the convex function and its
boundary value.

\begin{theorem}[Cylindrical-data uniqueness for gradient weights]
\label{thm:weighted-MA-uniqueness}
Let $\Omega\subset\mathbb R^n$, $n\geq2$, be a bounded uniformly
convex $C^{1,1}$ domain, and let
$\vartheta\colon\mathbb R^n\to(0,\infty)$ be a known locally
integrable Borel function satisfying
\begin{equation}
\label{eq:bounds-weight-vartheta}
\vartheta(p)\geq c_\vartheta>0
\qquad\text{for a.e. }p\in\mathbb R^n.
\end{equation}
Let $f_1,f_2\in L^1(\Omega)$ be nonnegative. For
$j=1,2$, $\omega\in\mathbb S^{n-1}$, and $t>0$, let
$u_{j,t}^\omega\in C(\overline\Omega)$ be any Alexandrov solution of
\begin{equation}
\label{eq:weighted-MA-main}
\begin{cases}
\vartheta(\nabla u)\det D^2u=f_j
&\text{in }\Omega,\\
u=t\Phi_\omega
&\text{on }\partial\Omega.
\end{cases}
\end{equation}
Assume that, for every $\omega\in\mathbb S^{n-1}$, there exists a
sequence $(t_k^\omega)_{k\in\mathbb N}\subset(0,\infty)$ with
$t_k^\omega\to\infty$ such that
\[
q_{u_{1,t_k^\omega}^\omega}
=
q_{u_{2,t_k^\omega}^\omega}
\quad\text{on }\partial\Omega,
\qquad k\in\mathbb N.
\]
Then $f_1=f_2$ almost everywhere in $\Omega$.
\end{theorem}

The existence of the solutions appearing in
Theorem~\ref{thm:weighted-MA-uniqueness} is proved in
Theorem~\ref{thm:weighted-forward-existence}. Under the present minimal
assumptions on $\vartheta$, uniqueness for the weighted forward Dirichlet
problem is not asserted. Thus $u_{j,t}^\omega$ may be chosen from a
possibly non-singleton solution set; the conclusion holds for any such
families whose boundary slopes satisfy the stated identities. No upper
bound or continuity assumption on $\vartheta$ is used in the inverse
argument.

\subsection{Relation to earlier work}

Recent inverse boundary results for the Monge--Amp\`ere equation \cite{UhlmannZimmermann, LiimatainenLin,CarsteaGhosh} use either linearization near a background solution or large-amplitude boundary values. The distinction matters both for the available measurements and for the regularity of the source. We restrict the discussion to aspects most relevant to the present work. A more detailed comparison can be found in
\cite{UhlmannZimmermann}.

In our earlier work \cite{UhlmannZimmermann}, we studied equations of the form
\[
\det D^2u=a(x,u,\nabla u)
\]
near a fixed uniformly convex background solution. The first linearization leads to an anisotropic elliptic equation whose principal coefficient is the cofactor matrix of the background Hessian, while higher linearizations probe derivatives of the nonlinearity. In the pure source case, this approach gives uniqueness for positive H\"older continuous sources in dimension two, and in higher dimensions under the corresponding real-analytic hypotheses; see \cite[Corollary~8.1 and Remarks~8.2--8.3]{UhlmannZimmermann}. These results use local perturbations of compatible background Cauchy data. The present theorem uses unbounded-amplitude boundary values and applies to nonnegative $L^1$ sources in every dimension $n\geq2$. It does not establish uniqueness from local data near a fixed background.

Liimatainen and Lin \cite{LiimatainenLin} study source recovery in dimension two for smooth positive sources using a local classical Dirichlet-to-Neumann map and prescribed boundary information about the source. Their approach combines first and second linearizations with planar anisotropic inverse methods.

The closest large-amplitude result is due to C{\^a}rstea and Ghosh \cite{CarsteaGhosh}, who prove uniqueness for smooth strictly positive sources on smooth uniformly convex domains in every dimension $n\geq2$. They use the same cylindrical boundary profiles as in \eqref{eq:quadratic-tests} and extract line integrals of the source from an asymptotic correction to the leading profile. Their argument relies on bounded second derivatives of the chordwise correction, with estimates involving second derivatives of the
source, in the determinant expansion and comparison barriers. Consequently, the argument does not directly extend to rough sources; a lower-regularity version would require additional justification. We avoid these derivative estimates by replacing the correction analysis with an exact identity for Monge--Amp\`ere measures. This permits $L^1$ sources that may vanish, $C^{1,1}$ boundaries, and solutions without classical normal derivatives. It also shows that one unbounded sequence of amplitudes in each direction is sufficient.

\subsection{Strategy of the uniqueness proof}

The uniqueness proof begins with studying the interior subgradient image
\[
\mathrm D_u\vcentcolon=\partial u(\Omega).
\]
We show that the boundary value $\varphi$ and the extremal boundary slope $q_u$ determine $\mathrm D_u$ up to a Lebesgue null set. The ambiguity is confined to the boundary slope image
\[
\mathrm S_u\vcentcolon=
\{\nabla_\tau\varphi(x)+q_u(x)\nu(x):x\in\partial\Omega,\ q_u(x)<\infty\},
\]
whose nullity follows from the measurability of $q_u$ and the Lipschitz regularity of the boundary geometry.

Next, absolute continuity of $\mu_u=f\,dx$ ensures that the nondifferentiability set of $u$ has zero Monge--Amp\`ere measure. This yields the exact change of variables formula
\[
\int_\Omega H(\nabla u(x))f(x)\,dx
=
\int_{\mathrm D_u}H(p)\,dp
\]
for every bounded Borel function $H$. Equal boundary values and slopes therefore give equal weighted gradient integrals, even though the corresponding solutions and sources may differ.

Finally, for the boundary value $t\Phi_\omega$, the normalized solution $U_t=t^{-1}u_t$ has source $t^{-n}f$ and fixed boundary value $\Phi_\omega$. Stability of Alexandrov solutions under source perturbations gives $U_t\to\Phi_\omega$ uniformly, and convexity gives
\[
\nabla U_t(x)\to P_\omega x
\qquad\text{for almost every }x\in\Omega.
\]
Applying the integral identity to the normalized solutions with $H(p)=\psi(P_\omega p)$ and passing to the limit gives
\[
\int_\Omega\psi(P_\omega x)(f_1-f_2)(x)\,dx=0,
\qquad \psi\in C_c(\omega^\perp).
\]
Thus the X-ray transform of the zero extension of $f_1-f_2$ vanishes in every direction. The Fourier-slice identity and injectivity of the Fourier transform prove uniqueness.

Absolute continuity enters through the change of variables formula, rather than through a smoothness assumption on the solution. Singular Monge--Amp\`ere measures may charge its nondifferentiability set, so the single-valued-gradient argument need not apply. We do not claim that absolute continuity is a sharp uniqueness threshold; Section~\ref{sec:scope-limitations} discusses this limitation.

For the gradient-weighted equation \eqref{eq:intro-gradient-weighted-MA}, we replace Lebesgue measure in gradient space by the known measure $\vartheta(p)\,dp$ and define
\[
\mu_u^\vartheta(A)
=
\int_{\partial u(A)}\vartheta(p)\,dp.
\]
The corresponding change-of-variables formula is
\[
\int_\Omega H(\nabla u(x))f(x)\,dx
=
\int_{\mathrm D_u}H(p)\vartheta(p)\,dp.
\]
Thus equality of the boundary values and slopes again gives equality of the source integrals used in the large-amplitude argument. The positive lower bound for $\vartheta$ also forces the ordinary Monge--Amp\`ere masses of the normalized solutions to vanish as the amplitude tends to infinity. Consequently, the same cylindrical limit and the same Euclidean X-ray transform arise despite the gradient dependence of the equation.

\subsection{Structure of the paper}

Section~\ref{sec:prelim-Alexandrov-sol} introduces the notation and convex-analytic preliminaries, including a defining-function characterization of uniformly convex $C^{1,1}$ domains. Section~\ref{sec:Alexandrov-sol} reviews Monge--Amp\`ere measures and the Alexandrov Dirichlet theory. Section~\ref{sec:one-side-slope} constructs the extremal boundary slope and compares it with interior and classical normal derivatives. Sections~\ref{sec:det-subgradient-image} and~\ref{sec:Alexandrov-change-of-var} establish the boundary determination of the subgradient image and the integral identity, respectively. Section~\ref{sec:large-amplitude-bdry-values} treats the large-amplitude limit, which is used in Section~\ref{sec:Proof-uniqueness} to prove the uniqueness results for the unweighted equation. Section~\ref{sec:weighted-MA-extension} introduces weighted Monge--Amp\`ere measures and extends the inverse result to known positive gradient weights. Section~\ref{sec:scope-limitations} discusses the scope of the method. The defining-function representation theorem is proved in Appendix~\ref{sec:representation-thm}.

\section{Notation and convex-analytic preliminaries}
\label{sec:prelim-Alexandrov-sol}

In this section we introduce notational conventions and review convex-analytic preliminaries. Section~\ref{subsec:notation} contains our conventions on measures,
which are regarded here as outer measures, and introduces the
classical function spaces used throughout the paper. Sections~\ref{subsec:convex-domains}, \ref{sec:strictly-convex-domains}, \ref{subsec:uniform-convex-domains} discuss convex, strictly convex, and uniformly convex domains, respectively. Section~\ref{subsec:convex-functions} recalls definitions and properties of convex functions.

\subsection{Notation}
\label{subsec:notation}

Throughout the article, $x\cdot y$ and $|x|$ denote the Euclidean inner product and norm on $\R^n$. We write $B_r(x)$ for the open ball of radius $r>0$ and center $x$. To simplify the notation, we will abbreviate $B_r(0)$ by $B_r$. The boundary of the unit ball $\partial B_1$ is 
\[
 \mathbb{S}^{n-1}\vcentcolon =\partial B_1= \{x\in\R^n:|x|=1\}.
\] 

The $n$-dimensional identity matrix is denoted by $I=(\delta_{ij})_{1\leq i,j\leq n}$, where 
\[
    \delta_{ij}=\begin{cases}
        1,\quad&i=j,\\
        0,\quad&i\neq j
    \end{cases}
\]
is the Kronecker delta. For $a,b\in\R^n$, the tensor product $a\otimes b$ is the matrix with entries $(a\otimes b)_{ij}=a_i b_j$. It satisfies 
\[
    (a\otimes b)v=(b\cdot v)a,\quad \text{for all }v\in\R^n.
\]
In particular, for any $\omega\in \mathbb{S}^{n-1}$, 
\[
    Q_\omega\colon \R^n\to\R^n,\quad Q_\omega v \vcentcolon = (\omega\otimes\omega)v=(\omega\cdot v)\omega\in\R^n
\]
is the orthogonal projection onto the linear span of $\omega$, denoted as $\operatorname{span}\{\omega\}$. Thus,
\[
     P_\omega\vcentcolon =I-Q_\omega=I-\omega\otimes\omega
\]
is the orthogonal projection onto the orthogonal complement
\[
    \omega^\perp\vcentcolon = \{x\in\R^n:x\cdot\omega=0\}.
\]

\subsubsection{Measures}
\label{sec:measures}

We denote the $n$-dimensional Lebesgue measure by $\cL^n$ and, as usual, we write $dx$ for integration against it. For a measure $\mu$ on a set $X$, a map $T\colon X\to Y$, and a set $A\subset X$, the symbols $T_\#\mu$ and $\mu\res A$ denote the pushforward of $\mu$ under $T$ and the restriction of $\mu$ to $A$. Explicitly,
\[
    T_\#\mu(B)=\mu(T^{-1}(B))\quad\text{and}\quad (\mu\res A)(C)=\mu(A\cap C)
\]
for any sets $B\subset Y$ and $C\subset X$. Here and elsewhere, we use the usual operations from set theory, including union $A\cup B$, intersection $A\cap B$, complement $A^c=X\setminus A$, difference $A\setminus B=A\cap B^c$, and symmetric difference 
\[
    A\mathbin\triangle B=(A\setminus B)\cup (B\setminus A).
\]
If not otherwise specified, a measure on a set $X$ is an outer measure \cite{evans-mt}, that is, a set function $\mu\colon 2^X\to [0,\infty]$ such that
\begin{enumerate}[(i)]
    \item $\mu(\varnothing)=0$,
    \item $\mu$ is $\sigma$-subadditive, meaning that,
    \[
        A\subset \cup_{k=1}^{\infty}A_k\quad \Longrightarrow \quad \mu(A)\leq \sum_{k=1}^\infty \mu(A_k).
    \]
\end{enumerate}
Recall that in this context a set $A\subset X$ is called $\mu$-measurable, or measurable in the Carathéodory sense, whenever 
\[
    \mu(S)=\mu(S\cap A)+\mu(S\setminus A)
\]
for any set $S\subset X$. It turns out that the class of $\mu$-measurable sets, denoted by $\Sigma_\mu\subset 2^X$, forms a $\sigma$-algebra and hence $\nu\vcentcolon = \mu|_{\Sigma_\mu}\colon \Sigma_\mu \to [0,\infty]$ is a measure in the abstract measure theory sense, meaning that $\nu(\varnothing)=0$ and $\nu$ is $\sigma$-additive, that is, 
\[
    \nu\bigg(\bigcup_{k\geq 1}A_k\bigg)=\sum_{k\geq 1}\nu(A_k)
\]
for every pairwise disjoint countable family $(A_k)_{k\in\N}\subset\Sigma_\mu$. 

Furthermore, we use the following terminology for measures $\mu$ on a locally compact Hausdorff space $X$ (cf., e.g., \cite{evans-mt,federer2014geometric}):
\begin{enumerate}[(a)]
    \item $\mu$ is a \emph{Borel measure}, if every Borel set $B$ is $\mu$-measurable.
    \item $\mu$ is \emph{Borel regular}, if $\mu$ is a Borel measure and for every set $A\subset X$ there exists a Borel set $B$ such that $A\subset B$ and $\mu(A)=\mu(B)$.
    \item $\mu$ is \emph{locally finite}, if $\mu(K)<\infty$ for every compact set $K\subset X$.
    \item $\mu$ is a \emph{Radon measure}, if $\mu$ is a locally finite Borel measure such that 
    \begin{equation}
    \label{eq:inner-regular}
        \mu(V)=\sup_{K\subset V: K\text{ compact}}\mu(K)
    \end{equation}
    for any open set $V\subset X$ and
    \begin{equation}
    \label{eq:outer-regular}
        \mu(A)=\inf_{A\subset U: U\text{ open}}\mu(U)
    \end{equation}
    for any set $A\subset X$.
\end{enumerate}

\begin{remark}
    We note that the outer regularity condition \eqref{eq:outer-regular}, in particular, ensures that $\mu$ is Borel regular. Recall that in the case $X=\R^n$, every locally finite Borel regular measure is automatically a Radon measure \cite[Chapter~1, Theorem~4]{evans-mt}.
\end{remark}

\begin{example}[Hausdorff measure]
An important example of a Borel regular measure is the $s$-dimensional Hausdorff measure $\mathcal{H}^s$ (see~\cite[Section~2]{evans-mt}). For $s\geq 0$ and $A\subset\R^n$, it is defined as
\begin{equation}
\label{eq:Hausdorff-measure}
    \mathcal{H}^s(A)\vcentcolon = \lim_{\delta\to 0}\mathcal{H}^s_{\delta}(A)=\sup_{0<\delta<\infty}\mathcal{H}^s_{\delta}(A),
\end{equation}
where
\[
    \mathcal{H}^s_{\delta}(A)\vcentcolon = \inf\left\{\sum_{j\geq 1}\omega_s\left(
    \frac{\operatorname{diam}C_j}{2}\right)^s : A\subset \bigcup_{j\geq 1}C_j,\,\operatorname{diam} C_j\leq \delta\right\},
\]
with normalization
\[
    \omega_s\vcentcolon =\frac{\pi^{s/2}}{\Gamma(s/2+1)}.
\]
Here, $\Gamma(z)=\int_0^{\infty} t^{z-1}e^{-t}\,dt$, $0<z<\infty$, is the usual Gamma function. In contrast to the Lebesgue measure, $\mathcal{H}^s$ is not a Radon measure for $0\leq s<n$.
\end{example}

\paragraph{\noindent\textit{Convention on inequalities between Borel measures.}}

Let $\mu_1,\mu_2\colon 2^X\to[0,\infty]$ be Borel measures on a locally compact Hausdorff space $X$. We write
\[
\mu_1\leq\mu_2
\quad\text{as Borel measures on }X
\]
if
\[
\mu_1(B)\leq\mu_2(B)
\qquad\text{for every Borel set }B\subset X.
\]
Similarly, $\mu_1=\mu_2$ as Borel measures on $X$ means that
\[
\mu_1(B)=\mu_2(B)
\qquad\text{for every Borel set }B\subset X.
\]
Unless explicitly qualified by the phrase ``as Borel measures,'' an equality or inequality between such measures is understood as an equality or inequality of outer measures; that is, it is required to hold on every subset of $X$.\\

\paragraph{\noindent\textit{Relation between measures and Monge--Ampère equations}}

The main objects of study in this work are Monge--Ampère equations of the type
\begin{equation}
\label{eq:MA-measure-theory}
    \det D^2 u=f\quad\text{in }\Omega,
\end{equation}
where $f$ is only assumed to be nonnegative and to belong to $L^1(\Omega)$. Under this assumption, a convex solution $u\colon \Omega\to\mathbb R$ cannot in general be expected to belong to $C^2$. We must therefore reinterpret \eqref{eq:MA-measure-theory} in a generalized sense; see Section~\ref{subsec:def-alexandrov}. To this end, we associate with $u$ its Monge--Ampère measure $\mu_u$.

In the standard construction, one first introduces a
$\sigma$-algebra $\mathcal F_u$ on which $\mu_u$ is $\sigma$-additive; see the proof of
\cite[Theorem~2.3]{Figalli}. Although $\mathcal F_u$ contains the Borel $\sigma$-algebra of $\Omega$, it depends on the convex function $u$. Since both $f$ and $u$ are unknown in the inverse problems considered here, it is convenient to avoid working with an ambient $\sigma$-algebra that depends on the unknown solution.

We therefore construct $\mu_u$ as an outer measure, and Theorem~\ref{lem:MA-measure} will show that it is a Borel measure. Consequently, the Monge--Ampère measures corresponding to different convex functions can be compared on the Borel $\sigma$-algebra.

\subsubsection{Function spaces}

Finally, we introduce some classical function spaces that we will need later. The symbols 
\[
    C^k(U),C^k_c(U),C^k(\overline{U}),\quad k\in\N_0\cup \{\infty\},
\]
denote the space of $k$-times continuously differentiable functions on an open set $U$, the space of $C^k(U)$-functions that have compact support in $U$, and the space of $C^k(U)$-functions whose derivatives up to order $k$ can be continuously extended to $\overline{U}$ and are bounded on this set. For $0<\alpha\leq 1$, $k\in\N_0$, and an open set $U\subset \R^n$, the spaces $C^{k,\alpha}(\overline{U})$ consist of all $C^k(\overline{U})$-functions such that the $k$-th order derivatives $\nabla^k u=(\partial^{\beta} u)_{|\beta|=k}$ are $\alpha$-H\"older continuous, that is, $v=\partial^\beta u$, $|\beta|=k$, satisfies
\[
    [v]_{C^{0,\alpha}(\overline{U})}\vcentcolon =\sup_{x\neq y\in U}\frac{|v(x)-v (y)|}{|x-y|^{\alpha}}<\infty.
\]
These spaces are endowed with the norms
\[
    \|u\|_{C^k(\overline{U})}\vcentcolon = \sum_{\ell=0^k}\|\nabla^{\ell}u\|_{L^{\infty}(U)}
\]
and
\[
    \|u\|_{C^{k,\alpha}(\overline{U})}= \|u\|_{C^k(\overline{U})}+[\nabla^k u]_{C^{0,\alpha}(\overline{U})},
\]
respectively. Moreover, we shall write $u\in C^{k,\alpha}(U)$, if $u\in C^{k,\alpha}(\overline{V})$ for any $V\Subset U$. 

For $k\in\N_0$ and $0\leq \alpha\leq 1$, we say that an open bounded set $U\subset \R^n$ is a $C^{k,\alpha}$ domain if it can be locally straightened by $C^{k,\alpha}$ diffeomorphisms; see \cite[Section~6.2]{trudinger2015elliptic}. In the range $k\geq 1$, this is equivalent to the fact that $\partial U$ can be represented by a $C^{k,\alpha}$ function depending on $n-1$ variables.

As usual $L^p(\Omega)$, $1\leq p\leq \infty$, stands for the classical Lebesgue spaces with respect to $\mathcal{L}^n$ and Lebesgue spaces over $\partial\Omega$ are endowed with the $(n-1)$-dimensional Hausdorff measure $\mathcal{H}^{n-1}$. Moreover, we write $\mathcal{B}(\partial\Omega)$ for the space of Borel measurable functions $g\colon \partial\Omega \to [-\infty,\infty]$ and $\mathcal B_b(\partial\Omega)$ for the bounded Borel measurable functions. Finally, the Sobolev space of order $k\in\N$ and integrability $1\leq p\leq\infty$ is denoted by $W^{k,p}(\Omega)$. These spaces carry the norm
\[
    \|u\|_{W^{k,p}(\Omega)}\vcentcolon = \left(\sum_{\ell=0}^k \|\nabla^\ell u\|_{L^p(\Omega)}^p\right)^{1/p},
\]
with the usual modification for $p=\infty$.

\subsection{Convex domains}
\label{subsec:convex-domains}

In this section, we discuss convex domains, with a particular focus on domains with $C^1$ regular boundaries. Let us first recall the definition of convex sets and supporting hyperplanes.
\begin{enumerate}[(A)]
    \item A set $A\subset \R^n$ is said to be \emph{convex} if $x_0,x_1\in A$ implies that
\[
     x_t\vcentcolon =x_0+t(x_1-x_0)=(1-t)x_0+t x_1\in A,\,\text{for all }0\leq t\leq 1.
\]
    \item An \emph{(affine) hyperplane} in $\R^n$ is a set of the form
\[
    H=[\ell=\alpha]\vcentcolon =\{z\in\mathbb R^n:\ell(z)=\alpha\},
\]
where $\ell\colon\mathbb R^n\to\mathbb R$ is a nonzero linear functional and $\alpha\in\mathbb R$.
\item  Let $F\subset\mathbb R^n$ be a closed convex set. A \emph{supporting hyperplane} of $F$ is a hyperplane $H$ with the following two properties:
\begin{enumerate}[(a)]
    \item $F$ is contained in one of the two closed half-spaces determined by $H$
    \item and $F\cap H\neq\varnothing$.
\end{enumerate}
    \item Let $F\subset\mathbb R^n$ be a closed convex set. A point $x\in F$ is called \emph{exposed}, if there exists a supporting hyperplane $H$ such that $F\cap H=\{x\}$.
\end{enumerate}
We have the following simple lemma.

\begin{lemma}
\label{lemma:convex-dom}
    Let $\Omega\subset\R^n$ be a bounded open convex set.
    \begin{enumerate}[(i)]
        \item\label{hyperplances-conv} For every $x\in \partial\Omega$, there exists a supporting hyperplane of $\overline{\Omega}$ containing $x$.
        \item\label{convex-c1-dom} Assume, in addition, that $\partial\Omega$ is of class $C^1$ and, for every $x\in\partial\Omega$, define the affine tangent hyperplane
        \begin{equation}
        \label{eq:affine-tangent-plane}
            H_x \vcentcolon= x+T_x\partial\Omega = x+\nu(x)^\perp
            = \left\{ y\in\mathbb R^n: \nu(x)\cdot(y-x)=0 \right\}.
        \end{equation}
        Then $H_x$ is the unique supporting hyperplane of $\overline\Omega$ containing $x$. Moreover,
        \begin{equation}
        \label{eq:left-tangent-plane}
            \overline\Omega \subset
            \left\{ y\in\mathbb R^n: \nu(x)\cdot(y-x)\leq0 \right\}.
        \end{equation}
    \end{enumerate}
\end{lemma}

\begin{proof}
    The first assertion \ref{hyperplances-conv} is an immediate consequence of the Hahn--Banach theorem \cite[Theorem~1.6]{Brezis} with $A=\Omega$ and $B=\{x\}$. 

    Thus, it remains to establish \ref{convex-c1-dom}. Fix $x\in\partial\Omega$, and let $H$ be a supporting hyperplane of $\overline\Omega$ containing $x$, whose existence follows from assertion~\ref{hyperplances-conv}. There exist
$q\in\mathbb S^{n-1}$ and $\alpha\in\mathbb R$ such that
\[
H
=
[\ell_q=\alpha],\quad\text{with}\quad \ell_q(y)=q\cdot y.
\]
Replacing $(q,\alpha)$ by $(-q,-\alpha)$ if necessary, we may assume
that
\[
\overline\Omega
\subset
[\ell_q\leq \alpha].
\]
Since $x\in H$, we have $\alpha=q\cdot x$, and hence
\begin{equation}\label{eq:supporting-normal-m}
q\cdot(y-x)\leq0
\qquad
\text{for every }y\in\overline\Omega.
\end{equation}

Now, let $\tau\in T_x\partial\Omega$. Since $\partial\Omega$ is of class $C^1$, there exists a $C^1$ curve $\gamma:(-\varepsilon,\varepsilon)\to\partial\Omega$ such that
\[
\gamma(0)=x,
\qquad
\gamma'(0)=\tau.
\]
By \eqref{eq:supporting-normal-m},
\[
q\cdot(\gamma(t)-x)\leq0
\qquad
\text{for every }|t|<\varepsilon.
\]
Dividing by $t>0$ and letting $t\downarrow0$ gives $q\cdot\tau\leq0$, whereas dividing by $t<0$ and letting $t\uparrow0$ gives $q\cdot\tau\geq0$.
Hence $q\cdot\tau=0$. Since this holds for every $\tau\in T_x\partial\Omega$, the vector $q$ is orthogonal to $T_x\partial\Omega$. Because $q$ and $\nu(x)$ are unit vectors, $q=\pm\nu(x)$.

To determine the sign, recall that, since $\partial\Omega$ is of class $C^1$ and $\nu(x)$ is the outward unit normal, there exists $\varepsilon_0>0$ such that
\[
x-s\nu(x)\in\Omega
\qquad\text{for every }0<s<\varepsilon_0.
\] 
If $q=-\nu(x)$, then
\[
q\cdot\bigl((x-s\nu(x))-x\bigr)=s>0,
\]
contradicting \eqref{eq:supporting-normal-m}. Hence $q=\nu(x)$. Consequently,
\[
H
=
\{y\in\mathbb R^n:\nu(x)\cdot(y-x)=0\}
=
H_x,
\]
and \eqref{eq:supporting-normal-m} becomes
\[
\overline\Omega
\subset
\{y\in\mathbb R^n:\nu(x)\cdot(y-x)\leq0\},
\]
which proves \eqref{eq:left-tangent-plane}.

Since the preceding argument applies to every supporting hyperplane
of $\overline\Omega$ containing $x$, each such hyperplane equals
$H_x$. Thus $H_x$ is unique.
\end{proof}

\subsection{Strictly convex domains}
\label{sec:strictly-convex-domains}

In this section, we define strictly convex domains, which naturally appear in the well-posedness theory of the Monge--Ampère equation. Strict convexity excludes nontrivial contact sets
with supporting hyperplanes.

We will first give the definition for smooth domains. If $\Omega$ is a convex $C^1$ domain, then, motivated by \eqref{eq:left-tangent-plane}, \emph{strict convexity} can be expressed as 
\begin{equation}
\label{eq:strictly-convex}
\nu(x)\cdot(y-x)<0
\qquad
\text{for every }x\in\partial\Omega
\text{ and }y\in\overline\Omega\setminus\{x\}.
\end{equation}

Using the hyperplane formulation of Lemma~\ref{lemma:convex-dom}, we may also give a more geometric formulation of strict convexity that does not require any regularity of the boundary. 
\begin{definition}[{Strictly convex domains, \cite[Section~2.4.2]{Figalli}}]
\label{def:strictly-convex-domains}
    An open convex set $\Omega$ is called
\emph{strictly convex} if every supporting hyperplane of
$\overline\Omega$ intersects $\overline\Omega$ at exactly one point.
\end{definition}

For a convex domain with $C^1$ boundary, this geometric definition is
equivalent to \eqref{eq:strictly-convex}. 
\begin{corollary}[Supporting hyperplanes and strict convexity]
\label{cor:supporting-hyperplanes}
    Let $\Omega\subset\mathbb R^n$ be a bounded convex domain with $C^1$ boundary. Then Definition~\ref{def:strictly-convex-domains} and property \eqref{eq:strictly-convex} are equivalent.
\end{corollary}

\begin{proof}
Suppose first that $\Omega$ is strictly convex in the sense of
Definition~\ref{def:strictly-convex-domains}. Fix
$x\in\partial\Omega$ and let $H_x=x+T_x\partial\Omega$. By Lemma~\ref{lemma:convex-dom}\ref{convex-c1-dom}, $H_x$ is a
supporting hyperplane of $\overline\Omega$ containing $x$. Hence,
Definition~\ref{def:strictly-convex-domains} implies that
\[
H_x\cap\overline\Omega=\{x\}.
\]
On the other hand, \eqref{eq:left-tangent-plane} gives
\[
\nu(x)\cdot(y-x)\leq0
\qquad
\text{for every }y\in\overline\Omega.
\]
If equality held for some
$y\in\overline\Omega\setminus\{x\}$, then $y\in H_x$, contradicting
$H_x\cap\overline\Omega=\{x\}$. Therefore
\[
\nu(x)\cdot(y-x)<0
\qquad
\text{for every }y\in\overline\Omega\setminus\{x\}.
\]
Since $x\in\partial\Omega$ was arbitrary, \eqref{eq:strictly-convex}
holds.

Conversely, suppose that $\Omega$ satisfies
\eqref{eq:strictly-convex}, and let $H$ be any supporting
hyperplane of $\overline\Omega$. By definition, $H\cap\overline\Omega\neq\varnothing$. Choose $x\in H\cap\overline\Omega$. Necessarily $x\in\partial\Omega$, since every hyperplane through an interior point of $\Omega$ has points of $\Omega$ on both sides. By Lemma~\ref{lemma:convex-dom}\ref{convex-c1-dom}, $H_x$ is the unique
supporting hyperplane of $\overline\Omega$ containing $x$. Hence $H=H_x$. Now \eqref{eq:strictly-convex} gives
\[
\nu(x)\cdot(y-x)<0
\qquad
\text{for every }y\in\overline\Omega\setminus\{x\}.
\]
Consequently, no point of
$\overline\Omega\setminus\{x\}$ belongs to $H_x$, and therefore
\[
H\cap\overline\Omega
=
H_x\cap\overline\Omega
=
\{x\}.
\]
Thus every supporting hyperplane of $\overline\Omega$ intersects
$\overline\Omega$ at exactly one point, so $\Omega$ is strictly
convex in the sense of
Definition~\ref{def:strictly-convex-domains}.
\end{proof}

\begin{remark}
Corollary~\ref{cor:supporting-hyperplanes} shows that, for a strictly
convex domain with $C^1$ boundary, every boundary point is exposed.
In particular, the boundary of such a domain contains no nontrivial
line segment.
\end{remark}

\subsection{Convex functions and subdifferentials}
\label{subsec:convex-functions}

Let $\Omega\subset\mathbb R^n$ be a convex set. A function
$u\colon\Omega\to\mathbb R$ is called \emph{convex} if
\[
u((1-t)x+ty)
\leq
(1-t)u(x)+tu(y)
\]
for every $x,y\in\Omega$ and $0\leq t\leq1$.

Suppose that $\Omega$ is open and $u$ is differentiable at $x\in\Omega$. Then convexity
implies the supporting-plane inequality
\begin{equation}
\label{eq:gradient-supporting-inequality}
u(y)
\geq
u(x)+\nabla u(x)\cdot(y-x)
\qquad
\text{for every }y\in\Omega.
\end{equation}
Indeed, for $0<t\leq1$, convexity gives
\[
u(x+t(y-x))
\leq
(1-t)u(x)+tu(y).
\]
Consequently,
\[
u(y)-u(x)
\geq
\frac{u(x+t(y-x))-u(x)}{t}.
\]
Letting $t\downarrow0$ yields
\eqref{eq:gradient-supporting-inequality}. Thus the affine function
\[
\ell_x(y)
\vcentcolon=
u(x)+\nabla u(x)\cdot(y-x)
\]
touches $u$ from below at $x$. 

For a convex function that is not differentiable, there
may be more than one affine function touching its graph from below. Their slopes are described by the subdifferential $\partial u(x)$. For example, the function $\R^n\ni x\mapsto u(x)=|x|$ is not
differentiable at the origin, and every affine function
\[
y\mapsto \ell_p(y)=p\cdot y,
\qquad |p|\leq1,
\]
supports $u$ from below at the origin. Accordingly, $\partial u(0)=\overline B_1$.

\begin{definition}[Subdifferential, \cite{rockafellar1966characterization}]
\label{def:subdiff}
Let $u\colon\Omega\to\mathbb R$ be convex. For $x\in\Omega$, the
\emph{subdifferential of $u$ at $x$} is
\begin{equation}
\label{eq:subdifferential}
\partial u(x)
\vcentcolon=
\left\{
p\in\mathbb R^n:
u(y)\geq u(x)+p\cdot(y-x)
\text{ for every }y\in\Omega
\right\}.
\end{equation}
The elements of $\partial u(x)$ are called \emph{subgradients} of
$u$ at $x$.
\end{definition}

\begin{remark}[{Subdifferential at differentiability points}]
\label{rem:rel-subgrad-grad}
If $u$ is differentiable at $x\in\Omega$, then
\begin{equation}
\label{eq:classical-der-at-diff-points}
\partial u(x)=\{\nabla u(x)\}.
\end{equation}
Conversely, if $\partial u(x)$ is a singleton for some point $x\in \Omega$, then $u$ is differentiable at this point and $\partial u(x)=\{\nabla u(x)\}$. This follows from \cite[Theorem~25.1]{Rockafellar}. Thus the subdifferential is a
natural set-valued extension of the classical gradient.
\end{remark}

The next result tells us that convex functions are almost everywhere differentiable.

\begin{lemma}[{Differentiability of convex functions,\cite[Corollary~A.23]{Figalli}}]
\label{lem:diff-convex}
    Let $\Omega\subset\R^n$ be an open convex set and suppose $u\colon \Omega\to\R$ is convex. Then $u$ is locally Lipschitz continuous and
    \begin{equation}
    \label{eq:Lipschitz}
        \|\nabla u\|_{L^{\infty}(K)}\leq \frac{2\|u\|_{L^{\infty}(\Omega')}}{\dist(K,\partial\Omega')}
    \end{equation}
    for every $K\Subset\Omega'\Subset \Omega$.
\end{lemma}

An important property of convex functions $v\colon \Omega\to\R$ is that their subgradients (up to a set of measure zero) are unique.

\begin{lemma}[{Unique subgradients, \cite[Lemma~A.30]{Figalli}}]
\label{lem:unique-subgradients}
    Let $\Omega\subset\R^n$ be an open convex set. Suppose that $v\colon \Omega \to \R$ is convex. Then the set of non-unique subgradients
    \begin{equation}
    \label{eq:equal-subgradients}
        Z_v\vcentcolon = \{p\in\R^n:p\in \partial v(x)\cap \partial v(y)\text{ for some }x\neq y\in \Omega\}
    \end{equation}
    has Lebesgue measure zero.
\end{lemma}

\begin{remark}
\label{proof-unique-subgradients}
    When $\Omega\subset\R^n$ is bounded, then this result follows from Rademacher's theorem \cite[Section~3.1]{evans-mt} applied to the \emph{Legendre transform} of $v$, defined as the map $v^*\colon \R^n\to\R\cup\{\infty\}$ with
    \begin{equation}
    \label{eq:Legendre}
     v^*(p)\vcentcolon = \sup_{x\in\Omega}(x\cdot p-v(x)),\quad p\in\R^n.
    \end{equation}
    It can indeed be invoked, since the boundedness of $\Omega$ implies the Lipschitz continuity of $v^*$. If $\Omega\subset\R^n$ is a general open convex set, then one writes $\Omega=\bigcup_{k\in\N} \Omega_k$ with $\Omega_k=\Omega\cap B_k$. Since the bounded case applies to 
    \[
         Z^k_v\vcentcolon = \{p\in\R^n:p\in \partial v(x)\cap \partial v(y)\text{ for some }x\neq y\in \Omega_k\},
    \]
    the inclusion $Z_v\subset \bigcup_{k\in\N} Z_v^k$ completes the proof.
\end{remark}

Some other properties of the subgradient, which we shall need frequently are collected in the next lemma.

\begin{lemma}[Compactness of subgradients]
\label{lemma:compactness-subgradinets}
    Let $\Omega\subset\R^n$ be an open convex domain and suppose $u\colon \Omega\to\R$ is a convex function. For any compact set $K\subset\Omega$, the image $\partial u(K)$ is compact. If, in addition, the compact set $K$ is nonempty, then $\partial u(K)\neq \varnothing$. 
\end{lemma}

\begin{proof}
    The first part follows from \cite[Lemma~A.22]{Figalli}. The second part follows from \cite[Theorem~24.7]{Rockafellar}, the fact that a proper convex function is closed if and only if it is lower semicontinuous, and Lemma~\ref{lem:diff-convex}.
\end{proof}

We shall repeatedly use the following elementary convergence fact for subgradients.

\begin{lemma}[Convergence of subgradients]
\label{lem:gradient-convergence}
Let $\Omega\subset\R^n$ be an open convex domain. Suppose $u_k\colon \Omega\to\R$ and $u\colon \Omega\to\R$ are convex.
\begin{enumerate}[(i)]
    \item\label{conv-subgrad-1} If $u_k\to u$ locally uniformly, and the sequences $x_k\in\Omega$, $p_k\in\partial u_k(x_k)$, $k\in\N$, satisfy
    \[
        x_k\to x\in\Omega\quad\text{and}\quad p_k\to p, \quad k\to\infty,
    \]
    then $p\in\partial u(x)$.
    \item\label{conv-subgrad-2}  If $u_k\to u$ locally uniformly, then
\begin{equation}
\label{eq:convergence-of-gradients}
  \nabla u_k(x)\to\nabla u(x),\qquad \text{as }k\to\infty
\end{equation}
    for a.e.~$x\in\Omega$.
\end{enumerate}
\end{lemma}

\begin{proof}
   \ref{conv-subgrad-1}: For each $k\in\N$, we have
\[
 u_k(z)\geq u_k(x_k)+p_k\cdot(z-x_k),\quad \text{for all }z\in\Omega.
\]
Since the convergence $u_k\to u$ is uniform on compact sets and $u$ is continuous, we have $u_{\ell}(x_{\ell})\to u(x)$, as $\ell\to\infty$. Thus passing to the limit $k\to\infty$ yields
\[
 u(z)\geq u(x)+p\cdot(z-x),\quad  \text{for all }z\in\Omega.
\]
Hence $p\in\partial u(x)$. 

\ref{conv-subgrad-2}: We first show that \eqref{eq:convergence-of-gradients} holds for every point in which $u_k$, $k\in\N$, and $u$ are differentiable. Suppose that $x\in \Omega$ is such a differentiablity point. Set $p_k\vcentcolon = \nabla u_k(x)$, for $k\in\N$.
We first show that $(p_k)_{k\in\N}$ is bounded. Choose $r>0$ such that
$\overline B_r(x)\subset\Omega$. Local uniform convergence implies
that there exists $M>0$ such that
\[
|u_k(y)|\leq M
\qquad
\text{for every }y\in\overline B_r(x)
\]
and every sufficiently large $k$. If $p_k\neq0$, apply the
subgradient inequality with
\[
y_k\vcentcolon = x+r\frac{p_k}{|p_k|}\in\overline B_r(x).
\]
Then
\[
M
\geq
u_k(y_k)
\geq
u_k(x)+p_k\cdot(y_k-x)
=
u_k(x)+r|p_k|
\geq
-M+r|p_k|.
\]
Consequently, $(p_k)_{k\in\N}$ is bounded and every subsequence of $(p_k)_{k\in\N}$ has a further convergent subsequence. If $p_{k_j}\to q$, the first part of the proof, applied with $x_{k_j}=x$, shows that $q\in\partial u(x)$. Since $u$ is differentiable at $x$, Remark~\ref{rem:rel-subgrad-grad} gives $q=\nabla u(x)$. Thus every convergent subsequence of $(p_k)_{k\in\N}$ converges to the same limit $\nabla u(x)$. It follows that the entire sequence converges to $\nabla u(x)$, which establishes \eqref{eq:convergence-of-gradients} for all common differentiability points of $u_k$, $k\in\N$, and $u$.

Finally, by Lemma~\ref{lem:diff-convex}, the functions $u_k$, $k\in\N$, and $u$ are locally Lipschitz in $\Omega$. Thus, Rademacher's theorem \cite[Section~3.1]{evans-mt} implies that the function $u_k,u$ are differentiable almost everywhere. Hence, the convergence \eqref{eq:convergence-of-gradients} holds for a.e.~$x\in\Omega$.
\end{proof}

\subsection{Uniformly convex \texorpdfstring{$C^{1,1}$}{C11} domains and their representation}
\label{subsec:uniform-convex-domains}

In this section, we define the notion of uniform convexity, which imposes a quantitative quadratic separation of the boundary from every tangent hyperplane. The main result of this section establishes a representation theorem for uniformly convex $C^{1,1}$ domains.

\begin{definition}
\label{def:uniformly-convex-domains}
    Let $\Omega\subset\R^n$ be a bounded $C^1$ domain. Then $\Omega$ is called \emph{uniformly convex}, if there exists a constant $\kappa_\Omega>0$ such that
\begin{equation}
\label{eq:geometric-uniform-convexity}
\nu(x)\cdot(y-x)
\leq
-\kappa_\Omega|y-x|^2
\qquad
\text{for all }x\in\partial\Omega
\text{ and }y\in\overline\Omega,
\end{equation}
where $\nu\in C^0(\partial\Omega)$ is the outward pointing unit normal field to $\partial\Omega$. 
\end{definition}

Lemma~\ref{lem:defining-function} shows, in particular, that a
uniformly convex $C^{1,1}$ domain is convex. Moreover,
\eqref{eq:geometric-uniform-convexity} then immediately implies
\eqref{eq:strictly-convex}, so such a domain is strictly convex. Moreover, this result establishes an equivalent reformulation of uniformly convex domains through defining-functions, which will be needed in the barrier argument (Proposition~\ref{prop:boundary-Lipschitz}). Since the proof is rather technical and lengthy we present it in Appendix~\ref{sec:representation-thm}.

\begin{lemma}[Defining function formulation]
\label{lem:defining-function}
Let $\Omega\subset\mathbb R^n$ be a bounded domain with $C^{1,1}$
boundary. Then the following conditions are equivalent:
\begin{enumerate}[(i)]
\item\label{unif-convex-1} $\Omega$ is uniformly convex in the sense of Definition~\ref{def:uniformly-convex-domains}.
\item\label{unif-convex-2} There exist an open convex neighborhood $U$ of $\overline\Omega$, a function $\rho\in C^{1,1}(U)$, and a constant $c_0>0$ such that $\rho$ is \emph{uniformly convex} with constant $c_0$, meaning that
\begin{equation}
\label{eq:uniform-convex-fct}
    D^2\rho\geq c_0I\quad\text{a.e. in }U,
\end{equation}
and $\overline{\Omega}$ can be represented as
\begin{equation}
\label{eq:strong-defining-function}
    \Omega=\{x\in U:\rho(x)<0\}\quad\text{and}\quad
\partial\Omega=\{x\in U: \rho(x)=0\}.
\end{equation}
\end{enumerate}
In fact, under \ref{unif-convex-1}, one may take
$U=\mathbb R^n$. Moreover, every defining function $\rho$ satisfying \ref{unif-convex-2} obeys
\begin{equation}\label{eq:rep-normal}
\nu(x)
=
\frac{\nabla\rho(x)}{|\nabla\rho(x)|},
\qquad x\in\partial\Omega.
\end{equation}
\end{lemma}

\begin{remark}\label{rem:Figalli-defining-function}
 We point out that in the smooth
setting, \cite[Section~3.1.2, especially footnote~1 on
pp.~40--41]{Figalli} constructs a uniformly convex function
vanishing on $\partial\Omega$ by first using
\[
-\dist(\,\cdot\,,\partial\Omega)
+\dist(\,\cdot\,,\partial\Omega)^2
\]
near the boundary and then extending it as a uniformly convex
function through the interior. Lemma~\ref{lem:defining-function} gives the corresponding $C^{1,1}$ defining-function characterization used here.
\end{remark}

\section{Alexandrov solutions}
\label{sec:Alexandrov-sol}

In this section, we review the theory of generalized solutions of the Monge--Ampère equation, nowadays called \emph{Alexandrov solutions}. A.~D.~Alexandrov introduced the higher-dimensional theory of generalized solutions in his study of the Dirichlet problem \cite{AL4}, building on his earlier work on convex surfaces \cite{AL0,AL1,AL2,AL3}. Standard references for the modern theory are the monographs \cite{Gutierrez,Figalli,le2024analysis,aubin2012nonlinear}. 

The motivation for and precise definition of Alexandrov solutions
are given in Section~\ref{subsec:def-alexandrov}. On the way we also introduce the Monge--Ampère measures $\mu_u$, whose properties are studied in Section~\ref{subsec:prop-MA-meas}. Then in Section~\ref{sec:Dirichlet-theory} we revisit the Dirichlet theory of Alexandrov solutions.

\subsection{Definition of Alexandrov solutions}
\label{subsec:def-alexandrov}

An Alexandrov solution of the Monge--Ampère equation
\begin{equation}
\label{eq:mot-MA-eq}
    \det D^2 u=f\quad \text{in }\Omega
\end{equation}
is a convex function $u\colon \Omega\to \R$, which solves \eqref{eq:mot-MA-eq} only in a generalized or weak sense. Here and throughout this section, $\Omega\subset\R^n$ denotes an open convex set. Similar to elliptic PDEs in divergence form, such as the conductivity equation $\Div(A\nabla v)=0$, an appropriate notion of generalized or weak solutions should have the following properties:
\begin{enumerate}[(a)]
    \item\label{lower-reg} \textit{regularity}: Weak formulation makes sense for convex functions $u$, which are not twice differentiable.
    \item\label{strong-are-weak} \textit{strong $\Rightarrow$ weak}: Every strong solution of \eqref{eq:mot-MA-eq} is a weak solution.
    \item\label{weak-eq-strong} \textit{weak = strong}: If $u\in C^2(\Omega)$ is a weak solution of \eqref{eq:mot-MA-eq} with $f\in C(\Omega)$, then $u$ solves the PDE pointwise.
\end{enumerate}
For the divergence form equation $\Div(A\nabla v)=0$ the weak formulation is found by integration by parts. For the Monge--Ampère equation \eqref{eq:mot-MA-eq}, the appropriate replacement for integration by parts is provided by the \emph{area formula} \cite[Section~3.3]{evans-mt}, which can be formulated as follows. For every Lipschitz function $F\colon \R^n\to\R^m$ ($m\geq n$) and every Lebesgue measurable set $A\subset \R^n$ one has 
\begin{equation}
\label{eq:area-formula}
    \int_{A}JF\,dx=\int_{\R^m}\mathcal{H}^0(A\cap F^{-1}(y))\,d\mathcal{H}^n (y).
\end{equation}
We recall that $\mathcal{H}^0$ is the counting measure and $JF$ is the Jacobian of $DF\colon \R^n\to\R^m$, which can be calculated as
\[
    JF(x)\vcentcolon = \sqrt{\det [(DF(x))^T(DF(x))]}\quad\text{for a.e. }x\in\R^n.
\]
This implies the next result.
\begin{lemma}
\label{lemma:Area-formula-grad}
    Let $\Omega\subset\R^n$ be an open convex set and assume that $u\in C^{1,1}(\overline{\Omega})$ is convex. Then, for every Lebesgue measurable set $A\subset\Omega$,
    \begin{equation}
\label{eq:area-formula-gradient}
    \int_A \det D^2 u\,dx= \mathcal{L}^n(\nabla u(A)).
\end{equation}
\end{lemma}
\begin{proof}
    First, by  Kirszbraun's theorem \cite[Section~2.10.34]{federer2014geometric}, we may choose a Lipschitz function $V\colon \R^n\to\R^n$ such that $V=\nabla u$ on $\overline{\Omega}$. Thus, applying the area formula \eqref{eq:area-formula} with $F=V$, we get
    \[   
    \int_{A}JV\,dx=\int_{\R^n}\mathcal{H}^0(A\cap V^{-1}(y))\,d\mathcal{H}^n(y).
    \]
    Since $u\in C^{1,1}(\overline{\Omega})$ is convex, we deduce that $D^2u$ is symmetric and nonnegative for a.e.~$x\in\Omega$. Therefore, 
    \[
        JV=|\det DV|=|\det D^2 u|=\det D^2 u\quad\text{in }\Omega.
    \]
    Thus, $A\subset \Omega$, $V=\nabla u$ in $\Omega$, and $\mathcal{H}^n=\mathcal{L}^n$ on $\R^n$ \cite[Section~2.2]{evans1991measure}, imply that
     \[   \int_{A}\det D^2 u\,dx=\int_{\R^n}\mathcal{H}^0(A\cap (\nabla u)^{-1}(y))\,dy.
    \]
    The right hand side can be simplified as follows
    \[
        \begin{split}
            \int_{\R^n}\mathcal{H}^0(A\cap (\nabla u)^{-1}(y))\,dy&=\int_{\R^n}\mathcal{H}^0(\{x\in A: \nabla u(x)=y\})\,dy\\
            &=\int_{\nabla u(A)}\mathcal{H}^0(\{x\in A: \nabla u(x)=y\})\,dy\\
            &=\mathcal{L}^n(\nabla u(A)),
        \end{split}
    \]
    where in the last step we applied Lemma~\ref{lem:unique-subgradients}. This completes the proof.
\end{proof}

Recall that for every convex function $v\vcentcolon \Omega\to\R$ one has $\partial v(x)=\{\nabla v(x)\}$, whenever $v$ is differentiable at $x\in\Omega$. Thus formula \eqref{eq:area-formula-gradient} can be rewritten as
\begin{equation}
\label{eq:mot-MA-meas}
    \int_A \det D^2 u\,dx= \mathcal{L}^n(\partial u(A)),
\end{equation}
when $u\in C^{1,1}(\overline{\Omega})$. The right-hand side remains meaningful for an arbitrary convex function $u\vcentcolon \Omega\to\mathbb R$ and every set $A\subset\Omega$. This suggests giving it its own name and defining generalized solutions through it.

\begin{definition}
\label{def:MA-measure}
    Let $u\colon \Omega\to\R$ be a convex function, defined on a convex open set $\Omega\subset\R^n$. Then the \emph{Monge--Ampère measure} induced by $u$ is defined as
\begin{equation}
\label{eq:MA-measure}
  \mu_u\colon 2^\Omega \to [0,\infty],\quad \mu_u(A)=\cL^n(\partial u(A)).
\end{equation}
\end{definition}

Recall that, according to our convention, measures are outer measures defined on all subsets of the underlying space. This differs slightly from the convention in
\cite[proof of Theorem~2.3, Step~1]{Figalli}, where $\mu_u$ is considered on the $u$-dependent $\sigma$-algebra
\[
    \mathcal{F}_u\vcentcolon = \{A\subset \Omega: \partial u(A)\text{ is Lebesgue measurable}\},
\]
which contains the Borel sets in $\Omega$. The next result, Lemma~\ref{lem:MA-measure}, shows that every Borel set is $\mu_u$-measurable and thereby justifies the name Monge--Ampère measure. It also establishes its connection with the classical Monge--Ampère operator.

\begin{lemma}[{Monge--Ampère measure}]
\label{lem:MA-measure}
    Let $\Omega\subset\R^n$ be an open convex set. Assume that $u\colon\Omega\to\R$ is convex and let $\mu_u$ be the Monge--Amp\`ere measure defined by \eqref{eq:MA-measure}.
    \begin{enumerate}[(i)]
        \item\label{general-case} The set function $\mu_u$ is a locally finite Borel measure on $\Omega$.
        \item\label{C1-case} If $u\in C^1(\Omega)$, then $\mu_u$ is a Radon measure.
        \item\label{C11-case} If $u\in C^{1,1}(\overline{\Omega})$, then $\mu_u=(\det D^2u)\,dx$ with $\det D^2u\in L^\infty(\Omega)$.
    \end{enumerate}
\end{lemma}

\begin{proof}
  \ref{general-case}: We first show that $\mu_u$ is an (outer) measure. Clearly, $\mu_u(\varnothing)=0$. The $\sigma$-subadditivity of $\mu_u$ follows from the identity
\[
    \partial u \left(\bigcup_{j=1}^{\infty} A_j\right)=\bigcup_{j=1}^{\infty}\partial u(A_j)
\]
and the $\sigma$-subadditivity of $\mathcal{L}^n$. Thus $\mu_u$ is a measure.

We next prove that every Borel subset of $\Omega$ is
$\mu_u$-measurable. Introduce the exceptional set
\[
Z_u
\vcentcolon=
\left\{
p\in\mathbb R^n:
p\in\partial u(x)\cap\partial u(y)
\text{ for some }x,y\in\Omega,\ x\neq y
\right\}.
\]
By Lemma~\ref{lem:unique-subgradients}, $\cL^n(Z_u)=0$. Moreover, the proof of \cite[Theorem~2.3, Step~1]{Figalli} shows that $\partial u(A)$ is Lebesgue measurable whenever $A\subset\Omega$ is Borel.

Let $S\subset\Omega $ be arbitrary and let $A\subset\Omega $ be Borel. We need to verify that
\[
    \mu_u(S)=\mu_u(A\cap S)+\mu_u(S\setminus A).
\]
To this end, set
\[
C\vcentcolon=\partial u(S\cap A),
\qquad
D\vcentcolon=\partial u(S\setminus A),
\qquad
M\vcentcolon=\partial u(A).
\]
Then
\[
\partial u(S)=C\cup D,
\quad 
C\subset M, \quad\text{and}\quad D\cap M\subset Z_u.
\]
By \cite[Theorem~2.3, Step~1]{Figalli}, $M$ is Lebesgue measurable. Setting $D_0=D\setminus M$, we have
\[
\cL^n(D_0)\leq \cL^n(D)=\cL^n(D\setminus M)+\cL^n(D\cap M)\leq \cL^n(D_0)+\cL^n(Z_u)=\cL^n(D_0).
\]
Similarly,
\[
\cL^n(C\cup D_0)\leq \cL^n(C\cup D)\leq \cL^n(C\cup D_0)+\cL^n(D\cap M)=\cL^n(C\cup D_0).
\]
Since $C\subset M$ and $D_0\subset \R^n\setminus M$, the $\mathcal{L}^n$-measurablity of $M$ gives
\[
    \cL^n(C\cup D_0)=\cL^n(C)+\cL^n(D_0), 
\]
because $(C\cup D_0)\cap M=C$ and $(C\cup D_0)\setminus M=D_0$. Therefore,
\[
\begin{aligned}
\mu_u(S)
&=\cL^n(C\cup D)\\
&=\cL^n(C\cup D_0)\\
&=\cL^n(C)+\cL^n(D_0)\\
&=\cL^n(C)+\cL^n(D)\\
&=\mu_u(S\cap A)+\mu_u(S\setminus A).
\end{aligned}
\]
Thus $A$ is $\mu_u$-measurable. Hence $\mu_u$ is a Borel (outer) measure.

Local finiteness is then immediate. If $K\Subset \Omega$ is compact, \cite[Lemma~A.22]{Figalli} implies that $\partial u(K)$ is compact, and therefore locally finiteness of $\mathcal{L}^n$ gives
\[
\cL^n\bigl(\partial u(K)\bigr)
<\infty.
\]
Hence $\mu_u$ is locally finite.

\ref{C1-case}: From now on assume that $u\in C^1(\Omega)$. It remains to show that $\mu_u$ is Borel regular. First note that since $u\in C^1(\Omega)$, Remark~\ref{rem:rel-subgrad-grad} gives
\[
\partial u(x)=\{\nabla u(x)\},
\qquad x\in\Omega.
\]
Consequently,
\[
\mu_u(C)=\mathcal L^n(\nabla u(C))
\qquad\text{for every }C\subset\Omega.
\] 
Now, let $A\subset \Omega$ be an arbitrary subset. Since $\mathcal{L}^n$ is Borel regular, there exists a Borel set $B\subset \R^n$ such that $\nabla u(A)\subset B$ and $\mathcal{L}^n(B)=\mu_u(A)$. The continuity of $\nabla u$ implies that $S\vcentcolon = (\nabla u)^{-1}(B)$ is a Borel set. Since
\[
    A\subset (\nabla u)^{-1}((\nabla u)(A))\subset (\nabla u)^{-1}(B)=S
\]
and
\[
    (\nabla u)(S)=(\nabla u)(\nabla u)^{-1}(B)\subset B,
\]
we get
\[
    \mu_u(A)\leq \mu_u(S)\leq \mathcal{L}^n(B)=\mu_u(A).
\]
Thus $\mu_u(S)=\mu_u(A)$, which proves that $\mu_u$ is Borel regular. Since $\mu_u$ is also locally finite, it is a Radon measure.

\ref{C11-case}: Finally, assume that the convex function $u$ belongs to $C^{1,1}(\overline{\Omega})$. Since $\Omega$ is a convex open set, $u\in C^{1,1}(\overline{\Omega})$ implies that $u\in W^{2,\infty}(\Omega)$ (see \cite[Proposition~9.3]{Brezis}). Hence, $\det D^2 u\in L^{\infty}(\Omega)$. Moreover, convexity of $u$ implies that $D^2u$ is positive semidefinite almost everywhere and hence
$\det D^2u\geq0$ almost everywhere.

Since $u$ is differentiable everywhere, Remark~\ref{rem:rel-subgrad-grad} gives $\partial u(x)=\{\nabla u(x)\}$. Therefore Lemma~\ref{lemma:Area-formula-grad} yields
\begin{equation}
\label{eq:equality-Lebesgue-sets}
\mu_u(A)
=
\mathcal L^n(\partial u(A))
=
\mathcal L^n(\nabla u(A))
=
\int_A\det D^2u(x)\,dx,
\end{equation}
for every Lebesgue measurable set $A\subset\Omega$.  In particular,
\[
\mu_u=(\det D^2u)\,dx
\]
as Borel measures on $\Omega$, meaning that 
\[
    \mu_u(B)
=
\int_B \det D^2u(x)\,dx\quad\text{for all Borel sets }B\subset \Omega.
\]
Hence $\mu_u\ll\mathcal L^n\res\Omega$, and its density $\det D^2 u$ belongs to $L^\infty(\Omega)$.

Let $\nu$ be the Borel regular extension of $\det D^2u\,dx$, that is, 
\[
    \nu(A)\vcentcolon = \inf \left\{\int_B \det D^2 u\,dx:A\subset B,\,B\subset\Omega\text{  Borel}\right\},\quad A\subset \Omega;
\]
cf., proof of Proposition~\ref{prop:weak-convergence-borel-measures}. Since $\det D^2u\in L^\infty(\Omega)$, $\nu$ is a Radon measure. By \ref{C1-case}, $\mu_u$ is Radon as well. Moreover, the measures $\mu_u$ and $\nu$ agree on Borel sets by
\eqref{eq:equality-Lebesgue-sets}. Hence, for every
$A\subset\Omega$, outer regularity gives
\[
\begin{aligned}
\mu_u(A)
&=
\inf\{\mu_u(V):A\subset V,\ V\subset\Omega\text{ open}\}\\
&=
\inf\left\{\int_V \det D^2 u\,dx:A\subset V,\ V\subset\Omega\text{ open}\right\}\\
&=
\inf\{\nu(V):A\subset V,\ V\subset\Omega\text{ open}\}\\
&=
\nu(A).
\end{aligned}
\]
Consequently, $\mu_u=(\det D^2u)\,dx$ as (outer) measures on $\Omega$.
\end{proof}

\begin{remark}[Singular Monge--Ampère measures]
     We note that Lebesgue's decomposition theorem \cite[Section~1.6]{evans-mt} demonstrates the existence of a nonnegative function $f\in L^1_{\mathrm{loc}}(\Omega)$ and a Radon measure $\mu_u^{\mathrm{s}}$, singular with respect to $\mathcal L^n$, such that
\[
\mu_u(B)
=
\int_B f\,dx+\mu_u^{\mathrm{s}}(B)
\qquad
\text{for every Borel set }B\subset\Omega.
\]
Indeed, the application of this decomposition theorem can be justified rigorously in at least two ways. One possibility is to notice that 
\[
    C_c(\Omega)\ni \varphi\mapsto L_u \varphi\vcentcolon = \int_\Omega \varphi\,d\mu_u
\]
defines a continuous functional with respect to its natural locally convex topology. Indeed,
for every compact set $K\subset\Omega$,
\[
|L_u(\varphi)|
\leq
\mu_u(K)\|\varphi\|_{L^\infty(\Omega)}
\qquad
\text{whenever }\operatorname{supp}\varphi\subset K.
\]
Hence, Riesz's representation theorem \cite[Section~1.8]{evans-mt} yields the existence of a unique Radon measure $\widetilde{\mu}_u$ such that
\[
    \int_\Omega \varphi\,d\mu_u=\int_\Omega \varphi\,d\widetilde{\mu}_u,\quad \varphi\in C_c(\Omega).
\]
Alternatively, the construction in the proof of
Proposition~\ref{prop:weak-convergence-borel-measures} shows directly that $\widetilde\mu_u$ is the Borel regularization of $\mu_u$, namely
\[
    \widetilde \mu_u (A)=\widehat \mu_u(A)\vcentcolon =\inf\bigl\{
\mu_u (B):
A\subset B,\ B\subset\Omega\text{ Borel}
\bigr\},\quad A\subset \Omega.
\]
\end{remark}

The next examples show that the singular part can concentrate on sets of different dimensions (see also \cite[Example~2.2]{Figalli}).

\begin{example}[Singular Monge--Ampère measures]
\label{Example:singular-MA-measures}
The following examples illustrate two prototypical types of singular Monge--Ampère measures. In the first, the measure has an atom concentrated at a single point, whereas in the second, its singular part is supported on a hyperplane.
\begin{enumerate}[(i)]
    \item\label{first-sing-MA}  If $\Omega\subset \R^n$ is a convex open set containing $x_0\in\R^n$ and $M>0$, then 
\[
u\colon \Omega\to \R,\quad u(x)=M|x-x_0|
\]
describes geometrically a cone and we have
\[
    \partial u(x)=\begin{cases}
       \big\{ M\frac{x-x_0}{|x-x_0|}\big\},\quad & x\neq x_0,\\
        \overline{B_M},\quad & x=x_0.
    \end{cases}
\]
Indeed, $p\in\partial u(x_0)$ if and only if
\[
p\cdot(z-x_0)\leq M|z-x_0|
\qquad\text{for every }z\in\Omega.
\]
Since $x_0$ is an interior point of $\Omega$, this is equivalent to $|p|\leq M$. Consequently,
\[
\mu_u
=
M^n\mathcal L^n(B_1)\,\delta_{x_0}
\] 
is a purely atomic measure, which is concentrated at the vertex $x_0$.\\

More generally, the existence theory in Section~\ref{sec:Dirichlet-theory} shows that, on a bounded convex domain, finite sums of positive Dirac measures can occur as Monge--Ampère measures.
    \item Let $e\in\mathbb S^{n-1}$, $c\in\mathbb R$, and consider the hyperplane
\[
H\vcentcolon = \{x\in\mathbb R^n:x\cdot e=c\}.
\]
Suppose that $\Omega\subset\R^n$ is a convex open set with $\Omega\cap H\neq \varnothing$. The convex function
\[
u\colon \Omega\to \R,\quad u(x)
=
\frac12|x|^2+\frac12|x\cdot e-c|
\]
satisfies
\begin{equation}
\label{eq:MA-meas-hyper}
\mu_u
=
\mathcal L^n\res\Omega
+
\mathcal H^{n-1}\res(H\cap\Omega).
\end{equation}
To verify this, write
\[
x_\parallel\vcentcolon = x\cdot e,
\qquad
x_\perp\vcentcolon = x-(x\cdot e)e.
\]
The subdifferential sum rule gives
\begin{equation}
\label{eq:subgrad-hyper}
\partial u(x)
=
\begin{cases}
\{x+\frac12e\},
&x_\parallel>c,\\[1mm]
\{x-\frac12e\},
&x_\parallel<c,\\[1mm]
x+[-\frac12,\frac12]e,
&x_\parallel=c.
\end{cases}
\end{equation}
Indeed,
\[
\partial\frac{|x|^2}{2}=\{x\},\quad x\in\R^n,
\]
and
\[
\partial\frac{|x\cdot e-c|}{2}
=
\begin{cases}
\{\frac12e\},&x_\parallel>c,\\
\{-\frac12e\},&x_\parallel<c,\\
[-\frac12,\frac12]e,&x_\parallel=c.
\end{cases}
\]

Let $A\subset\Omega$ be arbitrary and set
\[
A_+\vcentcolon = A\cap\{x_\parallel>c\},
\qquad
A_-\vcentcolon = A\cap\{x_\parallel<c\},
\qquad
A_0\vcentcolon = A\cap H.
\]
Since $H$ and the two half-spaces determined by $H$ are Borel sets,
they are $\mu_u$-measurable. Consequently,
\[
\mu_u(A)
=
\mu_u(A_+)+\mu_u(A_-)+\mu_u(A_0).
\]
Using formula \eqref{eq:subgrad-hyper}, we find
\[
\partial u(A_{\pm})=A_{\pm}\pm\frac12e\quad\text{and} \quad \partial u(A_0)
=
A_0+\left[-\frac12,\frac12\right]e.
\]
Translation invariance of the Lebesgue measure gives
\[
\mu_u(A_+)=\mathcal L^n(A_+),
\qquad
\mu_u(A_-)=\mathcal L^n(A_-).
\]
On the other hand, in the coordinates $(x_\parallel,x_\perp)\in
\mathbb R\times e^\perp$, the set $\partial u(A_0)$ has the product representation
\[
\partial u(A_0)
=
\left[c-\frac12,c+\frac12\right]
\times
\{x_\perp:(c,x_\perp)\in A_0\}.
\]
The product formula for the Lebesgue measure therefore yields
\[
\begin{aligned}
\mu_u(A_0)
&=
\mathcal L^n\bigl(\partial u(A_0)\bigr)\\
&=
\mathcal L^1\left(
\left[c-\frac12,c+\frac12\right]
\right)
\mathcal H^{n-1}(A_0)\\
&=
\mathcal H^{n-1}(A_0).
\end{aligned}
\]
Combining the preceding identities, we obtain
\[
\begin{aligned}
\mu_u(A)
&=
\mathcal L^n(A_+)
+\mathcal L^n(A_-)
+\mathcal H^{n-1}(A\cap H)\\
&=
\mathcal L^n(A\setminus H)
+\mathcal H^{n-1}(A\cap H)\\
&=
\mathcal L^n(A)
+\mathcal H^{n-1}(A\cap H),
\end{aligned}
\]
where the last equality follows from
$\mathcal L^n(H)=0$. Hence
\[
\mu_u
=
\mathcal L^n\res\Omega
+
\mathcal H^{n-1}\res(H\cap\Omega)
\]
as outer measures on $\Omega$.
\end{enumerate}
\end{example}

Now, we formulate the precise definition of Alexandrov solutions to Monge--Ampère equations and associated Dirichlet problems.

\begin{definition}[Alexandrov solution]
\label{def:alexandrov-solution}
Let $\Omega\subset\R^n$ be an open convex set. Assume that $\nu$ is a Borel measure on $\Omega$ and $\varphi\in C(\partial\Omega)$ is a continuous boundary value. 
\begin{enumerate}[(i)]
    \item An \emph{Alexandrov solution} of
\begin{equation}
\label{eq:measure-equation}
 \det D^2u=\nu\quad\text{in }\Omega
\end{equation}
is a convex function $u\colon \Omega\to\R$ such that
\begin{equation}
\label{eq:alexandrov-definition}
 \mu_u=\nu
 \quad\text{as Borel measures in }\Omega,
\end{equation}
that is, $\mu_u(B)=\nu(B)$ for all Borel sets $B\subset \Omega$.

If $\nu=f\,dx$ for some measurable function
$f\colon\Omega\to[0,\infty]$, then we also say that $u$ solves
\[
\det D^2u=f
\quad\text{in }\Omega
\]
in the Alexandrov sense.
    \item An \emph{Alexandrov solution of the Dirichlet problem }
\[
\begin{cases}
    \det D^2u=\nu\quad&\text{in }\Omega,\\
    u=\varphi\quad&\text{on }\partial\Omega
\end{cases}
\]
is a function $u\colon\Omega\to\mathbb R$ such that
\begin{enumerate}[(a)]
\item $u\in C(\overline\Omega)$,
\item $u$ is an Alexandrov solution of
      \eqref{eq:measure-equation}, and
\item $u=\varphi$ pointwise on $\partial\Omega$.
\end{enumerate}
\end{enumerate}
\end{definition}

\begin{remark}[Alexandrov vs. classical solutions]
\label{rem:Alexandrov-Classica-sol}
Suppose that $\nu=f\,dx$, for some nonnegative $f\in L^1_{\rm loc}(\Omega)$. If $u\in C^{1,1}(\overline\Omega)$ is convex, then Lemma~\ref{lem:MA-measure} gives $\mu_u=(\det D^2u)\,dx$. Consequently, $u$ is an Alexandrov solution of $\det D^2u=f$ if and only if $\det D^2u=f$ almost everywhere in $\Omega$. Thus, in the $C^{1,1}$ class, the Alexandrov and almost-everywhere notions of solution agree.

In particular, every classical convex solution $u\in C^{1,1}(\overline{\Omega})\cap C^2(\Omega)$ is an Alexandrov solution. Conversely, if an Alexandrov solution belongs to $C^{1,1}(\overline{\Omega})\cap C^2(\Omega)$ and $f$ is continuous, then $\det D^2u=f$ pointwise in $\Omega$ and hence $u$ is a classical solution.
\end{remark}

\begin{remark}
    Note that the convexity of an Alexandrov solution can be viewed as a consistency requirement, since a $C^2$ function $v$ is convex if and only if $D^2 v \geq 0$. 
\end{remark}

\subsection{Properties of Monge--Ampère measures}
\label{subsec:prop-MA-meas}

We record several fundamental properties of Monge--Ampère measures: scaling, superadditivity, monotonicity, and weak continuity under locally uniform convergence. These properties will be used repeatedly below, and are particularly relevant in the Dirichlet theory.

\begin{proposition}[{Scaling, translation invariance, superadditivity, and monotonicity}]
\label{prop:scaling}
Let $\Omega\subset\R^n$ be a convex open set, let $\lambda>0$, let $\ell_{p,c}\colon \R^n\to \R$ be the affine function $\ell_{p,c}(x)\vcentcolon = p\cdot x+c$, and suppose that $u,v$ are convex functions on $\Omega$. The Monge--Ampère measure has the following properties.
\begin{enumerate}[(a)]
    \item\label{scaling} \emph{Scaling}: $\mu_{\lambda u}=\lambda^n\mu_u$.
    \item\label{affine} \emph{Translation invariance}: $\mu_{u + \ell_{p,c} }=\mu_u$.
     \item\label{superadd} \emph{Superadditivity}: $\mu_{u+v}\geq \mu_u+\mu_v$ as Borel measures.
    \item\label{monotonicity} \emph{Monotonicity}: Let $O\subset \Omega$ be a bounded open set such that $u,v\in C(\overline{O})$. If the convex functions satisfy
    \begin{equation}
    \label{eq:monotonicity}
        u\leq v\text{ in }O,\quad
        u=v\text{ on }\partial O,
    \end{equation}
    then $\mu_u(O)\geq \mu_v(O)$.
\end{enumerate}
\end{proposition}

\begin{proof}
    The properties~\ref{scaling},~\ref{superadd},~and \ref{monotonicity} in the sense of Borel measures are established in \cite[Lemmas~2.7 and~2.9]{Figalli}. The scaling property \ref{scaling} and translation invariance \ref{affine} hold for all subsets $A\subset \Omega$ and not only for Borel sets, since its proof only relies on properties of the subdifferential and the scaling property and translation invariance of the Lebesgue measure, which hold for arbitrary sets; see~\cite[Section~2.1]{evans-mt}.
    
    Hence it remains to prove \ref{affine}. 
    By the sum rule for the subdifferential and Remark~\ref{rem:rel-subgrad-grad}, we have
    \[
        \partial (u+\ell_{p,c})(x)=\partial u(x)+\{\nabla \ell_{p,c}(x)\}=\partial u(x)+\{p\}.
    \]
    Thus, the translation invariance of the Lebesgue measure yields \ref{affine}.
\end{proof}

Before stating weak continuity under local uniform convergence, let us recall the following classical theorem.

\begin{proposition}[Weak convergence of locally finite Borel measures]
\label{prop:weak-convergence-borel-measures}
    Let $\Omega\subset\R^n$ be an open set. Assume that $\nu_k$, $k\in\mathbb N$, and $\nu$ are locally finite Borel measures on $\Omega\subset\mathbb R^n$ (see~Section~\ref{sec:measures}). Then the following conditions are equivalent:
\begin{enumerate}[(i)]
\item $\nu_k\weakstar \nu$, that is, 
\[
\int_\Omega\varphi\,d\nu_k \to
\int_\Omega\varphi\,d\nu
\qquad
\text{for every }\varphi\in C_c(\Omega)
\]
as $k\to\infty$.
\item\label{weak-conv-2} For every compact set $K\subset\Omega$ and every open set
$U\subset\Omega$,
\[
\limsup_{k\to\infty}\nu_k(K)\leq\nu(K)
\quad\text{and}\quad
\nu(U)\leq\liminf_{k\to\infty}\nu_k(U).
\]
\item For every relatively compact Borel set $B\Subset\Omega$
satisfying $\nu(\partial B)=0$,
\[
\nu_k(B)\to\nu(B).
\]
\end{enumerate}
Moreover, it is enough to consider in \ref{weak-conv-2} relatively compact open sets $U\Subset \Omega$.
\end{proposition}

\begin{proof}
Let $\eta$ be a Borel measure on $\Omega$. Define
\[
\widehat\eta\colon 2^\Omega\to[0,\infty],
\qquad
\widehat\eta(A)
\vcentcolon = 
\inf\bigl\{
\eta(B):
A\subset B,\ B\subset\Omega\text{ Borel}
\bigr\}.
\]
We first verify that $\widehat\eta$ is a Borel regular measure.

Clearly, $\widehat\eta(\varnothing)=0$. To prove
$\sigma$-subadditivity, let $A\subset\bigcup_{j=1}^\infty A_j$. For every $\varepsilon>0$, choose Borel sets $B_j\subset\Omega$
such that
\[
A_j\subset B_j
\quad\text{and}\quad
\eta(B_j)\leq\widehat\eta(A_j)+\frac{\varepsilon}{2^j}.
\]
Then $A\subset\bigcup_{j=1}^\infty B_j$, and hence
\[
\widehat\eta(A)
\leq
\eta\left(\bigcup_{j=1}^\infty B_j\right)
\leq
\sum_{j=1}^\infty\widehat\eta(A_j)+\varepsilon.
\]
Letting $\varepsilon\downarrow0$ shows that $\widehat\eta$ is an
outer measure.

Next, every Borel set $E\subset\Omega$ is
$\widehat\eta$-measurable. Indeed, if $A\subset\Omega$ and
$B\supset A$ is Borel, then
\[
\eta(B)
=
\eta(B\cap E)+\eta(B\setminus E)
\geq
\widehat\eta(A\cap E)+\widehat\eta(A\setminus E).
\]
Taking the infimum over all such Borel sets $B$ yields
\[
\widehat\eta(A)
\geq
\widehat\eta(A\cap E)+\widehat\eta(A\setminus E).
\]
The reverse inequality follows from the $\sigma$ subadditivity of $\widehat\eta$. Thus $\widehat\eta$ is a Borel measure.

Moreover,
\begin{equation}
\label{eq:regularization-agrees-on-Borel}
\widehat\eta(B)=\eta(B)
\qquad
\text{for every Borel set }B\subset\Omega.
\end{equation}
Indeed, the inequality $\widehat\eta(B)\leq\eta(B)$ follows by using $B$ itself in the defining infimum. Conversely, if $C\supset B$ is Borel, then $\eta(B)\leq\eta(C)$, and taking the infimum over $C$ gives the reverse inequality.

It remains to verify Borel regularity. Let $A\subset\Omega$. If $\widehat\eta(A)=\infty$, then
\[
\widehat\eta(\Omega)\geq\widehat\eta(A)=\infty,
\]
so that $\Omega$ is a Borel superset of $A$ having the same measure.
Suppose that $\widehat\eta(A)<\infty$. For every $j\in\mathbb N$,
choose a Borel set $C_j\subset\Omega$ such that
\[
A\subset C_j
\quad\text{and}\quad
\eta(C_j)\leq\widehat\eta(A)+\frac1j.
\]
Set
\[
B_j\vcentcolon = \bigcap_{\ell=1}^j C_\ell,
\qquad
B\vcentcolon = \bigcap_{j=1}^\infty B_j.
\]
Then $(B_j)_{j\in\mathbb N}$ is decreasing, $A\subset B$, and $\eta(B_1)<\infty$. Hence continuity from above, together with \eqref{eq:regularization-agrees-on-Borel}, gives
\[
\widehat\eta(B)
=
\eta(B)
=
\lim_{j\to\infty}\eta(B_j)
=
\widehat\eta(A).
\]
Thus $\widehat\eta$ is Borel regular. 

Furthermore, if $\eta$ is locally finite, then for every compact set $K\subset\Omega$,
\[
\widehat\eta(K)=\eta(K)<\infty.
\]
Consequently, $\widehat\eta$ is a Radon measure, whenever $\eta$ is a locally finite Borel measure.

Now, let $\widehat\nu_k$ and $\widehat\nu$
denote the Borel regularizations of $\nu_k$ and $\nu$, respectively. By \eqref{eq:regularization-agrees-on-Borel}, these Radon measures agree with the original measures on all Borel sets. In particular,
\[
\int_\Omega\varphi\,d\widehat\nu_k
=
\int_\Omega\varphi\,d\nu_k,
\qquad
\int_\Omega\varphi\,d\widehat\nu
=
\int_\Omega\varphi\,d\nu
\]
for every $\varphi\in C_c(\Omega)$, and the regularized measures have
the same values as the original measures Borel
subsets of $\Omega$.

Therefore, the assertions follow from \cite[Section~1.9, Theorem~1]{evans-mt}, applied to $\widehat\nu_k$ and $\widehat\nu$. Although the cited theorem is stated on $\mathbb R^n$, the same proof applies on the locally compact space $\Omega$, with relatively compact Borel sets in place of bounded Borel sets.

Finally, we show that it is enough to consider in \ref{weak-conv-2} relatively compact open sets $U\Subset \Omega$. Therefore, let $U\subset \Omega$ be given and define
\[
U_j
\vcentcolon = 
\left\{
x\in U:
\operatorname{dist}(x,\mathbb R^n\setminus U)>\frac1j
\right\}
\cap B_j.
\]
Then $(U_j)_{j\in\mathbb N}$ is an increasing sequence of relatively compact open sets such that
\[
U_j\Subset U
\qquad\text{and}\qquad
U=\bigcup_{j=1}^\infty U_j.
\]
Consequently, if the second condition in \ref{weak-conv-2} holds for all relatively compact open sets, then continuity from below gives
\[
\begin{aligned}
\nu(U)
&=\sup_{j\in\mathbb N}\nu(U_j)\\
&\leq
\sup_{j\in\mathbb N}\liminf_{k\to\infty}\nu_k(U_j)\\
&\leq
\liminf_{k\to\infty}\nu_k(U).
\end{aligned}
\]
Thus, it holds automatically for all open subsets $U\subset \Omega$ as we wanted to demonstrate. This completes the proof.
\end{proof}

\begin{proposition}[{Weak continuity property}]
\label{prop:weak-cont}
    Let $\Omega\subset\R^n$ be a convex open set. Assume that $u_k\colon \Omega\to\R$, $k\in\N$, and $u\colon \Omega\to\R$ are convex functions. If $u_k\to u$ locally uniformly as $k\to\infty$, then the Monge--Ampère measures $\mu_{u_k}$ converge weak-$*$ to $\mu_u$. 
\end{proposition}

    \begin{proof}
By Lemma~\ref{lem:MA-measure}, the Monge--Ampère measures
$\mu_{u_k}$ and $\mu_u$ are locally finite Borel measures. The proof of \cite[Proposition~2.6]{Figalli} shows that
\[
\limsup_{k\to\infty}\mu_{u_k}(K)
\leq
\mu_u(K)
\]
for every compact set $K\subset\Omega$, and
\[
\mu_u(U)
\leq
\liminf_{k\to\infty}\mu_{u_k}(U)
\]
for every open set $U\Subset\Omega$.

Hence Proposition~\ref{prop:weak-convergence-borel-measures} yields the asserted weak-$*$ convergence.
\end{proof}

\subsection{Dirichlet theory and comparison principles}
\label{sec:Dirichlet-theory}

Two basic results in the theory of Alexandrov solutions are the Alexandrov maximum principle and the comparison principle. Both results are essentially a consequence of Proposition~\ref{prop:scaling}. 

\begin{proposition}[{Alexandrov maximum principle,\cite[Theorem~2.8]{Figalli}}]
\label{prop:max-principle}
    Let $\Omega\subset\R^n$ be a convex bounded open set and let $u\in C(\overline{\Omega})$ be a convex function that vanishes on the boundary $\partial\Omega$. There exists a dimensional constant $C_n>0$ such that
    \begin{equation}
    \label{eq:Alexandrov-max-princ}
        |u(x)|^n\leq C_n \operatorname{diam}(\Omega)^{n-1}\mu_u(\Omega)\dist(x,\partial\Omega)
    \end{equation}
    for every $x\in\Omega$.
\end{proposition}

\begin{proposition}[{Comparison principle,\cite[Theorem~2.10]{Figalli}}]
\label{prop:Comparison}
    Let $\Omega\subset\R^n$ be a convex open set and let $u,v\colon \Omega\to\R$ be convex functions. Assume that $O\subset \Omega$ is a bounded open set. If the functions $u,v\in C(\overline{O})$ satisfy
    \[
        u\geq v\text{ on }\partial O
    \]
    and
    \[
        \mu_u\leq \mu_v\quad\text{as Borel measures on }O,
    \]
    then $u\geq v$ in $O$.
\end{proposition}

We use the following standard well-posedness and stability theory for
Alexandrov solutions.

\begin{proposition}[Well-posedness of the Dirichlet problem]
\label{prop:forward}
Let $\Omega\subset\mathbb R^n$ be a bounded strictly convex domain in the sense of Definition~\ref{def:strictly-convex-domains}.
\begin{enumerate}[(i)]
\item\label{forward-existence}
\textit{Existence and uniqueness:}
For every finite Borel measure $\nu$ on $\Omega$ and every
$\varphi\in C(\partial\Omega)$, there exists a unique convex function
$u\in C(\overline\Omega)$ satisfying
\begin{equation}
\label{eq:MA-well-posedness}
\begin{cases}
\det D^2u=\nu & \text{in }\Omega,\\
u=\varphi & \text{on }\partial\Omega
\end{cases}
\end{equation}
in the Alexandrov sense (see~Definition~\ref{def:alexandrov-solution}).

\item\label{forward-source-stability}
\textit{Stability under source perturbations:}
Let $(\nu_k)_{k\in\mathbb N}$ be finite Borel measures on $\Omega$
such that
\[
\sup_{k\in\mathbb N}\nu_k(\Omega)<\infty
\qquad\text{and}\qquad
\nu_k\weakstar\nu\quad \text{as }k\to\infty.
\]
For a fixed $\varphi\in C(\partial\Omega)$, let $u_k$ and $u$ be the solutions corresponding to $(\nu_k,\varphi)$ and $(\nu,\varphi)$,
respectively. Then
\[
u_k\to u
\qquad\text{uniformly on }\overline\Omega\text{ as }k\to\infty.
\]

\item\label{forward-joint-stability}
\textit{Stability under simultaneous perturbations:}
Let $(\nu_k)_{k\in\N}$ be finite Borel measures satisfying the assumptions in \ref{forward-source-stability}. Let $\varphi_k,\varphi\in C(\partial\Omega)$, $k\in\N$, satisfy
\[
\varphi_k\to\varphi
\qquad\text{uniformly on }\partial\Omega\text{ as }k\to\infty.
\]
If $u_k$ is the solution corresponding to $(\nu_k,\varphi_k)$ and $u$ is the solution corresponding to $(\nu,\varphi)$, then
\[
u_k\to u
\qquad\text{uniformly on }\overline\Omega\text{ as }k\to\infty.
\]
\end{enumerate}
\end{proposition}

\begin{proof}
    \ref{forward-existence}: By the comparison principle (Proposition~\ref{prop:Comparison}), the uniqueness assertion holds under the weaker conditions that $\Omega\subset \R^n$ is a convex bounded open set and $\nu$ a Borel measure. Existence of a solution to the Dirichlet problem \eqref{eq:MA-well-posedness} is established in \cite[Theorem~2.14]{Figalli}.  \\
    
    \noindent \ref{forward-source-stability}: The proof of \cite[Proposition~2.16]{Figalli} shows that $u_k$ converges locally uniformly to the unique solution of the Dirichlet problem \eqref{eq:MA-well-posedness}. Hence, it remains to show that the convergence holds uniformly in $\overline\Omega$. 
    
    Fix $\varepsilon>0$ and $x\in \partial\Omega$. Since $\Omega$ is a strictly convex domain, \cite[Lemma~2.15]{Figalli} yields the existence of affine functions $\ell_{x,\eps}^{\pm}$ satisfying
    \begin{equation}
    \label{eq:barriers}
       \ell_{x,\eps}^{-}\leq \varphi\leq \ell_{x,\eps}^{+}\quad\text{on }\partial\Omega \quad\text{and}\quad \varphi(x)-\eps\leq \ell_{x,\eps}^{-}(x)\leq \ell_{x,\eps}^{+}(x)\leq \varphi(x)+\eps.
    \end{equation}
    We observe that $u_k$ satisfies
    \begin{equation}
    \label{eq:barrier-est}
        w_k+\ell_{x,\eps}^{-}\leq u_k\leq w_k+\ell_{x,\eps}^+\quad \text{in }\Omega
    \end{equation}
    where $w_k$ is the unique solution to 
    \begin{equation}
    \label{eq:MA-well-posedness-hom}
    \begin{cases}
        \det D^2w=\nu_k & \text{in }\Omega,\\
        w=0 & \text{on }\partial\Omega.
    \end{cases} 
    \end{equation}
    Indeed, Proposition~\ref{prop:scaling} \ref{affine} shows that $w_k+\ell_{x,\eps}^{\pm}$ solves \eqref{eq:MA-well-posedness} with source $\nu_k$ and boundary values $\ell_{x,\eps}^\pm |_{\partial\Omega}$. Thus, \eqref{eq:barrier-est} is a consequence of the comparison principle (Proposition~\ref{prop:Comparison}) and \eqref{eq:barriers}. 
    
    Moreover, by Proposition~\ref{prop:max-principle}, the convexity of $w_k$, and $w_k|_{\partial\Omega}=0$, we also have
    \begin{equation}
    \label{eq:two-sided-hom}
        -C \dist(y,\partial\Omega)^{1/n}\leq w_k(y)\leq 0\quad\text{for all }y\in\Omega.
    \end{equation}
    Therefore \eqref{eq:barrier-est} yields
    \[
        -C\dist(y,\partial\Omega)^{1/n}+\ell_{x,\eps}^{-}(y)\leq u_k (y) \leq \ell_{x,\eps}^+ (y),\quad \text{for all }y\in\Omega.
    \]
   Hence, in a sufficiently small collar of the boundary we have $C\dist(y,\partial\Omega)^{1/n}<\eps$. Thus the continuity of $\ell^\pm_{x,\eps}$ and \eqref{eq:barriers} yield a neighborhood $U_x \subset \overline{\Omega}$ of $x$ such that
\begin{equation}
\label{eq:barrier-uk}
    \varphi (x)-3\varepsilon \leq u_k(y) \leq \varphi(x)+2\varepsilon
\end{equation}
for all $k$ and all $y\in U_x\cap\Omega$ sufficiently close to $\partial\Omega$. Up to shrinking the neighborhood $U_x$, the same estimate holds for $u$ in place of $u_k$, since $u\in C(\overline{\Omega})$ and $u|_{\partial\Omega}=\varphi$.

By compactness of $\partial\Omega$, finitely many of the sets $U_x$ cover the boundary. Hence, there exists $\delta>0$ such that
\[
|u_k(y)-u(y)|\leq  5 \varepsilon
\]
for every $k$ whenever $\operatorname{dist}(y,\partial\Omega)<\delta$. Indeed, if $u_k(y)\geq u(y)$ for some $y\in \Omega$ with $\dist(y,\partial\Omega)<\delta$, then it lies in some neighborhood $U_x$ and \eqref{eq:barrier-uk} implies
\[
    |u_k(y)-u(y)|=(u_k(y)-\varphi(x))-(u(y)-\varphi(x))\leq 5\eps
\]
and similarly for the case $u_k(y)\leq u(y)$. Since on the compact set
\[
K_\delta
\vcentcolon = 
\{z\in\overline\Omega:
\operatorname{dist}(z,\partial\Omega)\geq\delta\}
\Subset\Omega,
\]
the convergence is uniform by local uniform convergence. Therefore
\[
\limsup_{k\to\infty}
\|u_k-u\|_{L^\infty(\Omega)}
\leq5\varepsilon.
\]
Since $\varepsilon>0$ was arbitrary, $u_k\to u$ uniformly on $\overline\Omega$. This completes the proof of \ref{forward-source-stability}.\\
    
    \noindent \ref{forward-joint-stability}: Let $v_k$ denote the solution corresponding to $(\nu_k,\varphi)$ and set
\[
\varepsilon_k
\vcentcolon = 
\|\varphi_k-\varphi\|_{L^\infty(\partial\Omega)}.
\]
Proposition~\ref{prop:scaling} \ref{affine} demonstrates that $v_k\pm \eps_k$ solve \eqref{eq:MA-well-posedness} with source $\nu_k$ and boundary values $\varphi\pm \eps_k$. Since $u_k=\varphi_k$ on $\partial\Omega$, the comparison principle (Proposition~\ref{prop:Comparison}) therefore gives
\[
v_k-\varepsilon_k
\leq u_k
\leq v_k+\varepsilon_k
\qquad\text{on }\overline\Omega.
\]
Consequently,
\[
\|u_k-v_k\|_{L^\infty(\overline\Omega)}
\leq\varepsilon_k.
\]
By assertion~\ref{forward-source-stability}, $v_k\to u$ uniformly on $\overline\Omega$, while $\varepsilon_k\to0$. Hence $u_k\to u$ uniformly on $\overline\Omega$.
\end{proof}

\begin{remark}[Remarks on the existence of solutions]
\label{rem:existence-of-sols}
As mentioned the existence of a solution to \eqref{eq:MA-well-posedness} is provided by \cite[Theorem~2.14]{Figalli}, when $\varphi\in C(\partial\Omega)$ is arbitrary, and by \cite[Theorem~2.13]{Figalli} for the special case $\varphi=0$. Both results apply the Perron method, where the subsolution set is
    \[
        \mathcal{S}[\nu]\vcentcolon = \{w\colon \Omega\to \R:w\in C(\overline{\Omega})\text{ convex},\, w \leq \varphi\text{ on }\partial\Omega,\,\mu_w\geq \nu\text{ as Borel measures}\},
    \]
    the weak-$*$ approximation of $\nu$ by nonnegative finite linear combinations of Dirac masses,  $\sum_{n=1}^N a_n \delta_{x_n}$, $a_n\geq 0$, and the stability under source perturbations \ref{forward-source-stability}. We will also use the same strategy to construct solutions for the weighted Monge--Ampère equation (see Section~\ref{sec:weighted-MA-extension}).

    When $\varphi=0$, the source-stability assertion \ref{forward-source-stability} of Proposition~\ref{prop:forward} holds on arbitrary bounded convex domains and is a special case of \cite[Corollary~2.12]{Figalli}. The proof uses the Alexandrov maximum principle (Proposition~\ref{prop:max-principle}) to control solutions near the boundary, together with compactness, weak continuity of Monge--Ampère measures, and uniqueness.

The argument of \cite[Corollary~2.12]{Figalli} also allows varying bounded convex domains $\Omega_k$, provided\footnote{In \cite[Corollary~2.12]{Figalli}, the condition is written as $\dist(\partial\Omega_k,\partial\Omega)\to 0$. If $\dist$ denotes the usual infimum distance between sets, this condition does not imply the convergence of the domains or their diameters used in the proof: shrinking balls externally tangent to $\Omega$ give a counterexample. We therefore use Hausdorff distance, which provides the required convergence for bounded convex domains.}
\[
d_H(\partial\Omega_k,\partial\Omega)\to 0
\]
and the source measures have uniformly bounded total masses. Here, for nonempty sets $A,B\subset\mathbb R^n$,
\[
d_H(A,B)\vcentcolon = 
\max\left\{
\sup_{x\in A}\operatorname{dist}(x,B),
\sup_{y\in B}\operatorname{dist}(y,A)
\right\}.
\]
The resulting convergence of the related solutions $u_k$ is locally uniform in $\Omega$. For a fixed domain, the maximum principle estimate \eqref{eq:Alexandrov-max-princ} and the uniform boundedness of $\nu_k(\Omega)$ imply
\[
    |u_k(x)|\leq C \dist(x,\partial\Omega)^{1/n},\quad x\in\Omega.
\]
Hence in this case the local uniform convergence can be upgraded to uniform convergence on $\overline\Omega$.

We highlight that for the general case \eqref{eq:MA-well-posedness} the strict convexity of $\Omega$ is needed to obtain boundary barriers at every point of $\partial\Omega$, which ensure that the limiting solution attains an arbitrary prescribed datum $\varphi\in C(\partial\Omega)$ continuously. More precisely, they are needed in the construction of $\ell_{x,\eps}^\pm$ (see~\cite[Lemma~2.15]{Figalli} and proof of Proposition~\ref{prop:forward} \ref{forward-source-stability}).
\end{remark}

\begin{remark}[Interior $C^1$-regularity and strict convexity in dimension two]
\label{rem:two-dimensional-regularity}
In two space dimensions, upper and lower bounds for the density of the Monge--Ampère measure impose strong geometric and regularity properties on Alexandrov solutions. Suppose that $\Omega\subset\mathbb R^2$ and that
\[
\mu_u=f\,dx
\quad\text{in }\Omega.
\]
Then the following assertions hold:
\begin{enumerate}[(i)]
\item\label{smoothness-2d} If $f\in L^\infty_{\mathrm{loc}}(\Omega)$, then either
      $u\in C^1(\Omega)$ or there exists a segment $\Sigma$ crossing
      $\Omega$ along which $u$ is affine and non-differentiable; see
      \cite[Theorem~2.17]{Figalli}.
\item\label{strict-convexity-2d} If $1/f\in L^\infty_{\mathrm{loc}}(\Omega)$, then $u$ is
      strictly convex in $\Omega$; see
      \cite[Theorem~2.19]{Figalli}.
\end{enumerate}
Consequently, if for any compact subset $K\Subset \Omega$ there exist constants $c_K,C_K>0$ such that
\[
    c_K\leq f(x)\leq C_K\quad \text{for a.e. }x\in K,
\]
then $u$ is strictly convex and belongs to $C^1(\Omega)$. Indeed, strict convexity rules out the affine-segment alternative in \ref{smoothness-2d}. 
\end{remark}

\section{One-sided convex boundary slope}
\label{sec:one-side-slope}

In this section, we introduce the one-sided convex boundary slope $q_u$, which turns out to be the appropriate generalization of the Neumann derivative $\partial_\nu u$. First, in Section~\ref{subsec:further-notation}, we review the subdifferential of convex functions $v\colon \Omega\to (-\infty,\infty]$. Afterwards in Section~\ref{subsec:extremal-slope} we introduce the one-sided convex boundary slope $q_u$, show that it suitably generalizes the normal derivative and demonstrate that it is Borel measurable under suitable convexity assumptions. Finally, in Section~\ref{subsec:interior-normal-derivative}, we show that the boundary slope coincides with interior normal derivative for Lipschitz functions.

\subsection{Further convex analysis notation}
\label{subsec:further-notation}

We need to slightly extend our convex-analytic preliminaries from our earlier sections. Everything covered in this section can be found in the classical monograph \cite{Rockafellar}. Let $C\subset \R^n$ be a convex set and $D\subset \R^n$ an arbitrary set.

\begin{enumerate}[(a)]
    \item\label{domain-epigraph-extended} Let $f\colon C\to (-\infty,\infty]$ be a function, then the \emph{(effective) domain} and the \emph{epigraph} of $f$ are the sets
    \[
        \operatorname{dom}(f)\vcentcolon = \{x\in C: f(x)<\infty\}\subset C
    \]
    and 
    \[
        \operatorname{epi}(f)\vcentcolon =\{(x,t)\in C\times \R:t\geq f(x)\}\subset C\times \R,
    \]
    respectively.
    \item\label{convexity-extended} A function $f\colon C\to (-\infty,\infty]$ is called \emph{convex} if its epigraph is convex, or equivalently, for all $x,y\in \operatorname{dom}(f)$ and all $0<t<1$
    \[
        f(tx+(1-t)y)\leq tf(x)+(1-t)f(y);
    \]
    see \cite[Theorem~4.1]{Rockafellar}.
    \item\label{proper} A convex function $f\colon C\to (-\infty,\infty]$ is called \emph{proper} if $\operatorname{dom}(f)$ is nonempty.
    \item\label{subgradient-subdifferential-extended} If $f\colon C\to (-\infty,\infty]$ is a convex function, then $p\in \R^n$ is called \emph{subgradient} of $f$ at $x\in C$, whenever
    \begin{equation}
    \label{eq:subgradient-ineq}
        f(y)\geq f(x)+p\cdot (y-x),\quad \text{for all }y\in C.
    \end{equation}
    The subgradient inequality \eqref{eq:subgradient-ineq} only needs to be checked for vectors $y$ in $\operatorname{dom}(f)$. The collection of all subgradients of $f$ at $x$ is called the \emph{subdifferential} of $f$ at $x$ and denoted by $\partial f(x)$. The subdifferential $\partial f(x)$ is always a convex and closed subset of $\R^n$. Furthermore, by \cite[Theorem~23.4]{rockafellar1966characterization}, $\partial f(x)=\varnothing$ for all $x\notin \operatorname{dom}(f)$ and $\partial f(x)$ is nonempty and bounded, when $x\in \operatorname{int}(\operatorname{dom}(f))$. 
    \item\label{indicator-function} The \emph{indicator function} $I_D\colon \R^n\to (-\infty,\infty]$ of the set $D$ is defined by
    \[
    I_D(x)=\begin{cases}
        0,\quad x\in D,\\
        \infty,\quad x\notin D,
    \end{cases}
    \]
    which is convex if and only if $D$ is convex \cite[p.~28]{Rockafellar}.
\end{enumerate}

Throughout this section, we assume that $\Omega\subset \R^n$ is at least a (nonempty) convex bounded open set. We observe that every convex function $u\colon\Omega\to\R$ can be extended to a proper lower semicontinuous convex function on $\R^n$ by
\begin{equation}
\label{eq:extension-convex}
  \bar u(x)\vcentcolon =\begin{cases}
    u(x),\quad& \text{if }x\in \Omega,\\
      \liminf_{\Omega\ni y\to x}u(y),\quad&\text{if }x\in \partial\Omega,\\
      +\infty,\quad&\text{if }x\in \R^n\setminus \overline{\Omega}.
  \end{cases}
\end{equation}
Indeed, the convexity of $\bar u$ follows from \ref{convexity-extended} and the convexity of $\Omega$ and $u$. The function $\bar u$ is proper, because $\Omega$ is nonempty. The lower semicontinuity of $\bar u$ follows from the continuity of $u$ in $\Omega$ (see \cite[Corollary~A.23]{Figalli}) and the definition \eqref{eq:extension-convex}. 

Given such a convex function $u\colon \Omega \to \R$, then we will always denote by $\bar u \colon \R^n\to (-\infty,\infty]$ its proper lower semicontinuous convex extension defined by \eqref{eq:extension-convex}, and write 
\begin{equation}
\label{eq:extended-subdifferential}
     \bar{\partial} u (x)\vcentcolon = \partial \bar{u}(x),\quad x\in\R^n,
\end{equation}
where $\partial \bar{u}$ is the subdifferential of $\bar{u}\colon \R^n\to (-\infty,\infty]$ (see~\ref{subgradient-subdifferential-extended}). 
We also call $\bar\partial u$ the \emph{relative boundary subdifferential} of $u$. Next we note the following simple lemma.

\begin{lemma}[Boundary subdifferential]
\label{lemma:bdry-subdifferential}
    Let $\Omega\subset \R^n$ be a convex open set and let $u\colon \Omega\to \R$ be a convex function.
    \begin{enumerate}[(i)]
        \item\label{consistency} For any $x\in \Omega$, one has $\bar{\partial }u(x)=\partial u(x)$.
        \item\label{C1-functions} Assume, in addition, that $\Omega$ is a $C^1$ domain and $u\in C^1(\overline\Omega)$. Then, for every $x\in\partial\Omega$, we have
        \begin{equation}
        \label{eq:subgradient-c1}
            \bar{\partial} u(x)=\{\nabla u(x)+s\nu(x):s\geq 0\},
        \end{equation}
        where $\nu$ is the outward pointing unit normal vector field.
    \end{enumerate}
\end{lemma}

\begin{proof}
    \ref{consistency}: The inclusion $\bar{\partial }u(x)\subset \partial u(x)$ follows from the definitions. Now, suppose that $p\in \partial u(x)$. Then
    \[
        u(y)\geq u(x)+p\cdot (y-x)\quad\text{for all }y\in \Omega.
    \]
    Since $x\in \Omega$, this is equivalent to 
     \begin{equation}
     \label{eq:subgrad-proof}
        u(y)\geq \bar{u}(x)+p\cdot (y-x)\quad\text{for all }y\in \Omega.
    \end{equation}
    Now, let $y\in \partial \Omega$ and choose a sequence $(y_n)_{n\in\N}\subset \Omega$ such that
    \[
        \lim_{n\to\infty}u(y_n)=\liminf_{\Omega\ni z\to y}u(z)=\bar{u}(y).
    \]
    Using \eqref{eq:subgrad-proof} with $y=y_n$ and passing to the limit $n\to\infty$, we find
    \[
        \bar u (y)\geq \bar{u}(x)+p\cdot (y-x).
    \]
    Since $\bar u=\infty$ in $\R^n\setminus\overline{\Omega}$, we deduce that $p\in \bar{\partial}u(x)$. \\

    \noindent \ref{C1-functions}: Fix $x\in\partial\Omega$ and choose $x_k\in\Omega$ with $x_k\to x$. By \ref{consistency} and $u\in C^1(\Omega)$, we have $\bar{\partial}u(x_k)=\partial \bar{u}(x_k)=\{\nabla u(x_k)\}$ for each $k\in\N$. Thus, 
\[
 u(z)\geq u(x_k)+\nabla u(x_k)\cdot(z-x_k),\qquad z\in\Omega
\]
for each $k\in\N$. Therefore, since $u\in C^1(\overline{\Omega})$, passing to the limit $k\to\infty$ shows 
\[
 u(z)\geq u(x)+\nabla u(x)\cdot(z-x),\qquad z\in\Omega.
\]
By continuity of $u$ on $\overline{\Omega}$ this holds for all $z\in\overline{\Omega}$. Since $u=\bar{u}$ on $\partial\Omega$ and $\bar u=\infty$ in $\R^n\setminus \overline{\Omega}$, the above inequality implies that
$\nabla u(x)\in\bar{\partial} u(x)$. By \eqref{eq:left-tangent-plane} we know that
$\nu(x)\cdot(y-x)\leq 0$ for $y\in \overline{\Omega}$, and hence it follows that $\nabla u(x)+s\nu(x)\in\bar{\partial} u(x)$ for every $s\geq 0$. Therefore, 
\[
    \{\nabla u(x)+s\nu(x):s\geq 0\}\subset \bar{\partial} u(x).
\]

Conversely, let $p\in\bar{\partial} u(x)$. Choose a $C^1$ boundary curve $\gamma:(-\varepsilon,\varepsilon)\to
\partial\Omega$ with $\gamma(0)=x$ and
$\gamma'(0)=\tau_0\in T_x\partial\Omega$. Since $p\in \bar\partial u(x)$ and $\bar u=u$ on $\partial \Omega$, we have
\[
 u(\gamma(t))-u(x)
 \geq p\cdot(\gamma(t)-x)\quad\text{for all }|t|<\eps.
\]
Dividing by positive and negative $t$ and passing to zero shows
\[
 p\cdot\tau_0=\nabla_{\tau} u(x)\cdot\tau_0,
\]
where $\nabla_\tau u \vcentcolon = \nabla u- (\partial_\nu u)\nu$ is the tangential derivative of $u$ along $\partial\Omega$.

Thus, by varying $\tau_0\in T_x\partial\Omega$, we see that every element of $\bar{\partial} u(x)$ has the form
$\nabla_\tau u(x)+a\nu(x)$.
For all sufficiently small $r>0$, the point $x-r\nu(x)$ lies in $\Omega$.
The subgradient inequality and the first-order expansion of $u$ at $x$ give
\[
 u(x)-r\nabla u(x)\cdot\nu(x)+o(r)
 =u(x-r\nu(x))\geq u(x)-ar.
\]
After division by $r$ and passage to the limit, this yields 
\[
    a\geq\nabla u(x)\cdot\nu(x)=\partial_\nu u(x).
\]
Since $\nabla u(x) =\nabla_\tau u(x)+(\partial_\nu u (x))\nu(x)$, formula \eqref{eq:subgradient-c1} follows. 
\end{proof}


\subsection{Extremal boundary slope}
\label{subsec:extremal-slope}

In this section we introduce the extremal boundary slope $q_u$, which turns out to be the appropriate generalization of the normal derivative. We study its compatibility with other notions of normal derivatives in the next section. From now one, when $\Omega\subset \R^n$ is a $C^1$ domain, we will denote its outward pointing normal vector field by $\nu(x)$ and $\tau$ for a tangential vector. Accordingly, $\partial_\nu u$ denotes the normal derivative and $\nabla_\tau u$ the tangential derivative. 

\begin{proposition}[Extremal boundary slope]
\label{def:slope}
    Let $\Omega\subset\mathbb R^n$ be a convex domain with $C^1$ boundary, let $\varphi\in C^1(\partial\Omega)$, and suppose that $u\in C(\overline\Omega)$ is convex and satisfies $u|_{\partial\Omega}=\varphi$. For every $x\in\partial\Omega$, define the \emph{extremal boundary slope} by
\begin{equation}
\label{eq:defq}
q_u(x)
\vcentcolon=
\inf\left\{
a\in\mathbb R:
\nabla_\tau\varphi(x)+a\nu(x)\in\bar\partial u(x)
\right\},
\end{equation}
with the convention $\inf\varnothing=+\infty$. Then
\[
q_u(x)\in(-\infty,\infty]
\qquad\text{for every }x\in\partial\Omega.
\]
\end{proposition}

\begin{proof}
   Fix $x\in\partial\Omega$. If the set appearing in
\eqref{eq:defq} is empty, then $q_u(x)=+\infty$, and there is nothing to prove. Otherwise, let $a\in\R$ satisfy
    \[
        \nabla_\tau \varphi(x)+a\nu(x)\in \bar \partial u(x).
    \]
    Then the subgradient inequality implies that
    \[
        u(z)\geq u(x)+(\nabla_\tau \varphi(x)+a\nu(x))\cdot (z-x)\quad\text{for all }z\in \overline{\Omega}.
    \]
    Since $\Omega\in C^1$ and $\nu(x)$ is outward pointing, there exists a small $r>0$ such that $z=x-r\nu(x)\in \Omega$. Choosing this vector in the subgradient inequality and rearranging terms we get
    \[
        a\geq \frac{u(x)-u(x-r\nu(x))}{r},
    \]
    because $\nabla_\tau\varphi(x)\cdot\nu(x)=0$. The right-hand side is a finite real number and the estimate holds for every admissible $a$. Taking the infimum over all such $a$ yields
\[
q_u(x)\geq
\frac{u(x)-u(x-r\nu(x))}{r}>-\infty.
\]
Thus $q_u(x)\in(-\infty,\infty]$.
\end{proof}

The following lemma shows that this is the correct weak interpretation of the exterior normal derivative $\partial_\nu u$.

\begin{lemma}[Boundary subdifferential ray]
\label{lem:ray}
 Let $\Omega\subset\mathbb R^n$ be a convex domain with $C^1$ boundary, let $\varphi\in C^1(\partial\Omega)$, and suppose that $u\in C(\overline\Omega)$ is convex and satisfies $u|_{\partial\Omega}=\varphi$. For every $x\in\partial\Omega$,
\begin{equation}
\label{eq:ray}
 \bar{\partial} u(x)=
 \begin{cases}
 \{\nabla_\tau\varphi(x)+a\nu(x):a\geq q_u(x)\},
     &q_u(x)<+\infty,\\
 \varnothing,&q_u(x)=+\infty.
 \end{cases}
\end{equation}
In particular, if $q_u(x)<\infty$, the infimum in
\eqref{eq:defq} is attained.
\end{lemma}

\begin{proof}
 We first show that every subgradient $p$ of $u$ at $x\in \partial\Omega$ can be decomposed as 
 \begin{equation}
 \label{eq:decomp-subgrad}
    p=\nabla_\tau\varphi(x)+a\nu(x)\quad\text{for some }a\in\R.
 \end{equation}
 We argue similarly as in the second half of the proof of Lemma \ref{lemma:bdry-subdifferential} \ref{C1-functions}. Fix $p\in\bar\partial u(x)$, $\tau_0\in T_x\partial\Omega$ and let $\gamma\colon (-\varepsilon,\varepsilon)\to\partial\Omega$ be a $C^1$ boundary curve with $\gamma(0)=x$ and $\gamma'(0)=\tau_0$. The subgradient inequality with $\varphi=u|_{\partial\Omega}$ and $y=\gamma(t)$ gives
\[
 \varphi(\gamma(t))-\varphi(x)
 \geq p\cdot(\gamma(t)-x),\quad \text{for all }|t|<\eps.
\]
Dividing by positive and negative $t$ and taking the limit $t\downarrow 0$ ($t\uparrow 0$) shows
\[
 p\cdot\tau_0=\nabla_\tau\varphi(x)\cdot\tau_0.
\]
Since $\tau_0$ is arbitrary, every element of $\bar{\partial} u(x)$ has the form $\nabla_\tau\varphi(x)+a\nu(x)$ for some $a\in\R$. This establishes \eqref{eq:decomp-subgrad}.

Recall that the convexity of $\Omega$ implies
\[
    \nu(x)\cdot(z-x)\leq 0\quad\text{for all }z\in \overline{\Omega}
\]
and hence, if $p\in\bar \partial u(x)$ and $s\geq0$, then the subgradient inequality shows that
\[
 p+s\nu(x)\in\bar\partial u(x).
\]
Consequently, since $\bar \partial u(x)$ is always a closed convex set and every subgradient $p$ can be written in the form $p=\nabla_\tau \varphi(x)+a\nu(x)$ (see~\eqref{eq:decomp-subgrad}), the possible normal components $a=p\cdot \nu(x)$ form either the empty set, a closed upper ray $[a_0,\infty)$, or all of $\R$.

Finally, we show that the last possibility cannot occur. To this end, fix a sufficiently small $r>0$ such that $z=x-r\nu(x)\in\Omega$. If $p=\nabla_\tau\varphi(x)+a\nu(x)\in\bar\partial u(x)$ for some $a\in\R$, then the subgradient inequality at $z$ gives
\[
 u(z)\geq u(x)-ra,
 \qquad\text{hence}\qquad
 a\geq\frac{u(x)-u(z)}{r}.
\]
Thus the normal components of the subgradients $p\in \bar\partial u(x)$ are bounded below whenever they are nonempty. By definition of $q_u$ (see~\eqref{eq:defq}), the closed-ray description is exactly \eqref{eq:ray}.
\end{proof}

We next prove that the boundary slope $x\mapsto q_u$ Borel measurable. This will be useful when proving that its graph is a null set (Lemma~\ref{lem:S-null}).

\begin{lemma}[Measurability]
\label{lem:measurable}
Let $\Omega\subset\mathbb R^n$ be a bounded strictly convex domain with $C^1$ boundary, let $\varphi\in C^1(\partial\Omega)$, and suppose that $u\in C(\overline\Omega)$ is convex and satisfies $u|_{\partial\Omega}=\varphi$. Then $q_u\colon \partial\Omega\to(-\infty,\infty]$ is Borel measurable.
\end{lemma}

\begin{proof}
 First note that by Corollary \ref{cor:supporting-hyperplanes} strict convexity of $\Omega$ is equivalent to the condition that for each $x\in \partial\Omega$ one has
 \begin{equation}
 \label{eq:equiv-cond-strict-proof}
        -\nu(x)\cdot(z-x)>0\quad\text{for all }z\in \overline{\Omega}\setminus \{x\}.
 \end{equation}
We claim that
\begin{equation}
\label{eq:q-sup}
 q_u(x)=\sup_{z\in\Omega}
 \frac{u(x)+\nabla_\tau\varphi(x)\cdot(z-x)-u(z)}
 {-\nu(x)\cdot(z-x)},\quad \text{for each }x\in\partial\Omega.
\end{equation}
If $q_u(x)<\infty$, then Lemma~\ref{lem:ray} implies that 
\[
    u(z)\geq u(x)+(\nabla_\tau \varphi(x)+a\nu(x))\cdot (z-x)\quad\text{for all }z\in \overline{\Omega}
\]
for every $a\geq q_u(x)$. Equivalently,
\[
    -a\nu(x)\cdot (z-x)\geq u(x)+\nabla_\tau \varphi(x)\cdot (z-x)-u(z),\quad \text{for all }z\in \overline{\Omega}.
\]
Now, we choose $a=q_u(x)$ and use the positivity condition \eqref{eq:equiv-cond-strict-proof} to obtain
\[
     q_u(x)\geq \frac{u(x)+\nabla_\tau \varphi(x)\cdot (z-x)-u(z)}{-\nu(x)\cdot (z-x)},\quad \text{for all }z\in \Omega.
\]
Hence 
\[
q_u(x)\geq \sup_{z\in\Omega} \frac{u(x)+\nabla_\tau \varphi(x)\cdot (z-x)-u(z)}{-\nu(x)\cdot (z-x)}.
\]
Conversely, set
\[
S(x)\vcentcolon = 
\sup_{z\in\Omega}
\frac{u(x)+\nabla_\tau\varphi(x)\cdot(z-x)-u(z)}
{-\nu(x)\cdot(z-x)}.
\]
By the inequality $q_u(x)\geq S(x)$ and $q_u(x)<\infty$, we have $S(x)<\infty$. Moreover, the condition $u\in C(\overline{\Omega})$ implies $S(x)>-\infty$. Hence, by the definition of
$S(x)$,
\[
u(z)\geq
u(x)+\bigl(\nabla_\tau\varphi(x)+S(x)\nu(x)\bigr)\cdot(z-x)
\qquad\text{for every }z\in\Omega.
\]
Since $u\in C(\overline{\Omega})$, this inequality also holds for every $z\in\overline\Omega$. Since $\bar u=+\infty$ outside $\overline\Omega$, it follows that
\[
\nabla_\tau\varphi(x)+S(x)\nu(x)\in\bar\partial u(x).
\]
Hence Lemma~\ref{lem:ray} gives $q_u(x)\leq S(x)$. Hence, we have verified \eqref{eq:q-sup}, when $q_u(x)<\infty$.

Next, let us consider the case $q_u(x)=\infty$. We need to show $S(x)=+\infty$. If $S(x)<\infty$, then we again have $\nabla_\tau\varphi(x)+S(x)\nu(x)\in\bar\partial u(x)$, but this contradicts \eqref{eq:ray}. Thus, $S(x)<\infty$ cannot hold. Hence also in this case \eqref{eq:q-sup} holds true.

To prove measurability, choose a countable dense subset
$D\subset\Omega$. For fixed $x\in\partial\Omega$, the function
\[
z\mapsto
\frac{u(x)+\nabla_\tau\varphi(x)\cdot(z-x)-u(z)}
{-\nu(x)\cdot(z-x)}
\]
is continuous on $\Omega$. Therefore,
\[
q_u(x)
=
\sup_{z\in D}
\frac{u(x)+\nabla_\tau\varphi(x)\cdot(z-x)-u(z)}
{-\nu(x)\cdot(z-x)}.
\]
For every fixed $z\in D$, the function
\[
x\mapsto
\frac{\varphi(x)+\nabla_\tau\varphi(x)\cdot(z-x)-u(z)}
{-\nu(x)\cdot(z-x)}
\]
is continuous on $\partial\Omega$, because the denominator is strictly positive by \eqref{eq:equiv-cond-strict-proof}. Consequently, $q_u$ is the supremum of countably many continuous functions and is therefore Borel measurable.
\end{proof}

\subsection{Generalized global Cauchy data}
\label{subsec:Def-Cauchy-data}

Now, using the forward theory for the Monge--Ampère equation and the definition of the extremal slope, we can rigorously introduce the generalized global Cauchy data set.

\begin{proposition}[Generalized global Cauchy data]
\label{def:gen-cauchy-data}
    Let $\Omega\subset\mathbb R^n$ be a bounded strictly convex domain with $C^1$ boundary. For every finite Borel measure $\nu$, we define the \emph{generalized global Cauchy data set} by
    \begin{equation}
    \label{eq:generalized-global-Cauchy-data}
        \mathcal{C}^{\rm glob}_{\nu}\vcentcolon = \left\{(\varphi,q_{u_{\nu,\varphi}}):\varphi\in C^1(\partial\Omega)\right\}.
    \end{equation}
    Here, $u_{\nu,\varphi}\in C(\overline{\Omega})$ is the Alexandrov solution of the Dirichlet problem 
    \begin{equation}
    \label{eq:MA-cauchy}
        \begin{cases}
            \det D^2 u=\nu,\quad&\text{in }\Omega,\\
            u=\varphi,\quad&\text{on }\partial\Omega,
        \end{cases}
    \end{equation}
    and $q_{u_{\nu,\varphi}}$ is the extremal boundary slope. The Cauchy data set is well-defined and is a subset of $C^1(\partial\Omega)\times\mathcal{B}(\partial\Omega)$. 
\end{proposition}

\begin{proof}
    We first argue that all objects in \eqref{eq:generalized-global-Cauchy-data} are well-defined. First, for any finite Borel measure $\nu$ and every boundary datum $\varphi\in C(\partial\Omega)$, the Dirichlet problem \eqref{eq:MA-cauchy} is well-posed by Proposition~\ref{prop:forward} \ref{forward-existence}. We will denote this unique Alexandrov solution by $u_{\nu,\varphi}\in C(\overline{\Omega})$. Second, Proposition~\ref{def:slope} shows that the boundary slope
    \begin{equation}
    \label{eq:slope-proof}
        \partial\Omega\ni x\mapsto q_{u_{\nu,\varphi}}(x)\in (-\infty,\infty]
    \end{equation}
    is well-defined, when $\varphi\in C^1(\partial\Omega)$. Hence, all objects in \eqref{eq:generalized-global-Cauchy-data} are well-defined. Moreover, Proposition~\ref{lem:measurable} ensures that the map \eqref{eq:slope-proof} is Borel measurable, which completes the proof.
\end{proof}

\begin{definition}[Generalized Cauchy data set for $L^1$ sources]
    Let $\Omega\subset\mathbb R^n$ be a bounded strictly convex domain with $C^1$ boundary. For any nonnegative $f\in L^1(\Omega)$, we set
    \[
        \mathcal{C}^{\rm glob}_{f}\vcentcolon = \mathcal{C}^{\rm glob}_{f\,dx},
    \]
    where $\mathcal{C}^{\rm glob}_{f\,dx}$ is the generalized Cauchy data set provided by Proposition~\ref{def:gen-cauchy-data}.
\end{definition}

\subsection{Comparison of generalized, interior, and classical Cauchy data sets}
\label{subsec:interior-normal-derivative}

In this section we prove two compatibility results for the generalized global Cauchy data set. In particular, we introduce classical Cauchy data sets and interior Cauchy data sets. 

\begin{definition}[Classical Cauchy data]
    Let $\Omega\subset\R^n$ be a bounded convex $C^1$ domain. Assume that $f\in C(\Omega)$ is nonnegative. We say that a tuple $(\varphi,\psi)\in C(\partial\Omega)\times C(\partial\Omega)$ is a \emph{classical Cauchy data pair} for the Monge--Ampère equation with source $f$ if $\varphi \in C^{1,1}(\partial\Omega)$ and $\psi=\partial_\nu u|_{\partial\Omega}$, where $u\in C^{1,1}(\overline{\Omega})\cap C^2(\Omega)$ is a convex pointwise solution of
     \begin{equation}
    \label{eq:classical-Cauchy}
    \begin{cases}
    \det D^2u=f,\quad&\text{in }\Omega,\\
     u=\varphi,\quad&\text{on }\partial\Omega.
    \end{cases}
    \end{equation}
\end{definition}

Our first result compares generalized Cauchy data to classical Cauchy data.

\begin{proposition}[Compatibility with classical Cauchy data]
\label{prop:classical-compatibility}
Let $\Omega\subset\mathbb R^n$ be a bounded convex domain with $C^1$ boundary. Let $\varphi\in C^1(\partial\Omega)$ and suppose that $f\in L^1(\Omega)\cap C(\Omega)$ is nonnegative. Assume that $u\in C^1(\overline\Omega)$ is convex and $\varphi=u|_{\partial\Omega}$. Then, for every $x\in\partial\Omega$,
\begin{equation}
\label{eq:classical-ray}
 q_u(x)=\partial_\nu u(x).
\end{equation}
If, in addition, $\varphi\in C^{1,1}(\partial\Omega)$ and $u$ is an Alexandrov solution of the Monge--Ampère problem 
\begin{equation}
\label{eq:classical-measure}
\begin{cases}
    \det D^2u=f,\quad&\text{in }\Omega,\\
 u=\varphi,\quad&\text{on }\partial\Omega
\end{cases}
\end{equation}
satisfying $u\in C^{1,1}(\overline{\Omega})\cap C^2(\Omega)$, then the generalized and classical Cauchy pairs coincide. 
\end{proposition}

\begin{proof}
    Since $u\in C^1(\overline{\Omega})$, Lemma~\ref{lemma:bdry-subdifferential} implies
    \[
        \bar{\partial}u(x)=\{\nabla u(x)+s\nu(x):s\geq 0\},\quad x\in\partial\Omega.
    \]
    Hence the decomposition $\nabla u(x)=\nabla_\tau \varphi(x)+(\partial_\nu u(x))\nu(x)$ and \eqref{eq:defq} show formula \eqref{eq:classical-ray}. 

    If, in addition, the Alexandrov solution $u$ belongs to $C^{1,1}(\overline{\Omega})\cap C^2(\Omega)$, then Remark~\ref{rem:Alexandrov-Classica-sol} yields that $u$ is a classical solution of \eqref{eq:classical-measure}. Using the identity \eqref{eq:classical-ray}, we conclude that a generalized Cauchy pair is also a classical Cauchy pair. Conversely, by Remark~\ref{rem:Alexandrov-Classica-sol} every classical solution $u\in C^{1,1}(\overline{\Omega})\cap C^2(\Omega)$ is an Alexandrov solution and hence, once again, \eqref{eq:classical-ray} implies that every classical Cauchy pair is a generalized Cauchy pair.
\end{proof}

Next we introduce the interior Cauchy data.

\begin{definition}[Interior Cauchy data]
\label{def:Interior-cauchy}
    Let $\Omega\subset\R^n$ be a bounded uniformly convex $C^{1,1}$ domain (see Definition~\ref{def:uniformly-convex-domains}). Assume that $f\in L^\infty (\Omega)$ is nonnegative. We say that a tuple $(\varphi,\psi)\in C(\partial\Omega)\times \mathcal{B}(\partial\Omega)$ is an \emph{interior Cauchy data pair} for the Monge--Ampère equation with source $f$ if $\varphi \in C^{1,1}(\partial\Omega)$ and 
    \begin{equation}
    \label{eq:interior-normal}
        \psi(x)=\partial^{-}_\nu u(x)\vcentcolon =\lim_{r\downarrow 0}\frac{u(x)-u(x-r\nu(x))}{r},\quad x\in\partial\Omega,
    \end{equation}
    where $u\in C^{0}(\overline{\Omega})$ is an Alexandrov solution of
     \begin{equation}
    \label{eq:interior-Cauchy}
    \begin{cases}
    \det D^2u=f,\quad&\text{in }\Omega,\\
     u=\varphi,\quad&\text{on }\partial\Omega.
    \end{cases}
    \end{equation}
\end{definition}

\begin{remark}[One-sided limits of convex functions]
\label{rem:one-sided-der}
For any proper convex function
$g:\mathbb R^n\to(-\infty,\infty]$, one may define the \emph{one-sided directional derivative} of $g$ at $x$ into direction $y$ by
\begin{equation}
\label{eq:one-sided-limit}
g'(x;y)
\vcentcolon=
\lim_{\lambda\downarrow0}
\frac{g(x+\lambda y)-g(x)}{\lambda}\in [-\infty,\infty].
\end{equation}
By \cite[Theorem~23.1]{Rockafellar}, this limit exists at any $x$, where $g$ is finite. Indeed, the difference quotient is a nondeceasing function of $\lambda>0$, so that
\begin{equation}
\label{eq:one-sided-der-inf}
    g'(x;y)=\inf_{\lambda>0}\frac{g(x+\lambda y)-g(x)}{\lambda}.
\end{equation}
Now, let $\bar u:\mathbb R^n\to(-\infty,\infty]$ be the canonical lower-semicontinuous convex extension of a convex function $u\colon \overline{\Omega}\to \R$ (see~\eqref{eq:extension-convex}). If $\Omega$ is a $C^1$ domain, then we have $x-r\nu(x)\in\Omega$ for all sufficiently small $r>0$ and hence
\[
\bar u'(x;-\nu(x))
=
\lim_{r\downarrow0}
\frac{\bar u(x-r\nu(x))-\bar u(x)}{r}
=
-\partial_\nu^-u(x).
\]
Thus $-\partial_\nu^-u(x)$ is precisely the one-sided directional derivative of $\bar u$ at $x$ in the inward direction $-\nu(x)$. Since $u\in C(\Omega)$ and the terms in the infimum are real-valued an nondecreasing, it follows that $\bar u'(x;-\nu(x))\in [-\infty,\infty)$ and hence $\partial_\nu^-u(x)\in (-\infty,\infty]$. 
\end{remark}

The following observation, together with Lemma~\ref{lem:measurable}, shows that Lipschitz continuous Alexandrov solutions yield well-defined interior Cauchy data pairs.

\begin{lemma}[Boundary slope as interior normal derivative]
\label{lem:radial-slope}
Let $\Omega\subset\mathbb R^n$ be a bounded strictly convex $C^1$ domain. Assume that $u\in C^{0,1}(\overline{\Omega})$ is a convex function and suppose that
\[
    u|_{\partial\Omega}=\varphi
    \qquad\text{with }\varphi\in C^1(\partial\Omega).
\]
    Then, for every $x\in\partial\Omega$, the boundary subdifferential $\bar \partial u(x)$ is nonempty and the \emph{interior normal derivative} of $u$ at $x$ is given by
    \begin{equation}
\label{eq:radial-slope}
    \partial^{-}_\nu u(x)\vcentcolon =
\lim_{r\downarrow0}
\frac{u(x)-u(x-r\nu(x))}{r}=q_u(x).
\end{equation}
    Moreover, the limit in \eqref{eq:radial-slope} is finite and there holds
\begin{equation}
\label{eq:bound-interior-normal}
|\partial^{-}_\nu u (x)|
\leq
    [u]_{C^{0,1}(\overline{\Omega})}
\qquad
\text{for every }x\in\partial\Omega.
\end{equation}
    In particular, generalized and interior Cauchy data pairs coincide, whenever the unique Alexandrov solution is Lipschitz on $\overline{\Omega}$.
\end{lemma}

\begin{proof}
Fix $x\in\partial\Omega$. Since $\Omega$ is of class $C^1$ and $\nu$ is the outward pointing normal field, we have 
\[
x-r\nu(x)\in\Omega
\qquad
\text{for }0<r<r_0,
\]
when $r_0>0$ is sufficiently small. Define 
\[
g\colon [0,r_0)\to\R,\quad g(r)\vcentcolon = u(x-r\nu(x)).
\]
Since $u$ is convex, $g$ is a convex function of one variable on $[0,r_0)$. In particular, the difference quotient $r\mapsto \frac{g(r)-g(0)}{r}$ is monotonically nondecreasing on $(0,r_0)$. Indeed, if $0<r_1<r_2<r_0$, then
\[
    g(r_1)=g(r_2 t)=g((1-t)0+tr_2)\leq (1-t)g(0)+tg(r_2),
\]
where we set $t=r_1/r_2\in (0,1)$. Thus, $g(0)=u(x)\in \R$ implies
\[
    g(r_1)-g(0)\leq t(g(r_2)-g(0))=\frac{r_1}{r_2}(g(r_2)-g(0)).
\]
This establishes the claim.

Now, since the difference quotient of $g$ is nondecreasing on $(0,r_0)$, its right-sided derivative $g'_+(0)$ exists and is given by
\[
    g'_+(0)\vcentcolon = \lim_{r\downarrow 0}\frac{g(r)-g(0)}{r}=\inf_{0<r<r_0}\frac{g(r)-g(0)}{r}\in [-\infty,\infty).
\]
Since $u$ is Lipschitz on $\overline{\Omega}$, the function $g$ is Lipschitz on $[0,r_0]$ and hence the right-sided derivative $g'_+(0)$ is finite with
\[
|g'_+ (0)|\leq [g]_{C^{0,1}([0,r_0])}\leq [u]_{C^{0,1}(\overline{\Omega})}.
\]
By definition of $g$, we conclude that $\partial^{-}_{\nu}u(x)$ exists, it is given by $-g'_+ (0)$, and the estimate \eqref{eq:bound-interior-normal} holds.

We next show that $\partial^{-}_{\nu}u(x)=q_u(x)$. Let
\[
    p=\nabla_\tau\varphi(x)+a\nu(x)\in\bar\partial u(x).
\]
Applying the subgradient inequality at $x-r\nu(x)$ gives
\[
u(x-r\nu(x))
\geq
u(x)-ra.
\]
Hence
\[
\frac{u(x)-u(x-r\nu(x))}{r}\leq a.
\]
Passing to the limit as $r\downarrow0$ and then taking the infimum over all admissible $a$ yields
\begin{equation}
\label{eq:ell-lower-q}
    \partial^{-}_{\nu}u(x)\leq q_u(x);
\end{equation}
see Proposition~\ref{def:slope}.
If there is no such subgradient $p$, then $q_u(x)=\infty$ and \eqref{eq:ell-lower-q} holds trivially.

It remains to prove the reverse inequality and, at the same time, the nonemptiness of $\bar\partial u(x)$. Let $r_k\downarrow0$ and set
\[
x_k\vcentcolon = x-r_k\nu(x)\in\Omega.
\]
Since $u\in C(\Omega)$ and $x_k\in \Omega$, we may choose $p_k\in\partial u(x_k)$. By the subgradient inequality and the Lipschitz continuity of $u$, we have
\begin{equation}
\label{eq:boundedness-p-lipschitz}
|p_k|\leq [u]_{C^{0,1}(\overline{\Omega})}.
\end{equation}
Indeed, if $p_k\neq 0$, then apply the subgradient inequality at $x_k$ with $z\vcentcolon =x_k+r\frac{p_k}{|p_k|}$, where $r>0$ is chosen sufficiently small such that $z\in\Omega$. We have
\[
    u(x_k+r p_k/|p_k|)-u(x_k)\geq r|p_k|.
\]
Hence, the Lipschitz continuity of $u$ yields \eqref{eq:boundedness-p-lipschitz}.

Therefore after passing to a subsequence, we may assume that $p_k\to p\in\mathbb R^n$. Since $u=\bar u$ in $\overline{\Omega}$ and $u$ is Lipschitz continuous up to the boundary, the subgradient inequality
\begin{equation}
\label{eq:subgrad-ineq-interior}
u(z)\geq u(x_k)+p_k\cdot(z-x_k)
\end{equation}
holds for all $z\in \overline{\Omega}$ and thus passing to the limit $k\to\infty$ shows that $p\in\bar\partial u(x)$. Thus, $\bar\partial u(x)$ is indeed nonempty.

Moreover, restriction of the preceding supporting inequality to the normal segment $z=x-t\nu(x)\in \Omega$, $t>0$, gives
\[
    g(t)\geq g(r_k)-p_k\cdot \nu(x)(t-r_k).
\]
Thus, $-p_k\cdot \nu(x)\in \partial g(r_k)$. For a one-dimensional convex function, if $r_k\downarrow 0$ and $\xi_k\in \partial g(r_k)$, then $\xi_k\to g'_+(0)$. Indeed, the subgradient inequality at $r_k$ with $t=0$, $t=s>0$, and subgradient $\xi_k$ gives
\[
    \frac{g(r_k)-g(0)}{r_k}\leq \xi_k\leq \frac{g(s)-g(r_k)}{s-r_k}.
\]
Hence, by $g\in C([0,r_0))$, passing to the limit $k\to\infty$ and then $s\downarrow 0$ yields $\xi_k\to g'_+(0)$.

Therefore, $p_k\to p$ yields
\[
-p\cdot\nu(x)=g'_+(0),
\qquad\text{and hence}\qquad
p\cdot\nu(x)=\partial^-_\nu u(x).
\]
Applying the subgradient inequality
\[
    u(z)\geq u(x)+p\cdot (z-x),\quad z\in \overline{\Omega}
\]
with $z=\gamma(t)$, where $\gamma\colon (-\eps,\eps)\to \partial \Omega$ is a $C^1$ boundary curves with $\gamma(0)=x$ and $\gamma'(0)=\tau_0\in T_x\partial\Omega$, we get
\[
    p\cdot \tau_0=\nabla_\tau\varphi(x)\cdot \tau_0.
\]
Since this holds for all tangential vectors $\tau_0$, we find
\[
p=\nabla_\tau\varphi(x)+a\nu(x)\in \bar \partial u(x)
\]
for some $a\in\R$. Since $p\cdot \nu(x)=\partial^-_\nu u(x)$, we have 
\[
    p=\nabla_\tau\varphi(x)+(\partial^-_\nu u(x))\nu(x)\in \bar \partial u(x).
\]
From the definition of  $q_u$ (see~\eqref{eq:defq}), we may conclude that 
\[
q_u(x)\leq \partial^-_\nu u(x).
\]
Together with \eqref{eq:ell-lower-q}, this proves \eqref{eq:radial-slope}.

Moreover, using $p\cdot\nu(x)=\partial^-_\nu u(x)$, $p_k\to p$, and \eqref{eq:boundedness-p-lipschitz} we may calculate
\[
    |\partial^-_\nu u(x)|=|p\cdot \nu(x)|=\lim_{k\to\infty}|p_k\cdot \nu(x)|\leq [u]_{C^{0,1}(\overline{\Omega})},
\]
which establishes \eqref{eq:bound-interior-normal}.

The last part of the lemma is a direct consequence of \eqref{eq:radial-slope} and Definition~\ref{def:Interior-cauchy}.
\end{proof}

Next we record the barrier statement, which will also be needed in the proof of
Corollary~\ref{cor:Linfty}.

\begin{proposition}[Boundary Lipschitz bound]
\label{prop:boundary-Lipschitz}
Let $\Omega \subset\R^n$ be a bounded uniformly convex $C^{1,1}$ domain, let $\varphi\in C^{1,1}(\partial\Omega)$, and suppose that $f\in L^\infty (\Omega)$ satisfies
\begin{equation}
\label{eq:upper-bound-f}
 0\leq f\leq \lambda\quad\text{a.e. in }\Omega.
\end{equation}
Then the unique Alexandrov solution $u$ of 
  \begin{equation}
    \label{eq:bdry-Lipschitz}
    \begin{cases}
    \det D^2u=f,\quad&\text{in }\Omega,\\
     u=\varphi,\quad&\text{on }\partial\Omega,
    \end{cases}
    \end{equation}
    provided by Proposition~\ref{prop:forward} \ref{forward-existence}, is Lipschitz on $\overline\Omega$. Moreover, its extremal boundary slope $q_u\colon \partial\Omega\to\R$ is a bounded Borel measurable function.
\end{proposition}

\begin{proof}
    Since $\Omega$ is a bounded uniformly convex $C^{1,1}$ domain, Lemma~\ref{lem:defining-function} implies that there exist $\rho\in C^{1,1}(\mathbb R^n)$ and $c_0>0$ such that
\[
\Omega=\{\rho<0\},\qquad
\partial\Omega=\{\rho=0\},\qquad
D^2\rho\geq c_0I
\quad\text{a.e. in }\mathbb R^n.
\]
Using boundary flattening charts and a finite partition of unity, we may construct an extension $\widetilde\varphi\in C^{1,1}(\mathbb R^n)$ of $\varphi\in C^{1,1}(\partial\Omega)$ (see, e.g., \cite[p.~94]{GT}). Since $\Omega$ is bounded, \cite[Section~4.2.3]{evans-mt} implies that $\widetilde \varphi \in W^{2,\infty}(\Omega)$. By Rademacher's theorem the almost-everywhere Hessian $D^2 \widetilde \varphi$ represents its weak Hessian. Therefore, we may estimate
\[
D^2\widetilde\varphi\geq -M_\varphi I
\quad\text{a.e. in }\Omega,
\]
where we set $M_\varphi\vcentcolon = \|D^2 \widetilde \varphi\|_{L^{\infty}(\Omega)}$. Set
\begin{equation}
\label{eq:lambda0}
\lambda_0\vcentcolon = \max\{1,\lambda^{1/n}\}
\end{equation}
and choose $A>0$ so that
\[
Ac_0-M_\varphi\geq\lambda_0.
\]
Then
\[
w\vcentcolon = \widetilde\varphi+A\rho
\]
satisfies
\begin{equation}
\label{eq:properties-w-lipschitz}
w=\varphi\quad\text{on }\partial\Omega,
\qquad
D^2w\geq\lambda_0I
\quad\text{a.e. in }\Omega.
\end{equation}
In particular, $w$ is strongly convex, that is, 
\[
    w(y)\geq w(x)+\nabla w(x)\cdot (y-x)+\frac{\lambda_0}{2}|x-y|^2
\]
for all $x,y\in \Omega$. If $w$ is of class $C^2$ this is a classical result and in our more general setup it can be justified by mollification. This argument is carried out in detail at the beginning of the proof of Lemma~\ref{lem:defining-function}.

By Lemma~\ref{lem:MA-measure} \ref{C11-case}, the Monge--Amp\`ere measure of $w$ is $\mu_w=\det D^2w\,dx$. Moreover, \eqref{eq:properties-w-lipschitz}, monotonicity of the determinant, \eqref{eq:lambda0}, \eqref{eq:upper-bound-f}, and the fact that $u$ is the Alexandrov solution of \eqref{eq:bdry-Lipschitz} imply
\[
\mu_w=\det D^2w\,dx
\geq\lambda_0^n\,dx
\geq\lambda\,dx
\geq f\,dx=\mu_u
\]
as Borel measures. Because $w=\varphi=u$ on $\partial\Omega$, the comparison
principle (Proposition~\ref{prop:Comparison}) gives
\[
w\leq u\quad\text{in }\Omega.
\]

Fix $x\in\partial\Omega$. Every relative supporting plane of $w$
at $x$ also supports $u$ there, since $w\leq u$ in
$\overline\Omega$ and $w(x)=u(x)$. Hence
\[
\bar\partial w(x)\subset \bar\partial u(x)
\]
and therefore \eqref{eq:defq} yields
\[
q_u(x)\leq q_w(x).
\]
Since $w\in C^{1,1}(\overline{\Omega})$ is convex, Proposition~\ref{prop:classical-compatibility} demonstrates that
\[
q_w(x)=\partial_\nu w(x)\in L^{\infty}(\partial\Omega).
\]
Thus $q_w$, and hence $q_u$ from above, is uniformly bounded on $\partial\Omega$.

The $C^{1,1}$ tubular-neighborhood theorem and compactness of $\partial\Omega$ give $r_0>0$ such that
\[
    z_x\vcentcolon = x-r_0\nu(x)\in\Omega \qquad\text{for every }x\in\partial\Omega.
\]
Let
\[
p=\nabla_\tau\varphi(x)+a\nu(x)\in\bar \partial u(x).
\]
The supporting inequality at $z_x$ gives
\[
u(z_x)\geq\varphi(x)-r_0a,
\]
and hence
\[
a\geq\frac{\varphi(x)-u(z_x)}{r_0}.
\]
Since a convex function on a compact convex set cannot exceed the
maximum of its boundary values,
\[
u(z_x)\leq\max_{\partial\Omega}\varphi.
\]
It follows that
\[
a\geq
\frac{\min_{\partial\Omega}\varphi
      -\max_{\partial\Omega}\varphi}{r_0}.
\]
Recalling the definition \eqref{eq:defq}, we see that taking the infimum over all admissible $a$ gives a uniform lower bound for $q_u$, namely,
\[
    q_u(x)\geq \frac{\min_{\partial\Omega}\varphi
      -\max_{\partial\Omega}\varphi}{r_0}.
\]
We have therefore proved $q_u\in L^{\infty}(\partial\Omega)$ with
\[
\sup_{x\in\partial\Omega}|q_u(x)|\leq \max\left(\frac{\max_{\partial\Omega}\varphi-\min_{\partial\Omega}\varphi}{r_0},\|\partial_\nu w\
\|_{L^{\infty}(\partial\Omega)}\right).
\]
Since Lemma~\ref{lem:measurable} guarantees the Borel measurability of $q_u$, we have $q_u\in \mathcal{B}_b(\partial\Omega)$.

It remains to prove the global Lipschitz estimate. By Lemma~\ref{lem:ray}, the infimum defining $q_u(x)$ is attained. Thus
\[
p_x\vcentcolon = \nabla_\tau\varphi(x)+q_u(x)\nu(x)
\in\bar\partial u(x),
\qquad x\in\partial\Omega.
\]
Since $\varphi\in C^{1,1}(\partial\Omega)$, $q_u\in L^{\infty}(\partial\Omega)$, and $\nu(x)\in \mathbb S^{n-1}$, there is $M>0$ such that
\[
|p_x|\leq M
\qquad\text{for every }x\in\partial\Omega.
\]
Now, let $x,y\in \overline{\Omega}$,
$x\neq y$, and let the line through $x$ and $y$ intersect $\overline{\Omega}$ in the chord $[a,b]$ for some boundary points $a,b\in \partial\Omega$. Let us set
\[
    \ell\vcentcolon = |b-a|,\quad  e\vcentcolon =\frac{b-a}{\ell}\in \mathbb S^{n-1}
\]
and define the function
\[
    g\colon [0,\ell]\to\R,\quad g(s)\vcentcolon = u(a+se).
\]
Since $u\in C(\overline{\Omega})$ is convex, the function $g$ is convex as well. Convexity of $g$ and the subgradient inequalities of $u$ associated with $p_a$ and $p_b$ imply
\begin{equation}
\label{eq:two-sided-bound}
p_a\cdot e
\leq\frac{g(t)-g(s)}{t-s}
\leq p_b\cdot e
\qquad
(0\leq s<t\leq|b-a|).
\end{equation}
Indeed, using the subgradient inequality and $b=a+\ell \, e$, we get
\begin{equation}
\label{eq:subgrad-1}
    g(t)=u(a+te)\geq u(a)+t\, p_a\cdot e=g(0)+t\,p_a\cdot e
\end{equation}
and
\begin{equation}
\label{eq:subgrad-2}
    g(s)=u(a+s e)\geq u(b)+(s-\ell)\,p_b \cdot e=g(\ell)+(s-\ell)\,p_b \cdot e.
\end{equation}
Moreover, the convexity gives
\[
\begin{split}
     g(s)&=g\left((1-s/t)0+(s/t)t\right)\\
     &\leq (1-s/t)g(0)+(s/t)g(t)\\
     &=(1-s/t)(g(0)-g(t))+g(t)
\end{split}
\]
and hence
\begin{equation}
\label{eq:cons-conv}
    \frac{g(t)-g(s)}{t-s}\geq \frac{g(t)-g(0)}{t}.
\end{equation}
Combined with \eqref{eq:subgrad-1}, we get
\[
    \frac{g(t)-g(s)}{t-s}\geq \frac{g(t)-g(0)}{t}\geq p_a\cdot e.
\]
Similarly, since
\[
    t=\frac{\ell-t}{\ell-s}s+\frac{t-s}{\ell-s}\ell,
\]
convexity of $g$ implies
\[
    g(t)\leq \frac{\ell-t}{\ell-s}g(s)+\frac{t-s}{\ell-s}g(\ell).
\]
Therefore,
\[
   \frac{g(t)-g(s)}{t-s}\leq \frac{g(\ell)-g(s)}{\ell-s}\leq p_b\cdot e. 
\]
This completes the verification of \eqref{eq:two-sided-bound}.

Since $|p_a|,|p_b|\leq M$ and $e\in \mathbb S^{n-1}$, we obtain
\[
|g(t)-g(s)|\leq M|t-s|.
\]
Applying this to the parameters corresponding to $x$ and $y$
gives
\[
|u(x)-u(y)|\leq M|x-y|.
\]
Thus $u\in C^{0,1}(\overline\Omega)$, and the proposition follows.
\end{proof}

The above result allows us to introduce rigorously the Dirichlet-to-interior-Neumann map.

\begin{proposition}[Dirichlet-to-interior-Neumann map]
\label{prop:DiNmap}
    Let $\Omega$ be uniformly convex with $C^{1,1}$ boundary and suppose that $f\in L^{\infty}(\Omega)$ satisfies $0\leq f\leq C$. Then, the \emph{global Dirichlet-to-interior-Neumann map} $\Lambda_f^-
\colon
C^{1,1}(\partial\Omega)
\to
\mathcal B_b(\partial\Omega)$, given by
\begin{equation}
\label{eq:def-DiN-one-sided-slope}
\Lambda_f^-(\varphi)
\vcentcolon = \partial_\nu^-u_{f,\varphi}\vcentcolon =\lim_{r\downarrow0}
\frac{u_{f,\varphi}(x)-u_{f,\varphi}(x-r\nu(x))}{r},
\end{equation}
is well-defined. Moreover, $\Lambda_f^-$ is completely determined by the generalized Cauchy data set $\mathcal{C}^{\rm glob}_f$.
\end{proposition}

\begin{proof}
    First, the Alexandrov solution $u_{f,\varphi}$ is well-defined by Proposition~\ref{prop:forward} \ref{forward-existence}. Then Proposition~\ref{prop:boundary-Lipschitz}  yields that $u_{f,\varphi}$ is Lipschitz continuous up to the boundary. Therefore, we can apply Lemma~\ref{lem:radial-slope} to deduce that the right hand side of \eqref{eq:def-DiN-one-sided-slope} is well-defined and coincides with the extremal boundary slope $q_{u_{f,\varphi}}$. Combined with Lemma~\ref{lem:measurable}, this shows that $\partial_\nu^-u_{f,\varphi}$ is a bounded Borel measurable function.

    The last part of the assertion follows from the formulas \eqref{eq:radial-slope} and \eqref{eq:def-DiN-one-sided-slope}. This completes the proof.
\end{proof}

\section{Determination of the interior subgradient image}
\label{sec:det-subgradient-image}

The goal of this section is to show that the extremal boundary slopes $q_u$ determine the subgradient images (Proposition~\ref{prop:D-characterization} and Corollary~\ref{cor:D-equality}). 

\subsection{Subdifferential images}

We start with the following definition, which introduces the interior and boundary parts of the subdifferential image $\bar\partial u(\overline{\Omega})$.

\begin{definition}[Subdifferential images]
\label{def:subdiff-image}
    Let $\Omega\subset\mathbb R^n$ be a bounded convex $C^1$ domain, let $\varphi\in C^1(\partial\Omega)$, and suppose that $u\in C(\overline\Omega)$ is convex and satisfies $u|_{\partial\Omega}=\varphi$. Moreover, let $q_u$ be its extremal boundary slope as defined in Proposition~\ref{def:slope}. Then we denote the \emph{subdifferential image} by
    \begin{equation}
    \label{eq:subdiff-image}
         \mathrm D_u\vcentcolon =\partial u(\Omega)\subset\R^n,
    \end{equation}
    the \emph{boundary subdifferential image} by 
    \begin{equation}
    \label{eq:bdry-subdiff-image}
         \mathrm B_u\vcentcolon =\bar\partial u(\partial \Omega)\subset\R^n,
    \end{equation}
    and the \emph{boundary slope image} by
    \begin{equation}
    \label{eq:S}
     \mathrm S_u\vcentcolon =\{\nabla_\tau\varphi(x)+q_u(x)\nu(x):
             x\in\partial\Omega,\ q_u(x)<\infty\}\subset \mathrm B_u.
    \end{equation}
    Moreover, for every $\psi\in C(\partial\Omega)$ and $p\in \R^n$, we denote by 
    \[
         \mathcal{X}_\psi (p)\vcentcolon = \argmax_{x\in\partial\Omega}
     \bigl(p\cdot x-\psi(x)\bigr)\subset \partial\Omega
    \]
    the set of \emph{boundary Legendre maximizers}.
\end{definition}

\begin{remark}[Boundary Legendre minimizers]
    We note that the continuiuty of $\psi$ and the compactness of $\partial\Omega$ imply that $\mathcal{X}_\psi (p)$ is always nonempty and compact.
\end{remark}

In view of Lemma~\ref{lem:ray}, whenever $x\in\partial\Omega$ satisfies $q_u(x)<\infty$, the boundary subdifferential $\bar\partial u(x)$ is the outward normal ray starting from the corresponding point of $\mathrm S_u$:
\[
\bar\partial u(x)
=
p_x+[0,\infty)\nu(x),
\qquad
p_x\vcentcolon = \nabla_\tau\varphi(x)+q_u(x)\nu(x)\in \mathrm S_u.
\]
On the other hand, when $q_u(x)=\infty$, then $\bar \partial u(x)=\varnothing$. Consequently,
\[
\mathrm B_u
=
\bar\partial u(\{q_u<\infty\})=\bigcup_{x\in\{q_u<\infty\}}(p_x+[0,\infty)\nu(x)).
\]
Moreover, Lemma~\ref{lemma:bdry-subdifferential} gives
\begin{equation}
\label{eq:equal-interior}
    \mathrm D_u=\bar\partial u(\Omega).
\end{equation}

Next, we recall that every interior subgradient belongs to the closed convex hull\footnote{Recall that for a set $A\subset\R^n$ its closed convex hull, denoted as $\overline{\operatorname{conv}}(A)$ is the smallest closed and convex set containing $A$.} of the gradients at differentiability points. This follows from the Krein--Milman type theorem \cite[Theorem~A.9]{Figalli} and the fact that every exposed point $p\in\partial u(x)$, $x\in\Omega$, can be approximated by $\nabla u(x_k)$ along a sequence of differentiability points $x_k\to x$; see \cite[Proposition~A.26]{Figalli}.
\begin{lemma}[{\cite[Corollary~A.27]{Figalli}}]
    Let $\Omega\subset\R^n$ be an open convex set and assume that $u\colon \Omega\to \R$ is convex. Let $\mathrm d_u$ denote the set of differentiability points of $u$. Then $D_u=\partial u(\Omega)$ is contained in the closed convex hull of $\nabla u(\mathrm d_u)$.
\end{lemma}
Hence, by combining the above observations, we have
\begin{equation}
\label{eq:whole-subdiff}
    \bar\partial u(\overline{\Omega})=\mathrm{D}_u\cup \mathrm{B}_u\subset \overline{\operatorname{conv}}(\nabla u(\mathrm{d}_u))\cup \bigcup_{x\in\{q_u<\infty\}}(p_x+[0,\infty)\nu(x)) .
\end{equation}

\subsection{Boundary slope image is Lebesgue null}

Our next result shows that the boundary slope image has vanishing Lebesgue measure.

\begin{lemma}[Nullity of the boundary slope image]
\label{lem:S-null}
Let $\Omega\subset\mathbb R^n$ be a bounded strictly convex $C^{1,1}$ domain and $\varphi\in C^{1,1}(\partial\Omega)$. Suppose that $u\in C(\overline\Omega)$ is convex and satisfies $u|_{\partial\Omega}=\varphi$. Then
\[
  \cL^n(\mathrm S_u)=0.
\]
\end{lemma}

\begin{proof}
The result follows by standard measure theory arguments. Recall that by Lemma~\ref{lem:measurable} the extremal boundary slope $q_u\colon \partial\Omega\to (-\infty,\infty]$ is Borel measurable. Let us denote the graph of $q_u$ by
\[
    \Gamma\vcentcolon = \operatorname{graph}(q_u)=\{(x,q_u(x)):x\in \partial\Omega\}\subset \partial\Omega\times (-\infty,\infty].
\]
For each $k\in\mathbb N$, define the Borel sets
\[
 E_k\vcentcolon =\{x\in\partial\Omega:|q_u(x)|\leq k\}\subset\R^n,
\]
which cover $\{q_u<\infty\}$ as $k$ runs through the whole natural numbers. Moreover, we write $\Gamma_k$ for the restriction of $\Gamma$ to $E_k$, that is, it is the graph of the Borel function $q_u^k\vcentcolon =q_u|_{E_k}$. More explicitly,
\[
 \Gamma_k\vcentcolon = \operatorname{graph}(q_u^k)=\{(x,q_u(x)):x\in E_k\}
       \subset\partial\Omega\times [-k,k]\subset \R^{n+1}.
\]
Since $q_u^k\colon E_k\to \R$ is Borel measurable, the sets $\Gamma_k$ are Borel as well. For any $k\in\N$, we define the Lipschitz continuous function 
\[
    \Psi_k\colon \partial\Omega\times [-k,k]\to \R^n,\quad \Psi_k(x,a)\vcentcolon = \nabla_\tau\varphi(x)+a\nu(x).
\]
Note that Lipschitz continuity of $\Psi_k$ follows from the fact that $\varphi\in C^{1,1}(\partial\Omega)$ and $\partial\Omega\in C^{1,1}$, which implies $\nu\in C^{0,1}$. More precisely,
\[
    |\Psi_k(x,a)-\Psi_k(y,b)|\leq ([\nabla_\tau \varphi]_{C^{0,1}(\partial\Omega)}+k[\nu]_{C^{0,1}(\partial\Omega)})|x-y|+|a-b|.
\]
Then, we may decompose $\mathrm S_u$ as
\[
 \mathrm S_u=\bigcup_{k=1}^\infty\Psi_k(\Gamma_k).
\]
By $\sigma$-subadditivity, we have
\begin{equation}
\label{eq:subadditivity-estimate}
    \mathcal{L}^n(\mathrm S_u)\leq \sum_{k\in\N}\mathcal{L}^n(\Psi_k (\Gamma_k)).
\end{equation}
We claim that $\mathcal H^n(\Gamma_k)=0$. To see this, cover $\partial\Omega$ by finitely many $C^{1,1}$ boundary charts $X_j:V_j\subset\mathbb R^{n-1}\to\partial\Omega$, $1\leq j\leq N$. For each $1\leq j\leq N$ and $k\in\N$, set
\[
    V_{j,k}\vcentcolon = X_j^{-1}(E_k),\quad X_{j,k}\vcentcolon = X_j|_{V_{j,k}}
\]
and
\[
G_{j,k}
\vcentcolon =\operatorname{graph}(q_u^k\circ X_{j,k})=
\left\{
(z,a):
z\in V_{j,k},\ a=q_u(X_j(z))
\right\}\subset \R^n.
\]
Since $E_k$ is a Borel set, $V_{j,k}$ and $G_{j,k}$ are Borel too. The vertical section $(G_{j,k})_z$ is always a singleton, namely $\{q_u (X_j(z))\}$, and hence $\mathcal{L}^1((G_{j,k})_z)=0$ for every $z\in V_{j,k}$. Since $G_{j,k}$ is Borel, Tonelli's theorem yields
\[
\begin{aligned}
\mathcal L^n(G_{j,k})
&=
\int_{V_{j,k}}\int_{-k}^{k}
\chi_{G_{j,k}}(z,a)\,da\,dz\\
&=
\int_{V_{j,k}}\mathcal L^1((G_{j,k})_z)\,dz
=0.
\end{aligned}
\]
We have
\[
\begin{split}
     \Gamma_k&=\bigcup_{j=1}^N\{(x,q_u(x)):x\in E_k\cap X_j(V_j)\}\\
     &=\bigcup_{j=1}^N\{(X_j(z),q_u(X_j(z)):z\in X_j^{-1}(E_k)\cap V_j \}\\
     &=\bigcup_{j=1}^N\{(X_j(z),q_u(X_j(z)):z\in V_{j,k} \}\\
     &=\bigcup_{j=1}^N\Phi_j(G_{j,k}),
\end{split}
\]
where $\Phi_j\colon V_j\times \R\to\R^{n+1}$ is the Lipschitz map $\Phi_j(z,a)\vcentcolon = (X_j(z),a)$. Hence, the Lipschitz estimate for Hausdorff measure \cite[Section~2.4.1]{evans-mt} therefore implies 
\[
    \mathcal H^n(\Gamma_k)\leq \sum_{j=1}^N \mathcal{H}^n(\Phi_j(G_{j,k}))\leq \sum_{j=1}^N [\Phi_j]^n_{C^{0,1}(V_j\times \R)}\mathcal{H}^n(G_{j,k})=0.
\]
Finally, since the maps $\Psi_k$ are Lipschitz continuous and $\mathcal{H}^n(\Gamma_k)=0$, the Lipschitz estimate for Hausdorff measure \cite[Section~2.4.1]{evans-mt} and $\mathcal{H}^n=\mathcal{L}^n$ on $\R^n$ imply that
\[
    \cL^n(\Psi_k(\Gamma_k))\leq [\Psi_k]^n_{C^{0,1}(\partial\Omega\times [-k,k])}\mathcal{H}^n(\Gamma_k)=0.
\]
Thus, \eqref{eq:subadditivity-estimate} yields $\mathcal{L}^n(\mathrm S_u)=0$.
\end{proof}

\subsection{Equality of subdifferential images}

Finally, we show that equal generalized Cauchy data lead to equal subdifferential images $\mathrm D_u$ in a measure theoretic sense. In the proof we will need a few facts on the Legendre transform \cite[p.~104]{Rockafellar}.

\begin{definition}[{Legendre transform}]
\label{def:Legendre-transform}
    Let $v\colon \R^n \to (-\infty,\infty]$ be a proper convex function. Then its \emph{Legendre transform} of $v$ (also called \emph{conjugate}) is the function $v^\star\colon \R^n\to (-\infty,\infty]$ defined by
    \begin{equation}
    \label{eq:Legendre-trafo}
        v^\star (p)\vcentcolon = \sup_{x\in\R^n}[p\cdot x-v(x)],\quad p\in\R^n.
    \end{equation}
    Similarly, if $u\colon \Omega\to \R$ is a convex function over a convex open set $\Omega\subset \R^n$, then we define its Legendre transform $u^\star\colon \R^n \to (-\infty,\infty]$ as
    \[
        u^\star (p)\vcentcolon = \sup_{x\in\Omega}[p\cdot x-u(x)],\quad p\in\R^n.
    \]
\end{definition}

The Legendre transform has the following useful properties.

\begin{proposition}[{Properties of Legendre transform}]
\label{prop:legendre-transform}
    Let $v\colon \R^n\to (-\infty,\infty]$ be a proper convex function.
    \begin{enumerate}[(i)]
        \item\label{Legendre-convex} The Legendre transform $v^\star\colon \R^n\to (-\infty,\infty]$ is a proper convex lower semicontinuous function. 
        \item\label{involution-legendre} The Legendre transform is an involution on the collection of proper convex lower semicontinuous functions, that is, $(v^\star)^\star=v$.
        \item\label{subddif-legendre} Assume that $v$ is lower semicontinuous. Then for all $x\in \R^n$ and $p\in\R^n$, the following conditions are equivalent:
        \begin{enumerate}[(a)]
            \item $p\in \partial v(x)$;
            \item $x\in \partial v^\star (p)$;
            \item $z\mapsto p\cdot z-v(z)$ achieves its supremum at $z=x$, or equivalently $v^\star (p)=p\cdot x-v(x)$;
            \item $q\mapsto q\cdot x-v^\star (q)$ achieves its supremum at $q=p$, or equivalently $(v^\star)^\star (x)=p\cdot x-v^\star (p)$;
            \item $v(x)+v^\star (p)=p\cdot x$.
        \end{enumerate}
    \end{enumerate}
\end{proposition}

\begin{proof}
    The properties \ref{Legendre-convex}--\ref{involution-legendre} are reformulations of \cite[Theorem~12.2]{Rockafellar}. The equivalent conditions \ref{subddif-legendre} can be found in \cite[Theorem~23.5]{Rockafellar}.
\end{proof}

\begin{proposition}[Boundary characterization of $\mathrm D_u$]
\label{prop:D-characterization}
Let $\Omega\subset\mathbb R^n$ be a bounded strictly convex $C^{1,1}$ domain and $\varphi\in C^{1,1}(\partial\Omega)$. Suppose that $u\in C(\overline\Omega)$ is convex and satisfies $u|_{\partial\Omega}=\varphi$. For every $p\notin \mathrm S_u$, we have the following characterization:
\begin{equation}
\label{eq:D-characterization}
 p\notin \mathrm D_u
 \,\Leftrightarrow\,
 \exists x\in \mathcal{X}_\varphi(p):\,
 q_u(x)<\infty\text{ and }p\cdot\nu(x)>q_u(x).
\end{equation}
Moreover, the pair $(\varphi,q_u)$ determines $\mathrm D_u$ up to a Lebesgue null set.
\end{proposition}

\begin{proof}
    Suppose first that $x\in \mathcal{X}_\varphi(p)$ and
    \begin{equation}
    \label{eq:pos-cond}
      q_u(x)<\infty \quad\text{and}\quad p\cdot\nu(x)>q_u(x).
    \end{equation}
By definition of the set $\mathcal{X}_\varphi (p)$, the boundary function $\Phi_p\vcentcolon \partial\Omega\to\R$.
$\Phi_p (y)\vcentcolon =p\cdot y-\varphi(y)$ has a maximum at $y=x$. Thus, let $\gamma\colon (-\eps,\eps)\to \partial\Omega$ be a $C^1$ curve with $\gamma(0)=x$ and $\gamma'(0)=\tau_0\in T_x\partial\Omega$. Then $\Phi_p\circ \gamma\colon (-\eps,\eps)\to\R$ has a maximum at $t=0$ and therefore
\[
    0=\left.\frac{d}{dt}\right|_{t=0}\Phi_p(\gamma(t))=(p-\nabla_{\tau}\varphi(x))\cdot \tau_0.
\]
Since $\tau_0\in T_x\partial\Omega$ is arbitrary, this yields
\[
    p=\nabla_\tau \varphi(x)+a\nu(x)\quad\text{for some }a\in\R.
\]
The condition \eqref{eq:pos-cond} implies $a>q_u(x)$. By Lemma~\ref{lem:ray},
\[
 p_0\vcentcolon =\nabla_\tau\varphi(x)+q_u(x)\nu(x) \in \bar \partial u(x).
\]
Hence, for every $z\in\Omega$,
\begin{align*}
 &[u(z)-p\cdot z]-[u(x)-p\cdot x]\\
 &\qquad\geq(p_0-p)\cdot(z-x)\\
 &\qquad=(q_u(x)-a)\nu(x)\cdot(z-x)>0.
\end{align*}
Here we used that $a>q_u(x)$ and $\nu(x)\cdot(z-x)<0$ for every $z\in\Omega$, which follows from the strict convexity assumption on $\Omega$. Hence, 
\[
    p\cdot x-u(x) > p\cdot z -u(z)\quad\text{for all }z\in \Omega
\]
and therefore the maximum of $\overline{\Omega}\ni z\mapsto p\cdot z-u(z)\in\R$ not achieved at an interior point. Thus, Proposition~\ref{prop:legendre-transform} \ref{subddif-legendre} guarantees that $p\notin \mathrm D_u$.

Conversely, suppose $p\notin \mathrm D_u$. Since $u\in C(\overline\Omega)$ and $\overline\Omega$ is compact,
the function $z\mapsto p\cdot z-u(z)$ attains its maximum. By Proposition~\ref{prop:legendre-transform}
\ref{subddif-legendre} and $p\notin\mathrm D_u$, no maximizer lies in $\Omega$. Hence we may choose a maximizer $x\in\partial\Omega$. Then $x\in \mathcal{X}_\varphi(p)$ and, by Proposition~\ref{prop:legendre-transform} \ref{subddif-legendre}, $p\in \bar \partial u(x)$. Lemma~\ref{lem:ray} gives $q_u(x)<\infty$ and
$p\cdot\nu(x)\geq q_u(x)$. Equality would imply
\[
 p=\nabla_\tau\varphi(x)+q_u(x)\nu(x)\in \mathrm S_u,
\]
contrary to the assumption $p\notin \mathrm S_u$. Hence the inequality is strict.

Therefore we have established the equivalence \eqref{eq:D-characterization}.

Finally, we show that $(\varphi,q_u)$ determines $\mathrm D_u$ up to a Lebesgue null set. We may write
\[
    \mathrm D_u=(\mathrm D_u\cap \mathrm S_u)\cup (\mathrm D_u\setminus \mathrm S_u).
\]
The set $\mathrm D_u\cap \mathrm S_u$ has vanishing $n$-dimensional Lebesgue measure by Lemma~\ref{lem:S-null}. On the other hand, we have
\[
(\mathrm D_u\setminus \mathrm S_u)=\{p\in\R^n:\forall x\in \mathcal{X}_\varphi(p)\text{ one has }x\in \{q_u=\infty\}\text{ or }p\cdot \nu(x)< q_u(x)\}.
\]
Indeed, if $p\in\mathrm S_u$, then for some $x\in\partial\Omega$ with $q_u(x)<\infty$,
\[
p=\nabla_\tau\varphi(x)+q_u(x)\nu(x)\in\bar\partial u(x);
\]
see Lemma~\ref{lem:ray}. Consequently, $x\in\mathcal X_\varphi(p)$ and
$p\cdot\nu(x)=q_u(x)$, so $p$ does not belong to the
right-hand side of the above identity. 

Since the identity for $\mathrm D_u\setminus \mathrm S_u$ only depends on $(\varphi,q_u)$, we deduce that it determines $\mathrm D_u$ up to a Lebesgue null set.
\end{proof}

\begin{corollary}[Equality of subdifferential images]
\label{cor:D-equality}
Let $\Omega\subset\mathbb R^n$ be a bounded strictly convex $C^{1,1}$ domain and $\varphi\in C^{1,1}(\partial\Omega)$. Let $u_1,u_2\in C(\overline{\Omega})$ be convex, have the same boundary value $\varphi\in C^{1,1}(\partial\Omega)$, and satisfy
$q_{u_1}=q_{u_2}$ on $\partial\Omega$. Then\footnote{Recall that we denote by $A\mathbin\triangle B$ the symmetric difference between the sets $A\subset \R^n$ and $B\subset\R^n$. More concretely,
\[
    A\mathbin\triangle B=(A\setminus B)\cup (B\setminus A).
\]}
\[
 \cL^n(\mathrm D_{u_1}\mathbin\triangle \mathrm D_{u_2})=0.
\]
\end{corollary}

\begin{proof}
Since $u_1,u_2$ have the same boundary value $\varphi$ and $q_{u_1}=q_{u_2}$ on $\partial\Omega$, their boundary slope images coincide:
\[
\mathrm S_{u_1}=\mathrm S_{u_2}=:\mathrm S.
\]
Proposition~\ref{prop:D-characterization} implies that, for every $p\notin\mathrm S$,
\[
p\in\mathrm D_{u_1}
\quad\Longleftrightarrow\quad
p\in\mathrm D_{u_2}.
\]
Consequently,
\[
\mathrm D_{u_1}\mathbin\triangle\mathrm D_{u_2}
\subset\mathrm S.
\]
By Lemma~\ref{lem:S-null}, $\mathcal L^n(\mathrm S)=0$, which proves the assertion.
\end{proof}

\section{Integral identity for equal Cauchy pairs}
\label{sec:Alexandrov-change-of-var}

In this section, we establish a change of variables formula for bounded Borel functions when the Monge–Ampère measure is absolutely continuous. This formula yields an integral identity for generalized Cauchy pairs. The preceding geometric results do not require absolute continuity of the Monge--Amp\`ere measure.

\begin{proposition}[Change of variables formula for Borel functions]
\label{prop:weighted}
    Let $\Omega\subset \R^n$ be an open convex set and let $f\in L^1(\Omega)$ be nonnegative. Suppose that the convex function $u\colon \Omega\to \R$ is an Alexandrov solution of the Monge--Ampère equation   
    \begin{equation}
    \label{eq:MA-integral}
        \det D^2 u=f\quad\text{in }\Omega.
    \end{equation}
Then, for every bounded Borel function $H\colon \R^n\to\R$, one has 
\begin{equation}
\label{eq:weighted}
     \int_\Omega H(\nabla u(x))f(x)\,dx =\int_{\mathrm D_u}H(p)\,dp.
\end{equation}
\end{proposition}

\begin{remark}[Gradient in change of variables formula]
    We note that the gradient in the left hand side of the change of variables formula \eqref{eq:weighted} may be defined arbitrarily on the nondifferentiability set of $u$, because convex functions are almost everywhere differentiable.
\end{remark}

\begin{proof}
    Since $u$ is a convex function, Lemma~\ref{lem:diff-convex} shows that $u$ is locally Lipschitz and hence, by Rademacher's theorem, $u$ is almost everywhere differentiable. Let us denote this exceptional set by
    \[
        N\vcentcolon = \{x\in\Omega : u\text{ is not differentiable in }x\}.
    \]
    Since $\mathcal{L}^n$ is Borel regular, there exists a Borel set $\mathcal{N}\subset\Omega$ such that $N\subset \mathcal{N}$ and $\mathcal{L}^n(\mathcal{N})=\mathcal{L}^n(N)=0$. Moreover, by definition of an Alexandrov solution to \eqref{eq:MA-integral}, we have
    \begin{equation}
    \label{eq:MA-change-variables}
        \mu_u=f\,dx\quad\text{as Borel measures}
    \end{equation}
   and hence
    \[
         \mu_u(\mathcal{N})=\int_{\mathcal{ N}} f\,dx=0.
    \]
By the definition of Alexandrov measure,
\[
 \cL^n(\partial u(\mathcal{N}))=0.
\]

Let us define the Borel set $A\vcentcolon =\Omega\setminus \mathcal{N}$.  By Remark~\ref{rem:rel-subgrad-grad}, we have $\partial u(x)=\{\nabla u(x)\}$ for every $x\in A$. Therefore
\[
 \mathrm D_u=\nabla u(A)\cup\partial u(\mathcal{N}).
\]
We deduce that
\begin{equation}
\label{eq:D-grad}
\begin{split}
     \cL^n(\mathrm D_u\mathbin\triangle\nabla u(A))&=\mathcal{L}^n(D_u\setminus \nabla u(A))=\mathcal{L}^n(\partial u(\mathcal{N})\setminus \nabla u(A))=0.
\end{split}
\end{equation}

Moreover, for any Borel set $B\subset \R^n$, the preimage
\[
 E_B\vcentcolon =(\nabla u|_A)^{-1}(B)=\{x\in A:\nabla u(x)\in B\}
\]
is Borel, because $\nabla u\colon A\to\R^n$ is Borel measurable. The Borel measurability of $\nabla u|_A$ follows from the fact that the components of $\nabla u(x)$, $x\in A$, are the pointwise limit of continuous functions, which are by themselves Borel measurable (see \cite[Section~3.1.2]{evans-mt}). Also, $\mathrm D_u$ is Borel measurable. Indeed, let $(K_n)_{n\in\N}\subset \Omega$ be an exhaustion by compact sets, then 
\[
    \mathrm D_u=\bigcup_{n\in\N}\partial u(K_n).
\]
Since $\partial u(K_n)$ are compact by Lemma~\ref{lemma:compactness-subgradinets}, $\mathrm D_u$ is a Borel set. Then
\[
 \mu_u(E_B)=\cL^n(\partial u(E_B))
 =\cL^n(\nabla u(E_B)).
\]
Using that for every function $F\colon X\to Y$ between two sets $X$ and $Y$, one has the $F(F^{-1}(Y'))=Y'\cap F(X)$, we have $\nabla u(E_B)=B\cap \nabla u(A)$. Hence \eqref{eq:D-grad} yields
\[
 \mu_u(E_B)=\cL^n(\nabla u(E_B))=\cL^n(B\cap \nabla u(A))=\cL^n(B\cap \mathrm D_u).
\]
Finally, let us define
\[
    \overline{\nabla u}(x)\vcentcolon =\begin{cases}
        \nabla u(x),\quad & x\in A,\\
        0,\quad & x\in \mathcal{N}=\Omega\setminus A,
    \end{cases}
\]
which is again a Borel measurable function. By
\eqref{eq:MA-change-variables} and $\mathcal{L}^n(\mathcal{N})=0$ we get
\begin{equation}
\label{eq:equal-as-borel}
 (\overline{\nabla u})_\#(f\,dx)=\cL^n\res \mathrm D_u\text{ as Borel measures}.
\end{equation}

Now, the integral identity \eqref{eq:weighted} follows from \eqref{eq:equal-as-borel}, since $\overline{\nabla u}_{A}=\nabla u$ and $\mathcal{L}^n(\Omega\setminus A)=\mathcal{L}^n(\mathcal{N})=0$. Indeed, it first gives the desired identity for Borel measurable indicator functions $\chi_B$ and then by linearity for Borel simple functions $\sum_{n=1}^N a_n\chi_{B_n}$. Finally, $\mu_u(\Omega)=\int_{\Omega}f\,dx<\infty$ allows us to conclude formula \eqref{eq:weighted} for bounded Borel functions, since we may invoke monotone convergence applied to positive and negative parts of $H$. This completes the proof.
\end{proof}

As an immediate corollary we obtain the following nonlinear integral identity for equal Cauchy data pairs, which will be used in the uniqueness proofs for the inverse problem.

\begin{corollary}[Integral identity for equal Cauchy pairs]\label{cor:identity}
Let $\Omega\subset\mathbb R^n$ be a bounded strictly convex $C^{1,1}$ domain, let $f_j\in L^1(\Omega)$ be nonnegative functions, and assume that $\varphi\in C^{1,1}(\partial\Omega)$. For $j=1,2$, let $u_j\in C(\overline\Omega)$ be an Alexandrov solution of the Dirichlet problem
\[
     \begin{cases}
    \det D^2u=f_j,\quad&\text{in }\Omega,\\
     u=\varphi,\quad&\text{on }\partial\Omega.
    \end{cases}
\]
If the extremal boundary slopes of the two solutions agree,
\[
    q_{u_1}(x)=q_{u_2}(x),\quad x\in \partial \Omega,
\]
then the integral identity
\begin{equation}
\label{eq:two-weighted}
 \int_\Omega H(\nabla u_1)f_1\,dx
 =\int_\Omega H(\nabla u_2)f_2\,dx
\end{equation}
holds for every bounded Borel function $H\colon \R^n\to \R$.
\end{corollary}

\begin{proof}
    First, from Corollary~\ref{cor:D-equality} we deduce that
    \[
        \mathcal{L}^n(\mathrm{D}_{u_1}\mathbin \triangle \mathrm{D}_{u_2})=0.
    \]
    Thus, using Proposition~\ref{prop:weighted}, we get for any bounded Borel function $H\colon \R^n\to\R$
    \[
    \begin{split}
        \int_\Omega H(\nabla u_1)f_1\,dx&=\int_{\mathrm D_{u_1}}H(p)\,dp\\
        &=\int_{\mathrm D_{u_2}}H(p)\,dp\\
        &=\int_\Omega H(\nabla u_2)f_2\,dx.
    \end{split}
    \]
    This completes the proof.
\end{proof}

\section{Large-amplitude boundary values}
\label{sec:large-amplitude-bdry-values}

In this section we study the asymptotic behaviour of Alexandrov solutions to the Dirichlet problem 
\begin{equation}
\label{eq:asymp-MA}
     \begin{cases}
    \det D^2u=f,\quad&\text{in }\Omega,\\
     u=t\Phi_\omega,\quad&\text{on }\partial\Omega,
    \end{cases}
\end{equation}
as $t\to\infty$. Throughout this section $\Omega\subset\R^n$ is a bounded strictly convex $C^{1,1}$ domain, $f\in L^1(\Omega)$ is nonnegative, $\omega\in \mathbb{S}^{n-1}$ is a fixed direction, and the boundary value $\Phi_\omega\colon \R^n\to \R$ is given by
\begin{equation}
\label{eq:Phi}
 \Phi_\omega(x)\vcentcolon =\frac12|P_\omega x|^2,
\end{equation}
where $P_\omega$ is the orthogonal projection
\[
 P_\omega\vcentcolon =I-\omega\otimes\omega
\]
onto $\omega^\perp$. Indeed, for each $t>0$, the well-posedness of \eqref{eq:asymp-MA} follows from Proposition~\ref{prop:forward} \ref{forward-existence}. We shall denote this unique solution by $u_{f,t}\in C(\overline{\Omega})$. Moreover, by the scaling property of the Monge--Ampère measure Proposition~\ref{prop:scaling} \ref{scaling} we see that the convex function $U_{f,t}\in C(\overline{\Omega})$
\begin{equation}
\label{eq:sol-U}
 U_{f,t}\vcentcolon =t^{-1}u_{f,t}
\end{equation}
solves 
\begin{equation}
\label{eq:rescaled-PDE}
    \begin{cases}
    \det D^2 U=t^{-n} f,\quad&\text{in }\Omega,\\
     U=\Phi_\omega,\quad&\text{on }\partial\Omega
    \end{cases}
\end{equation}
in the Alexandrov sense. Moreover, by Proposition~\ref{prop:forward} \ref{forward-existence}, the functions $U_{f,t}$ are the unique solutions of \eqref{eq:rescaled-PDE}.

\begin{lemma}[Degenerate limit]
\label{lem:limit}
  Let $\Omega\subset\R^n$ be a bounded strictly convex $C^{1,1}$ domain, let $f\in L^1(\Omega)$ be nonnegative, and suppose that $\omega\in \mathbb{S}^{n-1}$ is a fixed direction. Moreover, for each $t>0$, let $U_{f,t} \in C(\overline{\Omega})$ be the unique Alexandrov solution of \eqref{eq:rescaled-PDE}. Then
\[
 U_{f,t}\to\Phi_\omega
 \quad\text{uniformly on }\overline\Omega
\]
as $t\to \infty$. Furthermore, along every sequence $t_k\to\infty$,
\begin{equation}
\label{eq:gradient-conv}
 \nabla U_{f,t_k}(x)\to P_\omega x
 \quad\text{for a.e. }x\in\Omega.
\end{equation}
\end{lemma}

\begin{proof}
   Let us first observe that the convex function $\Phi_\omega\in C(\overline{\Omega})$ is the unique Alexandrov solution of the Dirichlet problem
   \begin{equation}
   \label{eq:MA-zero-source-Phi}
       \begin{cases}
    \det D^2 U=0,\quad&\text{in }\Omega,\\
     U=\Phi_\omega,\quad&\text{on }\partial\Omega.
    \end{cases}
   \end{equation}
   Indeed, $\Phi_\omega$ satisfies
\[
 \nabla\Phi_\omega (x)=P_\omega x\quad\text{for each }x\in\R^n,
\]
which yields
\[
    \mu_{\Phi_\omega}(A)=\mathcal{L}^n(\nabla \Phi_\omega (A))=\mathcal{L}^n(P_\omega (A))\leq \mathcal{L}^n(\omega^\perp)=0
\]
for every set $A\subset\R^n$, where we used that $\omega^\perp$ is an $(n-1)$ dimensional subspace of $\R^n$.

On the other hand, since $f\in L^1(\Omega)$,
\[
 t^{-n}\int_B f\,dx\to 0\quad \text{as }t\to\infty
\]
for every Borel set $B\subset \Omega$. In particular, the source stability of the Monge--Ampère problem (Proposition~\ref{prop:forward} \ref{forward-source-stability}) implies that 
\[
    U_{f,t}\to \Phi_\omega\quad\text{uniformly on }\overline{\Omega}.
\]
For a fixed sequence $t_k\to \infty$, let $\mathcal{N}_k\subset \Omega$ be a Borel set, which has vanishing Lebesgue measure and contains the nondifferentiability points of $U_{f,t_k}$ (see proof of Proposition~\ref{prop:weighted}). Hence, 
\[
    \mathcal{N}\vcentcolon = \bigcup_{k\in\N}\mathcal{N}_k
\]
is a Borel set with vanishing Lebesgue measure such that for each $k\in\N$ the function $U_{f,t_k}$ is differentiable on $\Omega\setminus \mathcal{N}$. Lemma~\ref{lem:gradient-convergence} \ref{conv-subgrad-2}, applied at every remaining point, gives the gradient convergence \eqref{eq:gradient-conv}.
\end{proof}

\section{Proof of the uniqueness theorems}
\label{sec:Proof-uniqueness}

\begin{proof}[Proof of Theorem~\ref{thm:cylindrical}]
Fix $\omega\in \mathbb{S}^{n-1}$ and let $t_k\vcentcolon = t_k^\omega\to \infty$ be the sequence
from the hypothesis of the theorem. By \eqref{eq:countably-equal-slopes}, the two solutions $u_{j,k}\vcentcolon = u^{\omega}_{j,t_k^\omega}$, $j=1,2$, of \eqref{eq:MA-main-thm} have
the same one-sided convex boundary slopes $q_{u_{j,k}}$, $j=1,2$. As in Section~\ref{sec:large-amplitude-bdry-values}, the scaling property of the Monge--Ampère measures (Proposition~\ref{prop:scaling} \ref{scaling}) and the well-posedness theorem of the Monge--Ampère equation (Proposition~\ref{prop:forward} \ref{forward-existence}) show that the convex function $U_{j,k}\vcentcolon = t_k^{-1}u_{j,k}\in C(\overline{\Omega})$ is the unique solution of 
\begin{equation}
\label{eq:rescaled-main}
    \begin{cases}
    \det D^2 U=t_k^{-n} f_j,\quad&\text{in }\Omega,\\
     U=\Phi_\omega,\quad&\text{on }\partial\Omega,
    \end{cases}
\end{equation}
for $j=1,2$. Moreover, Proposition~\ref{def:slope} and the scaling property $\partial (\lambda u)=\lambda \partial u$, for each $\lambda>0$, show that $q_{U_{j,k}}=t_{k}^{-1} q_{u_{j,k}}$, $j=1,2$, and hence
\begin{equation}
\label{eq:equal-slopes-reduced}
    q_{U_{1,k}}=q_{U_{2,k}}\quad\text{on }\partial\Omega,\quad k\in\N.
\end{equation}
Let $\psi\in C_c(\omega^\perp)$ and apply
Corollary~\ref{cor:identity} to the two equations \eqref{eq:rescaled-main}, with
$t=t_k$, and with
\[
H\colon \R^n\to\R,\quad H(p)\vcentcolon =\psi(P_\omega p).
\]
We obtain
\[
 t_k^{-n}\int_\Omega
 \psi(P_\omega\nabla U_{1,k})f_1\,dx
=t_k^{-n}\int_\Omega
 \psi(P_\omega\nabla U_{2,k})f_2\,dx.
\]
After cancelling $t_k^{-n}$,
\begin{equation}
\label{eq:prelimit}
 \int_\Omega\psi(P_\omega\nabla U_{1,k})f_1\,dx
 =\int_\Omega\psi(P_\omega\nabla U_{2,k})f_2\,dx.
\end{equation}
By Lemma~\ref{lem:limit}, we have
\[
 \nabla U_{j,k}(x)\to P_\omega x
 \quad\text{for a.e. }x \in \Omega
\]
as $k\to\infty$. Since $\psi$ is bounded, $f_j\in L^1(\Omega)$, and $P^2_\omega = P_\omega$, dominated convergence in \eqref{eq:prelimit} yields
\begin{equation}
\label{eq:projection-identity}
 \int_\Omega\psi(P_\omega x)(f_1-f_2)(x)\,dx=0.
\end{equation}
Now, extend $g=f_1-f_2$ by zero to an $L^1$ function on $\R^n$. Each $x\in \R^n$ can be written uniquely as $x=x_\perp+ x_\parallel$ with $x_\perp\vcentcolon = P_\omega x \in \omega^\perp$ and $x_\parallel=s\omega$ for some $s\in\R$. Thus, Fubini's theorem allows us to rewrite \eqref{eq:projection-identity} as
\begin{equation}
\label{eq:int-id}
 \int_{\omega^\perp}\psi(x_\perp )
 \left(\int_\R g(x_\perp+s\omega)\,ds\right)dx_\perp =0.
\end{equation}
By Fubini's theorem, the \emph{$X$-ray transform}
\[
    X_\omega g(x_\perp)\vcentcolon =\int_\R g(x_\perp+s\omega)\,ds
\]
is well-defined for a.e.~$x_\perp \in \omega^\perp$. Indeed, $g\in L^1(\R^n)$ and Fubini's theorem yield
\[
    X_\omega g\in L^1(\omega^\perp)\quad\text{with}\quad \|X_\omega g\|_{L^1(\omega^\perp)}\leq \|g\|_{L^1(\R^n)}.
\]
Since \eqref{eq:int-id} holds for every $\psi\in C_c(\omega^\perp)$, the fundamental lemma of the calculus of variations yields
\[
 X_\omega g(x_\perp) =\int_\R g(x_\perp+s\omega)\,ds=0
 \quad\text{for a.e. }x_\perp\in\omega^\perp.
\]
For every fixed vector $\eta\in\omega^\perp$, multiplying $\int_\R g(x_\perp+s\omega)\,ds$ by $e^{-i x_\perp\cdot \eta}$ and then integrating it over $\omega^\perp$ yields
\[
 \mathcal F_{\omega^\perp}(X_\omega g)(\eta)
 \vcentcolon =\int_{\omega^\perp}\int_{\R}g(x_\perp +s\omega)e^{-i x_\perp \cdot\eta}\,ds\,dx_\perp
 =\widehat g(\eta),
\]
where $\widehat{g}$ denotes the Fourier transform of $g$ in $\R^n$ and the last equality follows from $g\in L^1(\R^n)$, Fubini's theorem, and $x\cdot \eta= x_\perp \cdot \eta$. This is the Fourier-slice identity; see
\cite[Chapter~I, Section~2]{Helgason}. Consequently
\[
    \widehat g(\eta)=0\quad \text{whenever $\eta\in\R^n$ satisfies $\eta\cdot\omega=0$.}
\]
Since $\omega\in \mathbb{S}^{n-1}$ is arbitrary and for  every vector $\eta\in\R^n$ one can choose a unit vector $\omega\in\mathbb{S}^{n-1}$ such that $\omega\perp\eta$, we deduce that $\widehat g (\eta)= 0$ for all $\eta\in\R^n$. As the Fourier transform is injective on
$L^1(\R^n)$, we get $g=0$ almost everywhere.
\end{proof}

\begin{proof}[Proof of Corollary~\ref{thm:main}]
Fix $\omega\in \mathbb{S}^{n-1}$ and $t>0$. The pair
\[
 (t\Phi_\omega|_{\partial\Omega},q_{u_{1,t}^\omega})
\]
belongs to $\cC_{f_1}^{\mathrm{glob}}$ and therefore to
$\cC_{f_2}^{\mathrm{glob}}$. Here and throughout the proof, $u_{j,t}^\omega \in C(\overline{\Omega})$ for  $j=1,2$, $t>0$, and $\omega\in\mathbb S^{n-1}$, denotes the unique Alexandrov solution to the Monge--Ampère problem with source $f_j$ and boundary value $t\Phi_\omega$ (see~Proposition~\ref{prop:forward} \ref{forward-existence}). Uniqueness of the Alexandrov Dirichlet problem for the source $f_2$ implies that the element of the latter data set with
Dirichlet component $t\Phi_\omega|_{\partial\Omega}$ is
$(t\Phi_\omega|_{\partial\Omega},q_{u_{2,t}^\omega})$.  Consequently
$q_{u_{1,t}^\omega}=q_{u_{2,t}^\omega}$.  Theorem~\ref{thm:cylindrical},
applied for instance with $t_k=k$, gives $f_1=f_2$ almost everywhere.
\end{proof}

We next prove Corollary~\ref{cor:Linfty}. This result considers bounded sources as unknowns and the Dirichlet-to-interior-Neumann maps as measurement operators.

\begin{proof}[{Proof of Corollary~\ref{cor:Linfty}}]
First, Proposition~\ref{prop:boundary-Lipschitz} and
Lemma~\ref{lem:radial-slope} yield
\[
\Lambda_{f_j}^-(\varphi)
=
q_{u_{f_j,\varphi}}
\qquad\text{on }\partial\Omega
\]
for $j=1,2$ and every
$\varphi\in C^{1,1}(\partial\Omega)$. In particular, for every $\omega\in\mathbb S^{n-1}$ and $t>0$,
the extremal boundary slopes corresponding to the boundary value
$t\Phi_\omega|_{\partial\Omega}$ agree. Theorem~\ref{thm:cylindrical}, applied with $t_k^\omega=k$, therefore gives $f_1=f_2$ almost everywhere in $\Omega$.
\end{proof}


\section{Extension to gradient-weighted Monge--Amp\`ere equations}
\label{sec:weighted-MA-extension}

In this section, we establish existence for the weighted Alexandrov
Dirichlet problem and prove
Theorem~\ref{thm:weighted-MA-uniqueness}. 

\subsection{Weighted Monge--Ampère measures}

We begin by introducing
the weighted Monge--Ampère measure and the corresponding notion of
an Alexandrov solution. Let
$\vartheta\colon\mathbb R^n\to(0,\infty)$ be a locally integrable
Borel function. Let $\lambda_\vartheta$ be the outer measure induced by the
weighted Lebesgue measure $\vartheta\,dp$, that is,
\[
\lambda_\vartheta(A)
\vcentcolon = 
\inf\left\{
\int_B\vartheta(p)\,dp:
A\subset B,\ B\subset\mathbb R^n
\text{ Lebesgue measurable}
\right\}.
\]
This outer measure is Borel regular and satisfies
\begin{equation}
\label{eq:equal-Lebesgue-set}
\lambda_\vartheta(B)=\int_B\vartheta(p)\,dp
\end{equation}
for every Lebesgue measurable set $B\subset\mathbb R^n$.
For a convex function $u\colon\Omega\to\mathbb R$, define its
\emph{$\vartheta$-weighted Monge--Ampère measure} by
\begin{equation}
\label{eq:weighted-MA-measure}
\mu_u^\vartheta(A)
\vcentcolon =
\lambda_\vartheta(\partial u(A)),
\qquad A\subset\Omega.
\end{equation}
The arguments used in the proof of Lemma~\ref{lem:MA-measure}, with $\lambda_\vartheta$ in place of $\mathcal L^n$, show that $\mu_u^\vartheta$ is a locally finite Borel measure on $\Omega$. Indeed, local finiteness follows because $\partial u(K)$ is compact for every $K\Subset\Omega$ (Lemma~\ref{lemma:compactness-subgradinets}) and $\vartheta\in L^1_{\mathrm{loc}}(\mathbb R^n)$.

We point out that the weighted Monge–Ampère measure is in general neither superadditive nor translation invariant. Nevertheless, the proof of superadditivity in \cite[Lemma~2.9]{Figalli} can be restored for superadditive weights $\vartheta$, but superadditivity of $\vartheta$ is incompatible with strict positivity of $\vartheta$. Luckily, we do not need this property for our existence proof.

Moreover, if $u\in C^{1,1}(\overline\Omega)$, then
$\partial u(x)=\{\nabla u(x)\}$ for every $x\in\Omega$, and the
change-of-variables formula
\cite[Section~3.3.3, Theorem~3.9]{evans-mt} gives, for every Borel
set $A\subset\Omega$,
\[
\int_A\vartheta(\nabla u(x))\det D^2u(x)\,dx
=
\int_{\mathbb R^n} 
\vartheta(y)\,\mathcal{H}^0(A\cap (\nabla u)^{-1}(y))\,d\mathcal{H}^n(y);
\]
cf.~proof of Lemma~\ref{lemma:Area-formula-grad}.
The set of slopes supporting a convex function at more than one
point is Lebesgue-null; see Lemma~\ref{lem:unique-subgradients}. Hence
$\mathcal{H}^0(A\cap (\nabla u)^{-1}(y))=\mathbf 1_{\nabla u(A)}(y)$ for almost every $y\in\R^n$.
Since $\partial u(A)=\nabla u(A)$, we conclude that
\[
\int_A\vartheta(\nabla u(x))\det D^2u(x)\,dx
=\int_{\partial u(A)} 
\vartheta(y)\,dy=\mu_u^\vartheta(A).
\]
Therefore
\[
\mu_u^\vartheta
=
\vartheta(\nabla u)\det D^2u\,dx
\]
as Borel measures on $\Omega$. If, in addition,
$\vartheta\in L^\infty(\mathbb R^n)$, then the density
$\vartheta(\nabla u)\det D^2u$ belongs to $L^\infty(\Omega)$. Thus the measure equation
\begin{equation}
\label{eq:Alexandrov-cond}
 \mu_u^\vartheta=f\,dx\quad\text{as Borel measures on }\Omega
\end{equation}
is the Alexandrov formulation of
\begin{equation}
\label{eq:gradient-weighted-MA}
 \vartheta(\nabla u)\det D^2u=f
 \quad\text{in }\Omega.
\end{equation}
More generally, given any Borel measure $\nu$ and any boundary condition $\varphi\in C(\partial\Omega)$, we call a convex function $u\colon\Omega\to\R$ an Alexandrov solution of the Dirichlet problem
\begin{equation}
\label{eq:weighted-MA-Dirichlet}
\begin{cases}
     \vartheta(\nabla u)\det D^2u=\nu
 \quad& \text{in }\Omega,\\
 u=\varphi
 \quad&\text{on }\partial\Omega,
\end{cases}
\end{equation}
if $u\in C(\overline{\Omega})$, $\mu^{\vartheta}_u=\nu$ as Borel measures on $\Omega$, and $u=\varphi$ on $\partial\Omega$ pointwise.

\subsection{Existence of solutions to gradient-weighted Monge--Ampère equations}

We next establish the forward existence result needed in Theorem~\ref{thm:weighted-MA-uniqueness}. We first record the weighted counterpart of the weak continuity of ordinary Monge--Amp\`ere measures (Lemma~\ref{prop:weak-cont}).

\begin{proposition}[Weak continuity of weighted Monge--Amp\`ere measures]
\label{prop:weighted-weak-cont}
Let $\Omega\subset\R^n$ be a convex open set, and let
$u_k,u\colon\Omega\to\R$ be convex functions such that
$u_k\to u$ locally uniformly in $\Omega$. Then
\[
\mu_{u_k}^\vartheta\weakstar\mu_u^\vartheta
\qquad\text{as }k\to\infty.
\]
\end{proposition}

\begin{proof}
We adapt the proof of Proposition~\ref{prop:weak-cont}. Again by Proposition~\ref{prop:weak-convergence-borel-measures}, it is enough to show that
\[
    \limsup_{k\to\infty}\mu^{\vartheta}_{u_k}(K)\leq \mu^{\vartheta}_u (K)\quad\text{and}\quad \mu^{\vartheta}_u (O)\leq \liminf_{k\to\infty}\mu^{\vartheta}_{u_k}(O)
\]
for every compact set $K\subset \Omega$ and every relatively compact open set $O\Subset \Omega$. 

Let
$K\Subset\Omega$ be compact. Following \cite[Proof of Proposition~2.6]{Figalli}, we observe that local uniform convergence and the
subgradient inequality imply
\begin{equation}
\label{eq:weighted-subgradient-limsup}
\limsup_{k\to\infty}\partial u_k(K)\vcentcolon = \bigcap_{j=1}^\infty \bigcup_{k=j}^\infty \partial u_k (K)
\subset \partial u(K).
\end{equation}
Using $0\leq \vartheta\in
L^1_{\mathrm{loc}}(\R^n)$ and continuity of the Borel regular measure $\lambda_\vartheta$, we get
\[
\begin{split}
     \mu^{\vartheta}_u(K)= \lambda_\vartheta (\partial u(K))&\geq \lambda_{\vartheta}\left(\bigcap_{j=1}^\infty \bigcup_{k=j}^\infty \partial u_k (K)\right)\\
     &=\lim_{j\to\infty}\lambda_{\vartheta}\left(\bigcup_{k=j}^\infty \partial u_k (K)\right)\\
     &\geq \limsup_{j\to \infty}\lambda_{\vartheta}(\partial u_j(K))\\
     &=\limsup_{j\to\infty}\mu^{\vartheta}_{u_j}(K).
\end{split}
\]
In the third line, we used that local uniform convergence $u_k\to u$ implies that the sets $\partial u_k(K)$ are contained in a fixed compact subset of $\mathbb R^n$ for all sufficiently large $k$.

Now let $O\Subset\Omega$ be open. Since $\partial u(A)$ is Lebesgue measurable for every Borel set $A\subset \Omega$ (see Lemma~\ref{lem:MA-measure}) and $\mu_u^\vartheta$ is a locally finite Borel measure, \cite[Section~1.1.1]{evans-mt} implies that 
\begin{equation}
\label{eq:inner-reg}
    \mu_u^\vartheta (O)=\sup_{K\subset O:\,K\text{ compact}}\mu_u^\vartheta (K).
\end{equation}
Therefore it suffices to prove 
\[
    \mu_u^\vartheta (K) \leq \liminf_{k\to\infty}\mu_{u_k}^\vartheta (O)
\]
for any compact set $K\subset O$. Let us fix such a compact set $K\subset O$ and let $Z_u$ be the set of slopes that support $u$ at more than one point. By Lemma~\ref{lem:unique-subgradients}, $Z_u$ has Lebesgue measure zero. By \cite[Eq.~(2.6)]{Figalli}, we have 
\[
    \partial u(K)\setminus Z_u\subset \liminf_{k\to\infty}\partial u_k(O)\vcentcolon = \bigcup_{j=1}^\infty \bigcap_{k=j}^\infty \partial u_k (O).
\]
Since $\lambda_\vartheta\ll\mathcal L^n$, we have $\lambda_\vartheta(Z_u)=0$. Therefore
\[
\begin{split}
    \mu_u^\vartheta(K)=\lambda_\vartheta (\partial u(K)\setminus Z_u)&\leq \lambda_\vartheta \left(\bigcup_{j=1}^\infty \bigcap_{k=j}^\infty \partial u_k (O)\right)\\
    &=\lim_{j\to\infty}\lambda_\vartheta \left(\bigcap_{k=j}^\infty \partial u_k (O)\right)\\
    &\leq \liminf_{j\to\infty}\lambda_\vartheta \left(\partial u_j (O)\right)\\
    &=\liminf_{j\to\infty}\mu_{u_j}^\vartheta(O).
\end{split}
\]
Hence, we can complete the proof.
\end{proof}

\begin{theorem}[Existence for the weighted Dirichlet problem]
\label{thm:weighted-forward-existence}
Let $\Omega\subset\R^n$ be a bounded strictly convex domain, let $\vartheta\colon\R^n\to(0,\infty)$ be a locally integrable Borel function satisfying \eqref{eq:bounds-weight-vartheta}, and let $\nu$ be a finite Borel measure on $\Omega$. Then, for every $\varphi\in C(\partial\Omega)$, there exists an Alexandrov solution $u\in C(\overline\Omega)$ of the Dirichlet problem
\begin{equation}
\label{eq:weighted-general-Dirichlet}
\begin{cases}
     \vartheta(\nabla u)\det D^2u=\nu
 \quad& \text{in }\Omega,\\
 u=\varphi
 \quad&\text{on }\partial\Omega.
\end{cases}
\end{equation}
\end{theorem}

\begin{proof}
We follow the Perron construction in the proof of
\cite[Theorems~2.13 and~2.14]{Figalli} (see the discussion in Remark~\ref{rem:existence-of-sols}). We first assume that
\begin{equation}
\label{eq:weighted-atomic-source}
\nu=\sum_{i=1}^N\alpha_i\delta_{x_i},
\qquad
x_i\in\Omega,\quad \alpha_i>0.
\end{equation}
Let $h\in C(\overline\Omega)$ be the unique ordinary Alexandrov
solution of
\begin{equation}
\label{eq:MA-for-h}
\begin{cases}
     \det D^2 h=0
 \quad& \text{in }\Omega,\\
 h=\varphi
 \quad&\text{on }\partial\Omega.
\end{cases}
\end{equation}
whose existence is provided by Proposition~\ref{prop:forward}.
For every $x_i$, let $C_i$ be the conical convex function that vanishes on $\partial\Omega$ and satisfies $C_i(x_i)=-1$; see
Figure~\ref{fig:conical-function} and \cite[Figure~2.5]{Figalli}.

\begin{figure}[ht]
\centering
\begin{tikzpicture}[
    x={(1cm,0cm)},
    y={(0.48cm,0.25cm)},
    z={(0cm,1.6cm)},
    line join=round,
    line cap=round,
    scale=0.9
]
  \draw[->,gray!75] (-3.7,0,0) -- (4.1,0,0)
    node[right] {$x_1$};
  \draw[->,gray!75] (0,-2.6,0) -- (0,2.9,0)
    node[above right] {$x_2$};
  \draw[->,gray!75] (0,0,-1.25) -- (0,0,0.65)
    node[above] {$z$};

  \foreach \t in {0,10,...,350}{
    \pgfmathsetmacro{\tnext}{\t+10}
    \fill[red!65,opacity=0.22]
      (0.55,-0.25,-1)
      -- ({3*cos(\t)},{2*sin(\t)},0)
      -- ({3*cos(\tnext)},{2*sin(\tnext)},0)
      -- cycle;
  }

  \fill[blue!10,opacity=0.45]
    plot[smooth cycle, samples=80, domain=0:360]
      ({3*cos(\x)},{2*sin(\x)},0);

  \draw[blue!65!black,thick]
    plot[smooth cycle, samples=80, domain=0:360]
      ({3*cos(\x)},{2*sin(\x)},0);

  \draw[dashed,gray!65]
    (0.55,-0.25,0) -- (0.55,-0.25,-1);

  \fill (0.55,-0.25,0) circle (1.2pt)
    node[above left] {$x_i$};

  \fill (0.55,-0.25,-1) circle (1.7pt)
    node[below right]
      {$\bigl(x_i,C_i(x_i)\bigr)=(x_i,-1)$};

  \node[blue!65!black] at (2.55,1.35,0.12)
    {$\partial\Omega\times\{0\}$};
  \node[gray!75] at (-2.45,1.25,0.08)
    {$z=0$};
  \node[red!60!black] at (-1.65,-0.75,-0.45)
    {$\operatorname{graph} C_i$};
\end{tikzpicture}

\caption{Schematic graph of the conical convex function $C_i$.
Its boundary values lie in the plane $z=0$, while
$C_i(x_i)=-1$.}
\label{fig:conical-function}
\end{figure}
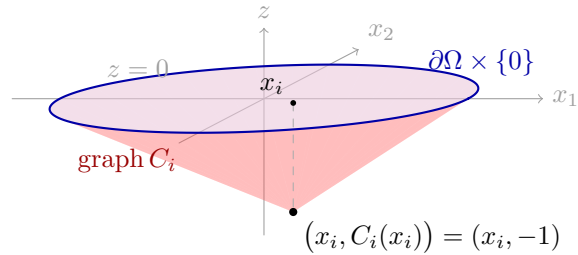
The Monge--Amp\`ere measure of $C_i$ is
concentrated at $x_i$, and
\[
    \mathcal L^n(\partial C_i(x_i))>0;
\]
see Step 1 in the proof of \cite[Theorem~2.13]{Figalli} and Example~\ref{Example:singular-MA-measures} \ref{first-sing-MA} (see also \cite[Proof of Theorem~2.8]{Figalli}).

Now, for every positive constant $L>0$, introduce the convex function
\[
v_L\vcentcolon =h+L\sum_{i=1}^N C_i.
\]
Then
\begin{equation}
\label{eq:correct-bdry-values}
    v_L =\varphi\quad\text{on }\partial\Omega.
\end{equation}
For each $i\in \{1,\ldots,N\}$, choose vectors
\[
    p_i^h \in \partial h(x_i),\quad p_i^j\in \partial C_j(x_i),\quad j\neq i,
\]
and set
\[
    a_{i,L}\vcentcolon = p_i^h +L\sum_{j\neq i}p_i^j.
\]
Since subdifferentials and scalings commute, the sum rule for subdifferentials yields
\[
    a_{i,L}+L\partial C_i(x_i)\subset \partial v_L (x_i).
\]
By Lemma~\ref{lemma:compactness-subgradinets}, $a_{i,L}+L\partial C_i(x_i)$ and $\partial v_L (x_i)$ are compact. Hence, $\mu^{\vartheta}_{v_L}(\{x_i\})<\infty$. Moreover, the monotonicity of $\lambda_\vartheta$, $\vartheta (p)\geq c_{\vartheta}>0$, and the properties of the Lebesgue measure imply that
\[
\begin{split}
    \mu^{\vartheta}_{v_L}(\{x_i\})&= \lambda_\vartheta (\partial v_L(x_i))\\
    &\geq \lambda_{\vartheta}(a_{i,L}+L\partial C_i(x_i))\\
    &=\int_{a_{i,L}+L\partial C_i(x_i)}\vartheta \,dp\\
    &\geq c_{\vartheta}\mathcal{L}^n(a_{i,L}+L\partial C_i(x_i))\\
    &=c_{\vartheta}L^n \mathcal{L}^n(\partial C_i(x_i))>0.
\end{split}
\]
Therefore, if we set
\[
    L\vcentcolon =\max_{1\leq i\leq N}\left(\frac{\alpha_i}{c_{\vartheta}\mathcal{L}^n(\partial C_i(x_i))}\right)^{1/n},
\]
then $\mu^{\vartheta}_{v_L}(\{x_i\})\geq \alpha_i$. This implies that 
\begin{equation}
\label{eq:supersol-prop}
    \mu^{\vartheta}_{v_L}\geq \sum_{i=1}^N \alpha_i \delta_{x_i}\quad\text{as Borel measures on }\Omega.
\end{equation}
Combining this with \eqref{eq:correct-bdry-values}, we see that the Perron family
\[
\mathcal S_\vartheta[\nu]
\vcentcolon =
\left\{
v\in C(\overline\Omega):
v\text{ is convex},\,
v\leq \varphi\text{ on }\partial\Omega,\,
\mu_v^\vartheta\geq\nu\text{ as Borel measures}
\right\}
\]
is nonempty.

If $v_1,v_2\in\mathcal S_\vartheta[\nu]$, then
$\max\{v_1,v_2\}\in\mathcal S_\vartheta[\nu]$. Indeed, since the proof of
Step~2 in \cite[Theorem~2.13]{Figalli} uses only properties of the subdifferential and the fact that $\mu_u$ is a measure, we can apply the same reasoning with $\lambda_\vartheta$ in place of $\mathcal L^n$. 

Now, let us define
\begin{equation}
\label{eq:definitio-u}
u(x)\vcentcolon =\sup_{v\in\mathcal S_\vartheta[\nu]}v(x),\quad \text{for all }x\in \overline{\Omega}.
\end{equation}
By \cite[Theorem~5.5]{Rockafellar}, $u\colon \overline{\Omega}\to [-\infty,\infty]$ is a convex function. The function $u$ is bounded from below on $\overline \Omega$, because $v_L\in\mathcal{S}_\vartheta [\nu]$ is finite-valued on $\overline{\Omega}$ we have 
\[
    u(x)\geq v_L(x)>-\infty,\quad \text{for every }x\in\overline\Omega.
\]
Since $\Omega$ is bounded and each $v\in \mathcal S_\vartheta[\nu]$ is convex and continuous on $\overline\Omega$, it attains its maximum on the boundary. Thus, the inequality $v\leq \varphi$ on $\partial\Omega$, which holds for each $v\in \mathcal S_\vartheta[\nu]$, implies that 
\[
    \max_{x\in \overline \Omega}v(x)\leq  \max_{x\in \partial\Omega}\varphi,\quad\text{for all }v\in \mathcal S_\vartheta [\nu].
\]
Thus, $u$ is indeed a real-valued convex function over $\overline{\Omega}$, satisfying
\[
    u(x)\leq \max_{x\in\partial\Omega} \varphi(x),\quad\text{for all }x\in \overline{\Omega}.
\]
Next, we show that $u$ is continuous up to the boundary and has boundary value $\varphi$. Observe that $\mathcal{S}_\vartheta [\nu]\subset \mathcal{S}[0]$, where 
\[
    \mathcal{S}[0]\vcentcolon =\left\{
w\in C(\overline\Omega):
w\text{ is convex},\,
w\leq \varphi\text{ on }\partial\Omega
\right\}.
\]
Moreover, by \cite[Proof of Theorem~2.14]{Figalli}, the Alexandrov solution $h\in C(\overline{\Omega})$ of \eqref{eq:MA-for-h} can be represented as
\begin{equation}
\label{eq:construct-h}
    h(x)=\sup_{w\in \mathcal{S}[0]}w(x),\quad x\in\overline{\Omega}.
\end{equation}
This in particular implies that
\[
    h(x)\geq u(x),\quad \text{for all }x\in \overline{\Omega}.
\]
On the other hand, we know that $v_L\in \mathcal{S}_\vartheta [\nu]$, and hence
\begin{equation}
\label{eq:sandwich-u}
    v_L\leq u\leq h\quad\text{for all }x\in \overline{\Omega}.
\end{equation}
By the sandwich criterion, $u$ is continuous up to the boundary and $v_L=h=\varphi$ on $\partial\Omega$ implies that $u=\varphi$ on $\partial\Omega$. Hence, 
\begin{equation}
\label{eq:continuity-of-u-bdry-value}
    u\in C(\overline{\Omega})\quad\text{with}\quad u=\varphi\text{ on }\partial\Omega.
\end{equation}

Since $C(\overline \Omega)$ is separable with respect to local uniform convergence, there exists a countable dense subset $(v_k)_{k\in\N}$ of $\mathcal S_{\vartheta}[\nu]\subset C(\overline{\Omega})$. By density and monotone convergence, it follows that
\[
    u=\sup_{k\in\N}v_k=\lim_{n\to\infty}w_n\quad\text{with }w_n\vcentcolon = \sup_{1\leq k\leq n}v_k.
\]
We note that the pointwise limit $w_n\uparrow u$ holds locally uniformly, by Lemma~\ref{lem:diff-convex} and the Arzelà--Ascoli theorem, or more directly by Dini's theorem. Furthermore, by iteration and Step 2, we deduce that $w_n\in \mathcal{S}_{\vartheta}[\nu]$, which yields that
\[
    \int_{\R^n}\psi \,d\mu_{w_n}^{\vartheta}\geq \int_{\R^n}\psi \,d\nu
\]
for every nonnegative $\psi\in C_c(\Omega)$ and $n\in\N$. Applying Proposition~\ref{prop:weighted-weak-cont}, we obtain $u\in \mathcal{S}_\vartheta[\nu]$ or more precisely
\begin{equation}
\label{eq:u-admissible}
u\in C(\overline\Omega),\quad
u=\varphi\text{ on }\partial\Omega,\quad
\mu_u^\vartheta\geq\nu\quad\text{as Borel measures}.
\end{equation}

We claim that
\begin{equation}
\label{eq:equality-borel-meas}
    \mu_u^\vartheta =\nu\quad\text{as Borel measures}.
\end{equation}
First, suppose that $\mu_u^\vartheta$ is not supported on the finite set
$\{x_1,\ldots,x_N\}$. By definition, this implies that there exists a Borel set $E\subset \Omega$ such that
\begin{equation}
\label{eq:counter-support}
   E\subset \Omega\setminus \{x_1,\ldots,x_N\}\quad\text{and}\quad \mu_u^\vartheta (E)>0.
\end{equation}
The set $\partial u(E)$ has necessarily positive Lebesgue measure as otherwise $\mu_u^\vartheta (E)=0$ (see~\eqref{eq:equal-Lebesgue-set}). Let $Z_u$ denote the set of slopes supporting $u$ at two distinct points. By Lemma~\ref{lem:unique-subgradients},
$\mathcal L^n(Z_u)=0$. Since $E\cap \{x_1,\ldots,x_N\}=\varnothing$,
\[
    \partial u(E)\cap\partial u(x_i)\subset Z_u,
    \qquad i=1,\ldots,N.
\]
Likewise,
\[
    \partial u(E)\cap\partial u(\partial\Omega)\subset Z_u.
\]
Consequently,
\[
\mathcal L^n\left(
\partial u(E)\setminus
\left[
\bigcup_{i=1}^N\partial u(x_i)
\cup\partial u(\partial\Omega)
\right]\right)>0.
\]
We may therefore choose $\bar p$ in this nonempty Lebesgue measurable set. Since
$\bar p\in\partial u(E)$, there exists $\bar x\in E$ such that
$\bar p\in\partial u(\bar x)$. By \cite[Remark~A.19]{Figalli}, there exists $\eps>0$ such that
\[
    u\geq \ell_{\bar x, \bar p}+2\eps,\quad\text{on }\{x_1,\ldots,x_N\}\cup \partial\Omega,
\]
where 
\[
\ell_{\bar x,\bar p}(z)
\vcentcolon = u(\bar x)+\bar p\cdot(z-\bar x).
\]
Now, we consider the convex function
\[
\widetilde u
\vcentcolon = 
\max\{u,\ell_{\bar x,\bar p}+\varepsilon\}\in C(\overline{\Omega}).
\]
Note that $\widetilde u\geq u$ and $\widetilde u= u$ in a neighborhood of $\{x_1,\ldots,x_N\}\cup \partial\Omega$. In particular, $\widetilde u=\varphi$ on $\partial\Omega$ and $\partial \widetilde  u(x_j)=\partial u(x_j)$, for $j=1,\ldots,N$, and consequently $\widetilde u\in\mathcal S_\vartheta[\nu]$. Since $u$ is the largest element in $\mathcal{S}_\vartheta [\nu]$ and $\widetilde u>u$ somewhere, we obtain a contradiction. Thus $\mu_u^\vartheta$ is necessarily supported on the prescribed atoms.

It remains to identify their masses $\mu_u^\vartheta (\{x_j\})$. Suppose, for some $j\in\{1,\ldots,N\}$, that
\[
\beta_j\vcentcolon =\mu_u^\vartheta(\{x_j\})>\alpha_j.
\]
Again, by absolute continuity of $\lambda_\vartheta$, $\partial u(x_j)$ has positive Lebesgue measure. Indeed, if $\mathcal{L}^n(\partial u(x_j))=0$, then Lemma~\ref{lemma:compactness-subgradinets} and \eqref{eq:equal-Lebesgue-set} would imply that
\[
    \mu_u^{\vartheta}(\{x_j\})=\lambda_{\vartheta}(\partial u(x_j))=\int_{\partial u(x_j)} \vartheta \,dp=0,
\]
which contradicts the fact that $\alpha_j>0$.
Since a convex subset of $\R^n$ with empty interior is contained in a proper affine subspace and therefore has zero $n$-dimensional Lebesgue measure, we conclude that $\partial u(x_j)$ has nonempty interior. Choose $p\in\operatorname{int}\partial u(x_j)$ and set
\begin{equation}
\label{eq:def-U}
U\vcentcolon =u-\ell\in C(\overline{\Omega})\quad\text{with}\quad \ell(z)\vcentcolon =u(x_j)+p\cdot(z-x_j).
\end{equation}
Then $\partial U(x_j)=\partial u(x_j)-p$ and $p\in \operatorname{int}\partial u(x_j)$ imply that $0\in\operatorname{int}\partial U(x_j)$. Hence there exists $r>0$ such that $\overline B_r\subset \partial U(x_j)$ and therefore, by the subgradient inequality, the vectors
\[
    q_z\vcentcolon = r\frac{z-x_j}{|z-x_j|}\in \partial U(x_j),\quad \text{for every }z\in \overline{\Omega}\text{ with }z\neq x_j
\]
satisfy
\[
    U(z)\geq U(x_j)+q_z\cdot (z-x_j)=U(x_j)+r|z-x_j|>U(x_j)
\]
for all $z\in\overline\Omega$ with $z\neq x_j$. This shows that $x_j$ is the unique minimizer of the continuous function $U$ over $\overline{\Omega}$. Since 
\[
    F_j\vcentcolon = \{x_i:i\neq j\}\cup\partial\Omega
\]
is compact, we have
\[
    \min_{z\in F_j} U(z)>U(x_j).
\]
Therefore, we may choose $a\in\R$ such that
\begin{equation}
\label{eq:min-point-U}
 U(z)>a>U(x_j)\quad\text{for every }z\in F_j.
\end{equation}
For $0<\delta<1$, define
\begin{equation}
\label{eq:def-U-tilde}
\widehat U
\vcentcolon =
\max\{U,(1-\delta)U+\delta a\}\in C(\overline \Omega).
\end{equation}
This function is convex and satisfies $\widehat U\geq U$. It agrees with $(1-\delta)U+\delta a$ near $x_j$, and with $U$ near the atoms $x_i$, $i\neq j$, or the boundary $\partial\Omega$. Moreover, it obeys 
\[
\partial\widehat U(x_j)
=(1-\delta)\partial U(x_j).
\]
Thus, the function 
\[
\widehat u\vcentcolon =\widehat U+\ell,
\]
satisfies
\[
\partial\widehat u(x_j)
=
p+(1-\delta)(\partial u(x_j)-p).
\]
Let us put
\[
    K\vcentcolon = \partial u(x_j)\quad \text{and}\quad K_\delta\vcentcolon = p+(1-\delta)(K-p),\quad 0<\delta\leq 1.
\]
Note that if $0<\delta_1<\delta_2\leq 1$, then $1-\delta_2<1-\delta_1$, convexity of $K$, and $p\in K$ give
\[
    K_{\delta_2}\subset K_{\delta_1}.
\]
Thus, $(K_\delta)_{\delta>0}$ increases monotonically as $\delta\downarrow 0$. Next, we show that the limit is $\operatorname{int}K$, that is,
\begin{equation}
\label{eq:exhaustion}
    \operatorname{int}K=\bigcup_{0<\delta\leq 1}K_\delta.
\end{equation}
Since $p\in\operatorname{int}K$, there exists $\varepsilon>0$ such that $B_\varepsilon (p)\subset K$. For every $x\in K$ and $0\leq t<1$, convexity gives
\[
tx+(1-t)B_\varepsilon(p)\subset K.
\]
Since
\[
tx+(1-t)B_\varepsilon(p)
=
B_{(1-t)\varepsilon}\bigl(p+t(x-p)\bigr),
\]
it follows that
\[
p+t(x-p)\in\operatorname{int}K.
\]
Consequently,
\[
p+t(K-p)\subset\operatorname{int}K,
\qquad 0\leq t<1.
\]
Setting $\delta\vcentcolon= 1-t$ yields $K_\delta \subset \operatorname{int}K$ for every $0<\delta\leq 1$.

Conversely, let $q\in\operatorname{int}K$, $q\neq p$, and set
\[
e\vcentcolon = \frac{q-p}{|q-p|}.
\]
Since $q$ is an interior point of $K$, there exists
$\eta>0$ such that
\[
x\vcentcolon = q+\eta e\in K.
\]
Then
\[
q=p+t(x-p),
\qquad
t\vcentcolon =\frac{|q-p|}{|q-p|+\eta}<1.
\]
Setting $\delta\vcentcolon = 1-t$ implies that $q\in K_\delta$. The case $q=p$ corresponds to $t=0$ or $\delta=1$. This completes the verification of \eqref{eq:exhaustion}.

Recall that the boundary of a convex body $\mathcal{K}$ (compact convex set with nonempty interior) is
Lebesgue null. Indeed, up to translations, we may assume that $0\in \operatorname{int}\mathcal{K}$. By convexity, we have $\lambda \mathcal K\subset \operatorname{int}\mathcal K$ for every $0\leq \lambda< 1$, and therefore
\[
    \partial \mathcal K\subset \mathcal K\setminus \lambda\mathcal K.
\]
Thus,
\[
    \mathcal{L}^n(\partial\mathcal K)\leq \mathcal{L}^n(\mathcal K\setminus \lambda\mathcal K)=\mathcal{L}^n(\mathcal{K})-\mathcal{L}^n(\lambda\mathcal{K})=(1-\lambda^n)\mathcal{L}^n(\mathcal{K}),
\]
which converges to 0 as $\lambda\uparrow 1$.

Since $K$ is a convex body and $\vartheta$ is locally integrable, we obtain
\[
\lambda_\vartheta (K_\delta)
\to\lambda_\vartheta (\operatorname{int}(K))=
\lambda_\vartheta(K)\quad\text{as }\delta\downarrow 0.
\]
Thus,
\[
    \mu_{\widehat u}^{\vartheta}(\{x_j\})\to \mu_u^{\vartheta}(\{x_j\})>\alpha_j\quad\text{as }\delta\to 0.
\]
Hence for all sufficiently small $\delta>0$, we have
$\mu_{\widehat u}^\vartheta(\{x_j\})>\alpha_j$. The modification $\widehat u$ does not change the other atomic masses or the boundary values (see \eqref{eq:def-U}, \eqref{eq:min-point-U}, and \eqref{eq:def-U-tilde}), so
$\widehat u\in\mathcal S_\vartheta[\nu]$ and
$\widehat u>u$, again contradicting maximality of $u$. 

Therefore, we have proved
$\mu_u^\vartheta=\nu$ for every source of the form
\eqref{eq:weighted-atomic-source}.

\medskip

We now consider an arbitrary finite Borel measure $\nu$. Choose
finite atomic measures $\nu_k$ such that
\[
\nu_k\weakstar\nu,
\qquad
\sup_k\nu_k(\Omega)<\infty,
\]
and let $u_k$ be the solutions just constructed with source $\nu_k$ and boundary value $\varphi$. The lower bound on $\vartheta$ (see \eqref{eq:bounds-weight-vartheta}) gives the measure inequality
\begin{equation}
\label{eq:ordinary-weighted-measure-bound}
\mu_{u_k}
\leq
c_\vartheta^{-1}\mu_{u_k}^\vartheta
=c_\vartheta^{-1}\nu_k\quad \text{as Borel measures}.
\end{equation}
Let $w_k\in C(\overline \Omega)$ be the unique Alexandrov solution of
\begin{equation}
\label{eq:MA-source-vk}
    \begin{cases}
     \det D^2 w=c_\vartheta^{-1}\nu_k
 \quad& \text{in }\Omega,\\
 w=\varphi
 \quad&\text{on }\partial\Omega.
\end{cases}
\end{equation}
Since $\mu_h=0\leq \mu_{u_k}\leq \mu_{w_k}$ as Borel measures on $\Omega$ and $u_k=w_k=h=\varphi$ on $\partial\Omega$, the comparison principle (Proposition~\ref{prop:Comparison}) yields
\begin{equation}
\label{eq:weighted-forward-sandwich}
w_k\leq u_k\leq h
\qquad\text{in }\Omega.
\end{equation}
By Proposition~\ref{prop:forward}\,\ref{forward-source-stability}, $w_k$ converges uniformly on $\overline\Omega$ to the solution $w$ of \eqref{eq:MA-source-vk} with source $c_\vartheta^{-1}\nu$ and boundary value $\varphi$. Using \eqref{eq:weighted-forward-sandwich}, we see that $u_k\in C(\overline \Omega)$ is uniformly bounded. Moreover, by Lemma~\ref{lem:diff-convex}, for every compact set $K\Subset \Omega$ there exists a constant $C>0$ such that
\[
    \|\nabla u_k\|_{L^\infty (K)}\leq C
\]
for all $k\in\N$. By the Arzelà--Ascoli theorem, there exists a subsequence, still denoted as $u_k$, and a function $u\colon \Omega\to\R$ such that $u_k\to u$ as $k\to\infty$. The limit function $u$ is necessarily convex, because all $u_k$ are. Additionally, Proposition~\ref{prop:weighted-weak-cont} and $\nu_k\weakstar \nu$ yield
\[
\mu_u^\vartheta=\nu.
\]
Passing to the limit in \eqref{eq:weighted-forward-sandwich} gives $w\leq u\leq h$. Since $w$ and $h$ extend continuously to $\partial\Omega$ with the same boundary value $\varphi$, the
sandwich $w\leq u\leq h$ implies that $u\in C(\overline\Omega)$ and
$u=\varphi$ on $\partial\Omega$. This completes the proof.
\end{proof}

\begin{remark}[Uniqueness of the weighted forward problem]
\label{rem:weighted-forward-uniqueness}
Theorem~\ref{thm:weighted-forward-existence} proves existence but
does not establish uniqueness under the minimal assumptions on
$\vartheta$. The missing ingredient is a weighted comparison
principle of the form
\[
\begin{cases}
    \mu_u^\vartheta\leq\mu_v^\vartheta&\text{ as Borel measures on }\Omega,\\
    u\geq v &\text{ on }\partial\Omega
\end{cases}
\quad \Longrightarrow\quad
u\geq v\text{ in }\Omega.
\]
The ordinary proof cannot be transferred directly. Its strict
perturbation uses the superadditivity and homogeneity of Lebesgue
measure to obtain, for a strictly convex quadratic $w$,
\[
\mu_{v+\varepsilon w}
\geq
\mu_v+\varepsilon^n\mu_w.
\]
As already pointed out, for a general Borel weight $\vartheta$, translations and dilations of
subgradient images change the density $\vartheta(p)$, and no
corresponding inequality for $\mu_{v+\varepsilon w}^\vartheta$ is available. The Perron construction above therefore produces a maximal weighted Alexandrov solution, but it does not show that every
other solution coincides with it. This issue does not affect Theorem~\ref{thm:weighted-MA-uniqueness}, which applies to any chosen families of solutions whose measured boundary slopes agree.
\end{remark}

\subsection{Cylindrical-data uniqueness for gradient weights}

Next, we prove our second main theorem, Theorem~\ref{thm:weighted-MA-uniqueness}.

\begin{proof}[Proof of Theorem~\ref{thm:weighted-MA-uniqueness}]
We first record the weighted change of variables formula. Let $u\colon\Omega\to\mathbb R$ be convex and suppose that it solves
\[
 \vartheta (\nabla u)\det D^2 u=f\quad \text{in }\Omega
\]
in the Alexandrov sense, where $f\in L^1(\Omega)$ is nonnegative. Then
\begin{equation}
\label{eq:weighted-MA-change-of-variables}
 \int_\Omega H(\nabla u(x))f(x)\,dx
 =
 \int_{\mathrm D_u}H(p)\vartheta(p)\,dp
\end{equation}
for every bounded Borel function
$H\colon\mathbb R^n\to\mathbb R$.

Indeed, let $\mathcal N\subset\Omega$ be a Borel set of Lebesgue measure zero containing the nondifferentiability set of $u$. Since $\mu_u^\vartheta=f\,dx$ as Borel measures in $\Omega$, we have
\[
 0
 =\int_{\mathcal{N}}f\,dx=
 \mu_u^\vartheta(\mathcal N)
 =
 \int_{\partial u(\mathcal N)}\vartheta(p)\,dp.
\]
Here and below, we use the observation that $\partial u(A)$ is Lebesgue measurable, whenever $A\subset \Omega$ is Borel (see~\cite[Theorem~2.3, Step~1]{Figalli}). Thus, we can use formula \eqref{eq:equal-Lebesgue-set} with $B=\partial u(A)$ for every Borel set $A\subset\Omega$. 

Using the lower bound \eqref{eq:bounds-weight-vartheta}, we infer from the previous identity that
\[
 \mathcal L^n(\partial u(\mathcal N))=0.
\]
Consequently, up to a Lebesgue null set,
$\mathrm D_u\vcentcolon =\partial u(\Omega)$ agrees with
$\nabla u(\Omega\setminus\mathcal N)$. Repeating the argument in the
proof of Proposition~\ref{prop:weighted}, with
$\vartheta(p)\,dp$ in place of $dp$, gives
\[
 (\nabla u)_\#(f\,dx)
 =
 (\vartheta\,\mathcal L^n)\res\mathrm D_u
\]
as Borel measures, where $\nabla u$ may be defined arbitrarily on
$\mathcal N$. This proves
\eqref{eq:weighted-MA-change-of-variables}.

Now fix $\omega\in\mathbb S^{n-1}$ and abbreviate
$t_k\vcentcolon = t_k^\omega$ and
\[
 u_{j,k}\vcentcolon = u_{j,t_k}^\omega,
 \qquad
 U_{j,k}\vcentcolon = t_k^{-1}u_{j,k},
 \qquad j=1,2.
\]
The equality of the boundary slopes and
Corollary~\ref{cor:D-equality} imply
\[
 \mathcal L^n\bigl(
 \mathrm D_{u_{1,k}}\mathbin\triangle
 \mathrm D_{u_{2,k}}
 \bigr)=0.
\]
Therefore, by \eqref{eq:weighted-MA-change-of-variables},
\begin{equation}
\label{eq:weighted-MA-equal-integrals}
 \int_\Omega H(\nabla u_{1,k})f_1\,dx
 =
 \int_\Omega H(\nabla u_{2,k})f_2\,dx
\end{equation}
for every bounded Borel function $H$.

We next determine the large-amplitude limit. Since
\[
 \partial U_{j,k}(A)
 =
 t_k^{-1}\partial u_{j,k}(A),
\]
the ordinary Monge--Amp\`ere measure of $U_{j,k}$ satisfies, for
every Borel set $A\subset\Omega$,
\begin{align*}
 \mu_{U_{j,k}}(A)
 &=
 t_k^{-n}\mathcal L^n(\partial u_{j,k}(A))\\
 &\leq
 \frac{t_k^{-n}}{c_\vartheta}
 \int_{\partial u_{j,k}(A)}\vartheta(p)\,dp\\
 &=
 \frac{t_k^{-n}}{c_\vartheta}
 \int_A f_j(x)\,dx.
\end{align*}
In particular,
\[
 \mu_{U_{j,k}}(\Omega)
 \leq
 \frac{t_k^{-n}}{c_\vartheta}
 \|f_j\|_{L^1(\Omega)}
 \to 0.
\]
Consequently, $\mu_{U_{j,k}}\weakstar0$ as $k\to\infty$. Moreover,
\[
 U_{j,k}=\Phi_\omega
 \quad\text{on }\partial\Omega.
\]
Hence, $U_{j,k}\in C(\overline{\Omega})$ are Alexandrov solutions of the Dirichlet problem
\[
\begin{cases}
\det D^2U=\nu_{j,k}
&\text{in }\Omega,\\
U=\Phi_\omega
&\text{on }\partial\Omega,
\end{cases}
\]
where we put $\nu_{j,k}\vcentcolon = \mu_{U_{j,k}}$. The source-stability result in Proposition~\ref{prop:forward}\,\ref{forward-source-stability} therefore gives
\[
 U_{j,k}\to\Phi_\omega
 \quad\text{uniformly on }\overline\Omega.
\]
Here we used that $\Phi_\omega$ is the unique convex solution of 
 \begin{equation}
       \begin{cases}
    \det D^2 U=0,\quad&\text{in }\Omega,\\
     U=\Phi_\omega,\quad&\text{on }\partial\Omega;
    \end{cases}
   \end{equation}
see proof of Lemma~\ref{lem:limit}. Lemma~\ref{lem:gradient-convergence} now yields
\[
 \nabla U_{j,k}(x)
 \to
 \nabla\Phi_\omega(x)=P_\omega x
 \quad\text{for a.e. }x\in\Omega.
\]
Given $\psi\in C_c(\omega^\perp)$, apply
\eqref{eq:weighted-MA-equal-integrals} with
\[
 H_k(p)\vcentcolon = \psi\left(\frac{P_\omega p}{t_k}\right).
\]
Since $\nabla u_{j,k}=t_k\nabla U_{j,k}$, we obtain
\[
 \int_\Omega
 \psi(P_\omega\nabla U_{1,k})f_1\,dx
 =
 \int_\Omega
 \psi(P_\omega\nabla U_{2,k})f_2\,dx.
\]
Dominated convergence gives
\[
 \int_\Omega
 \psi(P_\omega x)(f_1-f_2)(x)\,dx=0.
\]
As in the proof of Theorem~\ref{thm:cylindrical}, this means that the
X-ray transform of the zero extension of $f_1-f_2$ vanishes in the
direction $\omega$. Since $\omega\in\mathbb S^{n-1}$ was arbitrary,
the Fourier-slice identity implies
$f_1=f_2$ almost everywhere in $\Omega$.
\end{proof}

\begin{remark}[Choice of weighted solutions]
Theorem~\ref{thm:weighted-forward-existence} guarantees that all
solutions required in Theorem~\ref{thm:weighted-MA-uniqueness} exist.
Since uniqueness of the weighted forward problem is not known under
the present assumptions, the boundary-slope measurement may in
principle depend on the choice of a solution. The inverse theorem is
independent of this issue: for any two chosen families of weighted
Alexandrov solutions, equality of the indicated cylindrical
boundary-slope measurements implies equality of the sources.
\end{remark}

\section{Scope and limitations}
\label{sec:scope-limitations}

In this final section, we discuss the principal limitations of the present work, emphasize the distinction from our earlier work \cite{UhlmannZimmermann}, and explain the role of the extremal boundary slope even for $L^\infty$ sources.

\subsection{The role of absolute continuity}

The proof of the change of variables formula in Proposition~\ref{prop:weighted} requires that the nondifferentiability
set $N$ of the convex solution $u$ satisfies $\mu_u(N)=0$. This condition follows from the absolute continuity of
$\mu_u$, since $\mathcal L^n(N)=0$.

For a singular Monge--Ampère measure, such as those discussed in
Example~\ref{Example:singular-MA-measures}, the measure $\mu_u$ may
charge $N$. Moreover, the union of the subdifferential fibres over
$N$ may carry positive $n$-dimensional volume. Consequently, there
need not exist a canonical single-valued gradient defined
$\mu_u$-almost everywhere whose pushforward is
$\mathcal L^n\res\mathrm D_u$. The expression $H(\nabla u)$ used
in Proposition~\ref{prop:weighted} therefore has no canonical meaning
on the part of $\mu_u$ concentrated on $N$.

Section~\ref{sec:weighted-MA-extension} shows that the argument is
robust under the introduction of a known positive gradient weight
$\vartheta(\nabla u)$. The remaining obstruction is therefore not
the gradient dependence of the equation but the possible concentration
of the Monge--Ampère measure on the nondifferentiability set of the
solution. We do not claim that absolute continuity is an optimal
uniqueness threshold; it is the endpoint of the present
single-valued-gradient argument.

\subsection{Global versus local data}

Our uniqueness proof uses cylindrical boundary values
\[
t\Phi_\omega|_{\partial\Omega},\qquad t\to\infty.
\]
Although only a sequence of amplitudes is needed for each
direction, these data are not confined to a fixed neighborhood of a background boundary value. Consequently, the theorem does not establish uniqueness from the local nonlinear Cauchy data considered in \cite{UhlmannZimmermann}. This distinction is structural rather than a forward regularity issue.

\subsection{Extremal boundary slope versus interior normal derivatives}

At the $L^1$ level, Proposition~\ref{def:slope} is essential. If one defines the measurement only through the normal-line quotient (Definition~\ref{def:Interior-cauchy}), an additional theorem identifying that quotient with the extremal boundary slope is needed. The identification holds for boundary-Lipschitz solutions
(Lemma~\ref{lem:radial-slope}), and hence for nonnegative
$L^\infty$ sources under the assumptions of Corollary~\ref{cor:Linfty}. We do not assert it for arbitrary nonnegative $L^1$ sources.

\section*{Acknowledgments}

G.~Uhlmann was partially supported by the National Science Foundation. P.~Zimmermann was supported by the Swiss National Science Foundation under Grant No.~239130.

  \medskip

  \appendix

\section{~Representation theorem for uniformly convex domains:~proof}
\label{sec:representation-thm}

  In this section we deliver the proof of Lemma~\ref{lem:defining-function}, which uses the following geometric fact.

  \begin{lemma}[Almost-everywhere second-order boundary geometry]
\label{lem:second-order-boundary-geometry}
Let $\Omega\subset\mathbb R^n$ be a bounded domain with
$C^{1,1}$ boundary, and suppose that there exists $\kappa>0$ such
that
\begin{equation}
\label{eq:quadratic-boundary-separation}
\nu(x)\cdot(y-x)
\leq
-\kappa|y-x|^2
\qquad
\text{for every }x\in\partial\Omega
\text{ and }y\in\overline\Omega.
\end{equation}
Then there exists a set
$N\subset\partial\Omega$ satisfying $\mathcal H^{n-1}(N)=0$ such that the Gauss map $\nu:\partial\Omega\to\mathbb S^{n-1}$ is tangentially differentiable at every $\xi\in \partial\Omega\setminus N$.

At every $\xi\in \partial\Omega\setminus N$, the second fundamental form
\[
\mathrm{II}_\xi(\tau,\sigma)
\vcentcolon = 
D_\tau\nu(\xi)\cdot\sigma,
\qquad
\tau,\sigma\in T_\xi\partial\Omega,
\]
is well-defined and satisfies
\begin{equation}
\label{eq:II-lower}
\mathrm{II}_\xi(\tau,\tau)
\geq
2\kappa|\tau|^2,
\qquad
\tau\in T_\xi\partial\Omega.
\end{equation}
More precisely, for every $\xi\in \partial\Omega\setminus N$ and
$\tau\in T_\xi\partial\Omega$, there exists a boundary curve
$\gamma:(-\varepsilon,\varepsilon)\to\partial\Omega$ that is twice
differentiable at the origin and satisfies
\begin{equation}
\label{eq:boundary-curve-second-order}
\gamma(0)=\xi,
\qquad
\gamma'(0)=\tau,
\qquad
\mathrm{II}_\xi(\tau,\tau)
=
-\nu(\xi)\cdot\gamma''(0).
\end{equation}
\end{lemma}

\begin{proof}
Since $\partial\Omega$ is of class $C^{1,1}$, its Gauss map is Lipschitz continuous. Rademacher's theorem \cite[Chapter~3]{evans1991measure}, applied in local boundary
charts, therefore shows that $\nu$ is tangentially differentiable
at $\mathcal H^{n-1}$-almost every point of $\partial\Omega$.
Denote the set of non-differentiable points by $N$.

Let us describe the corresponding second-order boundary expansion.
After a rigid change of coordinates, a neighborhood of a boundary
point can be represented as
\begin{equation}
\label{eq:local-rep}
\partial\Omega\cap\mathcal U
=
\left\{
X(x')=(x',g(x')):x'\in V
\right\},
\end{equation}
where $V\subset\mathbb R^{n-1}$ is open and
$g\in C^{1,1}(V)$. Choosing the orientation so that the normal is
outward-pointing, the Gauss map is given by
\[
\nu(X(x'))
=
\frac{(-\nabla g(x'),1)}
{\sqrt{1+|\nabla g(x')|^2}}.
\]
The map
\[
q\longmapsto
\frac{(-q,1)}{\sqrt{1+|q|^2}}
\]
is smooth and has a smooth inverse on the relevant open
hemisphere. Consequently, tangential differentiability of $\nu$ at
$\xi=X(x_0')$ is equivalent to differentiability of
$\nabla g$ at $x_0'$. 

Thus, $\nabla g$ is differentiable on the set
\[
    \mathcal{V}\vcentcolon = X^{-1}((\partial\Omega\setminus N)\cap \mathcal{U}).
\]
For all $x'_0\in \mathcal{V}$, we have
\begin{equation}
\label{eq:gradient-g-second-order}
\nabla g(x_0'+h)
=
\nabla g(x_0')+D^2g(x_0')h+o(|h|)
\qquad\text{as }h\to 0.
\end{equation}
Moreover, $g\in C^{1,1}(V)$ implies
\begin{equation}
\label{eq:fund-calc}
g(x_0'+h)-g(x_0')
=
\int_0^1\nabla g(x_0'+th)\cdot h\,dt
\end{equation}
whenever the segment
$[x_0',x_0'+h]$ is contained in $V$. Note that the expansion
\eqref{eq:gradient-g-second-order}, applied with $th$ in place of
$h$, is uniform for $t\in[0,1]$. Indeed, the differentiability of $\nabla g$ at $x'_0$ gives
\[
\nabla g(x_0'+k)
=
\nabla g(x_0')+D^2g(x_0')k+r(k),
\qquad
r(k)=o(|k|).
\]
Thus, for any $\delta>0$, we have
\[
|r(k)|\leq\delta|k|
\]
whenever $|k|$ is sufficiently small. Hence, for sufficiently small
$h$ and every $t\in[0,1]$,
\[
|r(th)|\leq\delta t|h|\leq\delta|h|.
\]
Consequently,
\[
\sup_{0\leq t\leq1}|r(th)|=o(|h|),
\]
yielding the asserted uniform convergence. Combining this convergence result with \eqref{eq:fund-calc} yields
\begin{equation}
\label{eq:g-second-order-expansion}
g(x_0'+h)
=
g(x_0')+\nabla g(x_0')\cdot h
+\frac12D^2g(x_0')[h,h]+o(|h|^2).
\end{equation}
Thus the boundary admits a second-order expansion at every point of
$\partial\Omega\setminus N$.

Fix $\xi=X(x_0')\in \partial\Omega\setminus N$ and
$\tau\in T_\xi\partial\Omega$. Since
\[
DX(x_0'):
\mathbb R^{n-1}\to T_\xi\partial\Omega
\]
is an isomorphism, there exists a unique $v\in\mathbb R^{n-1}$ such that $DX(x_0')v=\tau$. For sufficiently small $\varepsilon>0$, define
\[
\gamma:(-\varepsilon,\varepsilon)\to\partial\Omega,
\qquad
\gamma(s)\vcentcolon = X(x_0'+sv).
\]
The expansion \eqref{eq:g-second-order-expansion} shows that
$\gamma$ is twice differentiable at $s=0$, with
\[
\gamma(0)=\xi,
\qquad
\gamma'(0)=\tau,
\qquad
\gamma''(0)=D^2X(x_0')[v,v].
\]
In particular,
\begin{equation}
\label{eq:curve-second-order-expansion}
\gamma(s)-\xi
=
s\tau+\frac{s^2}{2}\gamma''(0)+o(s^2)
\end{equation}
and
\begin{equation}
\label{eq:curve-distance-expansion}
|\gamma(s)-\xi|^2
=
s^2|\tau|^2+o(s^2).
\end{equation}

Applying \eqref{eq:quadratic-boundary-separation} with
$x=\xi$ and $y=\gamma(s)$ gives
\[
\nu(\xi)\cdot\bigl(\gamma(s)-\xi\bigr)
\leq
-\kappa|\gamma(s)-\xi|^2.
\]
Since $\nu(\xi)\cdot\tau=0$, equations
\eqref{eq:curve-second-order-expansion} and
\eqref{eq:curve-distance-expansion} yield
\[
\frac{s^2}{2}\nu(\xi)\cdot\gamma''(0)+o(s^2)
\leq
-\kappa s^2|\tau|^2+o(s^2).
\]
Dividing by $s^2/2$ and letting $s\to0$, we obtain
\begin{equation}
\label{eq:normal-curve-bound}
-\nu(\xi)\cdot\gamma''(0)
\geq
2\kappa|\tau|^2.
\end{equation}

Finally, differentiating
\[
\nu(\gamma(s))\cdot\gamma'(s)=0
\]
at $s=0$ gives
\[
D_\tau\nu(\xi)\cdot\tau
+
\nu(\xi)\cdot\gamma''(0)=0.
\]
Therefore
\[
\mathrm{II}_\xi(\tau,\tau)
=
D_\tau\nu(\xi)\cdot\tau
=
-\nu(\xi)\cdot\gamma''(0).
\]
Combining this identity with \eqref{eq:normal-curve-bound} proves
\eqref{eq:II-lower}.
\end{proof}

  \begin{proof}[Proof of Lemma~\ref{lem:defining-function}]
Assume first that \ref{unif-convex-2} holds. Let us define the function
\[
\psi\colon U\to\R,\quad \psi(x)\vcentcolon =\rho(x)-\frac{c_0}{2}|x|^2.
\]
Since $\rho\in C^{1,1}(U)$, its almost-everywhere Hessian $D^2 \rho$ agrees with its weak Hessian (see \cite[Section~3.1.2]{evans1991measure}). Hence, the function $\psi\in C^{1,1}(U)$ satisfies
\[
D^2\psi=D^2\rho-c_0I\geq0
\qquad\text{a.e. in }U,
\]
and therefore also in the sense of distributions. Since $U$ is
convex, it follows that $\psi$ is convex. 

We briefly justify the last implication. Since
$\psi\in C^{1,1}(U)$, \cite[Section~4.3.2]{evans1991measure} ensures that
$\psi\in W_{\mathrm{loc}}^{2,\infty}(U)$, and by Rademacher's theorem its almost-everywhere Hessian represents its weak Hessian. Fix $x,y\in U$. Since $U$ is an open convex set, the compact segment $[x,y]$ is contained in $U$, and hence
\[
d\vcentcolon = \operatorname{dist}([x,y],\mathbb R^n\setminus U)>0.
\]
For $0<\varepsilon<d/2$, let
\[
\psi_\varepsilon\vcentcolon = \psi*\eta_\varepsilon
\]
on a neighborhood of $[x,y]$, where $\eta_\varepsilon$ is a standard nonnegative mollifier. Then
\[
D^2\psi_\varepsilon
=
(D^2\psi)*\eta_\varepsilon\geq0
\]
pointwise. Thus, the smooth function
\[
h_\varepsilon(t)
\vcentcolon = 
\psi_\varepsilon((1-t)x+ty),
\qquad 0\leq t\leq1,
\]
satisfies
\[
h_\varepsilon''(t)
=
(y-x)^TD^2\psi_\varepsilon((1-t)x+ty)(y-x)
\geq0.
\]
It is well-known that this implies 
\[
\psi_\varepsilon((1-t)x+ty)
\leq
(1-t)\psi_\varepsilon(x)+t\psi_\varepsilon(y),\quad\text{for all }0\leq t\leq 1.
\]
By \cite[Section~4.2.1]{evans1991measure} and $\psi\in C(U)$, we have $\psi_\eps\to \psi$ uniformly on compact subsets of $U$, and hence passing to the limit as $\varepsilon\downarrow0$ yields
\[
\psi((1-t)x+ty)
\leq
(1-t)\psi(x)+t\psi(y).
\]
Since $x,y\in U$ and $t\in[0,1]$ were arbitrary, $\psi$ is convex.

As a consequence of the convexity of $\psi$, Remark~\ref{rem:rel-subgrad-grad} implies that
\[
\psi(y)\geq \psi(x)+\nabla \psi(x)\cdot(y-x),
\qquad x,y\in U.
\]
Substituting the definition of $\psi$ and rearranging, we obtain
\begin{equation}
\label{eq:strong-convexity-rho}
\rho(y)\geq
\rho(x)+\nabla\rho(x)\cdot(y-x)
+\frac{c_0}{2}|y-x|^2,
\qquad x,y\in U.
\end{equation}
Differentiable functions satisfying an estimate of this type are said to be strongly convex and they are in particular convex. Thus, every sublevel set of $\rho$ is convex, and therefore
$\Omega=\{\rho<0\}$ is convex.

Moreover, $\nabla\rho$ does not vanish on $\partial\Omega$. Indeed, if $x\in\partial\Omega$ and $\nabla\rho(x)=0$, then
\eqref{eq:strong-convexity-rho} would imply
\[
\rho(y)\geq\frac{c_0}{2}|y-x|^2>0
\qquad
\text{for every }y\in U\setminus\{x\},
\]
contradicting $\rho|_{\partial\Omega}=0$.

Furthermore, the conditions \eqref{eq:strong-defining-function} imply that $\nabla \rho$ points in the exterior normal direction on the boundary:
\begin{equation}
\label{eq:gradient-defining-fct}
\nabla\rho(x)=|\nabla\rho(x)|\nu(x),
\qquad x\in\partial\Omega.
\end{equation}
Indeed, since $\rho\in C^{1,1}(U)$ and $\rho|_{\partial\Omega}=0$, the tangential components of $\nabla \rho$ are identically zero and only the normal component of $\nabla \rho$ survives. Using $\nabla \rho\neq 0$ on $\partial\Omega$, we find
\[
    \frac{\nabla \rho(x)}{|\nabla \rho(x)|}=\pm \nu(x).
\]
Suppose that
\begin{equation}
\label{eq:contradiction-sign}
    \nabla \rho(x)=-|\nabla \rho(x)|\nu(x).
\end{equation}
Then $\rho(x)=0$, $x\in \partial\Omega$, implies that
\[
    \rho(x-s\nu(x))\geq s|\nabla \rho(x)|+\frac{c_0}{2}s^2>0,
\]
for all sufficiently small $s>0$. This contradicts $\rho(y)<0$ in $\Omega$, since $x-s\nu(x)\in \Omega$ for small $s>0$. 

Hence, we conclude the validity of \eqref{eq:gradient-defining-fct}. Therefore, recalling that $\nabla \rho\neq 0$ on $\partial\Omega$, this also establishes the representation formula \eqref{eq:rep-normal} of the normal field $\nu$.

For $x\in\partial\Omega$ and $y\in\overline\Omega$, \eqref{eq:strong-convexity-rho} and $\rho(y)\leq 0 = \rho(x)$ yield
\[
0\geq\rho(y)
\geq\nabla\rho(x)\cdot(y-x)
   +\frac{c_0}{2}|y-x|^2.
\]
Consequently, \eqref{eq:gradient-defining-fct} ensures that
\[
\nu(x)\cdot(y-x)
\leq
-\frac{c_0}{2|\nabla\rho(x)|}|y-x|^2.
\]
Since $\partial\Omega$ is compact and $\nabla\rho$ is continuous on the neighborhood $U$,
condition \eqref{eq:geometric-uniform-convexity} follows with
\[
\kappa_\Omega
=\frac{c_0}
       {2\|\nabla\rho\|_{L^{\infty}(\partial\Omega)}}.
\]

Conversely, assume \ref{unif-convex-1}. We first show that the geometric condition \eqref{eq:geometric-uniform-convexity} implies that $\Omega$ is convex. For each $x\in\partial\Omega $, define the closed half-space
\[
H_x^-
\vcentcolon = 
\left\{
y\in\mathbb R^n:
\nu(x)\cdot(y-x)\leq0
\right\}.
\]
We claim that
\begin{equation}
\label{eq:representation-om}
\overline\Omega
=
\bigcap_{x\in\partial\Omega}H_x^-.
\end{equation}
The inclusion ``$\subset$'' follows immediately from
\eqref{eq:geometric-uniform-convexity}. To prove the reverse inclusion,
suppose that
\begin{equation}
\label{eq:choice-y}
y\in
\bigcap_{x\in\partial\Omega}H_x^-
\setminus\overline\Omega.
\end{equation}
Since $\overline\Omega$ is compact and $y\notin \overline{\Omega}$, there exists
$x_0\in\overline\Omega$ such that
\[
0<|y-x_0|
=
\dist(y,\overline\Omega).
\]
Necessarily we have $x_0\in\partial\Omega$. Indeed, if $x_0\in \Omega$, then there exists $0<\eps<1$ such that $x_\eps\vcentcolon = x_0+\eps (y-x_0)\in \Omega$ and 
\[
    |y-x_\eps|=(1-\eps)|y-x_0|<|y-x_0|=\dist(y,\overline{\Omega}),
\]
contradicting minimality of $|y-x_0|$.

Let $\tau\in T_{x_0}\partial\Omega$ be a tangent vector and choose a $C^1$ curve $\gamma\colon(-\varepsilon,\varepsilon)\to\partial\Omega$ satisfying
\[
\gamma(0)=x_0,
\qquad
\gamma'(0)=\tau.
\]
Since $x_0$ minimizes the distance from $y$ to $\overline\Omega$, the function $s\to |y-\gamma(s)|^2$ has a minimum at $s=0$. Therefore,
\[
0
=
\left.\frac{d}{ds}\right|_{s=0}|y-\gamma(s)|^2
=
-2(y-x_0)\cdot\tau.
\]
Since this holds for every
$\tau\in T_{x_0}\partial\Omega$, the vector $y-x_0$ is normal to $\partial\Omega$ at $x_0$. Consequently,
\[
y-x_0=t\nu(x_0)
\]
for some $t\in\mathbb R\setminus\{0\}$. We claim that $t>0$.
Indeed, if $t<0$, then $y-x_0$ points in the inward normal
direction. Since $\partial\Omega$ is of class $C^1$, it follows that
\[
x_0+\varepsilon(y-x_0)\in\Omega
\]
for all sufficiently small $\varepsilon>0$. But then
\[
\bigl|y-\bigl(x_0+\varepsilon(y-x_0)\bigr)\bigr|
=(1-\varepsilon)|y-x_0|
<\operatorname{dist}(y,\overline\Omega),
\]
which is impossible. Hence $t>0$, and therefore
\[
\nu(x_0)\cdot(y-x_0)=t>0.
\]
Thus $y\notin H_{x_0}^-$, contradicting the choice of $y$. This
proves \eqref{eq:representation-om}.

Now, since the right-hand side of \eqref{eq:representation-om} is an intersection of closed half-spaces, $\overline\Omega$ is convex. Moreover, because $\Omega$ is a domain with $C^1$ boundary,
\[
    \Omega=\operatorname{int}(\overline\Omega).
\]
The interior of a convex set is convex, and hence $\Omega$ is
convex.

We now proceed with the construction of the defining function $\rho$. Fix a point $x_\ast\in\Omega$ and define
\[
\Omega_*\vcentcolon = \Omega-x_\ast.
\]
For notational convenience, we continue to denote by $\nu$ the
outward unit normal to $\partial\Omega_*$. It is related to the
original normal field on $\partial\Omega$ by translation.

The translated set $\Omega_*$ contains the origin, $0\in \Omega_*$. Let
\begin{equation}
\label{eq:Minkowski}
p\colon \R^n\to [0,\infty],\qquad p(z)\vcentcolon = \inf\{t>0:t^{-1}z\in \Omega_*\}
\end{equation}
be the Minkowski functional induced by $\Omega_*$. By \cite[Lemma~1.2]{Brezis}, $p$ has the following properties:
\begin{enumerate}[(a)]
    \item\label{vanishing} \textit{Vanishing at the origin}: $p(0)=0$,
    \item\label{pos-hom} \textit{Positive homogeneity}: $p(\lambda x)=\lambda p(x)$, for all $x\in \R^n$ and $\lambda>0$.
    \item\label{sub-add} \textit{Subadditivity}: $p(x+y)\leq p(x)+p(y)$, for all $x,y\in\R^n$.
    \item\label{continuity-p} \textit{Continuity}: There exists $C>0$ such that $0\leq p(x)\leq C|x|$ for all $x\in\R^n$. 
    \item\label{unit-p-ball} \textit{Unit ball}: $\Omega_*=\{p<1\}$.
\end{enumerate}
The properties \ref{pos-hom}, \ref{continuity-p}, and \ref{unit-p-ball} ensure that 
\begin{equation}
\label{eq:Minkowski-on-bdry}
    \partial \Omega_*=\{p=1\}.
\end{equation}
The inclusion $\partial \Omega_*\subset \{p=1\}$ is immediate, since if $p(z)>1$ or $p(z)<1$ for some $z\in \partial\Omega_*$, then, by continuity \ref{continuity-p}, we would have $p<1$ or $p>1$ on some open ball $B$ containing $z$. In each case, \ref{unit-p-ball} yields a contradiction. On the other hand, if $p(z)=1$, then \ref{pos-hom} implies
\[
p((1-\varepsilon)z)<1<p((1+\varepsilon)z),
\]
so every neighborhood of $z$ meets both $\Omega_\ast$ and its
complement. Thus, $z\in \partial\Omega_*$, which shows $\{p=1\}\subset \partial\Omega_*$.

Furthermore, since $0\in\Omega_\ast$ and $\Omega_\ast$ is bounded, there exist constants $0<r<R<\infty$ such that
\begin{equation}
\label{eq:balls-omega-star}
\overline B_r
\subset
\Omega_\ast
\subset
\overline B_R.
\end{equation}
Let $\xi\in\partial\Omega_\ast$. Since the uniform convexity condition \eqref{eq:geometric-uniform-convexity} is invariant under translations, we again have
\begin{equation}
\label{eq:translated-unif-conv}
\nu(\xi)\cdot(z-\xi)\leq -\kappa_\Omega |z-\xi|^2
\qquad
\text{for every }z\in\overline\Omega_\ast.
\end{equation}
Taking $z=r\nu(\xi)\in\overline B_r\subset\Omega_\ast$, we obtain
\[
r-\xi\cdot\nu(\xi)
=
\nu(\xi)\cdot\bigl(r\nu(\xi)-\xi\bigr)
\leq0.
\]
On the other hand, since
$\xi\in\partial\Omega_\ast\subset\overline B_R$,
the Cauchy--Schwarz inequality gives
\[
\xi\cdot\nu(\xi)\leq|\xi|\leq R.
\]
Consequently,
\begin{equation}
\label{eq:gauge-transversality}
r\leq\xi\cdot\nu(\xi)\leq R,
\qquad
\xi\in\partial\Omega_\ast.
\end{equation}

Next, let us define the exit time functional $\ell\colon \mathbb S^{n-1}\to [0,\infty)$ by
\begin{equation}
\label{eq:exit-time}
\ell(\theta)
\vcentcolon = 
\sup\left\{
t>0:t\theta\in\Omega_\ast
\right\},\quad \theta\in \mathbb{S}^{n-1}.
\end{equation}
Then condition \eqref{eq:balls-omega-star} implies that
\begin{equation}
\label{eq:bounds-l}
r\leq\ell(\theta)\leq R.
\end{equation}
Moreover, convexity of $\Omega_\ast$,
$0\in\Omega_\ast$, and definition \eqref{eq:exit-time} imply that
\[
\Omega_\ast\cap\{t\theta:t\geq0\}
=
\{t\theta:0\leq t<\ell(\theta)\}.
\]
Consequently, every ray emanating from the origin intersects
$\partial\Omega_\ast$ in precisely one point, namely
$\ell(\theta)\theta$. Conversely, for every boundary point $\xi\in\partial\Omega_*$, there is a unique ray starting at the origin and passing through that point, which necessarily has velocity vector $\theta \vcentcolon =\xi/|\xi|\in\mathbb S^{n-1}$. 

Thus the radial projection
\begin{equation}
\label{eq:radial-projection}
\pi\colon\partial\Omega_\ast\to\mathbb S^{n-1},
\qquad
\pi(\xi)\vcentcolon =\frac{\xi}{|\xi|},
\end{equation}
is bijective, with $\pi^{-1}(\theta)=\ell(\theta)\theta$. 

We next verify the regularity of $\pi^{-1}$. For
$\xi\in\partial\Omega_\ast$, set $\widehat\xi\vcentcolon = \frac{\xi}{|\xi|}$. The transversality estimate \eqref{eq:gauge-transversality} and the bound
$|\xi|\leq R$ give
\begin{equation}
\label{eq:uniform-radial-transversality}
\widehat\xi\cdot\nu(\xi)
=
\frac{\xi\cdot\nu(\xi)}{|\xi|}
\geq\frac rR>0.
\end{equation}
In particular, the radial direction $\widehat\xi$ does not belong to
$T_\xi\partial\Omega_\ast$.

The differential of the smooth map
\begin{equation}
\label{eq:def-phi}
    \Phi\colon \R^n\setminus\{0\}\to \R^n,\quad \Phi(x)=\frac{x}{|x|}
\end{equation}
at the point $\xi\in \partial\Omega_*$ is
\begin{equation}
\label{eq:diff-Phi}
D\Phi(\xi)v
=
\frac1{|\xi|}
\left(I-\widehat\xi\otimes\widehat\xi\right)v,\quad v\in\R^n,
\end{equation}
which is the orthogonal projection onto $\widehat \xi^\perp$ scaled by $|\xi|^{-1}$. Hence its kernel is $\operatorname{span}\{\xi\}$. Using the chain rule, we see that $D\pi(\xi) \colon
T_\xi\partial\Omega_\ast
\to
T_{\widehat\xi}\mathbb S^{n-1}$ is the restriction of $D\Phi(\xi)$ to $T_\xi\partial\Omega_\ast$. Since \eqref{eq:uniform-radial-transversality} implies that $D\pi(\xi)$ is injective, the rank--nullity theorem yields that $D\pi(\xi)$ is an isomorphism. Indeed, let $\tau\in T_\xi\partial\Omega_\ast$ and decompose $\widehat \xi$ as
\[
\widehat\xi
=
\bigl(\widehat\xi\cdot\nu(\xi)\bigr)\nu(\xi)
+
\widehat\xi_T,
\qquad
\widehat\xi_T\in T_\xi\partial\Omega_\ast.
\]
Then $\tau\cdot\nu(\xi)=0$ and $\widehat \xi \in\mathbb S^{n-1}$ yield
\[
\widehat\xi\cdot\tau
=
\widehat\xi_T\cdot\tau\quad\text{and}\quad|\widehat\xi_T|^2
=
1-\bigl(\widehat\xi\cdot\nu(\xi)\bigr)^2.
\]
Consequently, the Cauchy--Schwarz inequality gives
\[
|\widehat\xi\cdot\tau|^2=|\widehat\xi_T\cdot \tau|^2
\leq
\left(
1-\bigl(\widehat\xi\cdot\nu(\xi)\bigr)^2
\right)|\tau|^2.
\]
It follows that
\begin{align*}
\left|
\left(I-\widehat\xi\otimes\widehat\xi\right)\tau
\right|^2&=|\tau-(\widehat \xi\cdot\tau)\widehat\xi|^2\\
&=
|\tau|^2-|\widehat\xi\cdot\tau|^2\\
&\geq
\bigl(\widehat\xi\cdot\nu(\xi)\bigr)^2|\tau|^2.
\end{align*}
Using \eqref{eq:uniform-radial-transversality}, we therefore obtain
\[
|D\pi(\xi)\tau|
=
\frac{1}{|\xi|}
\left|
\left(I-\widehat\xi\otimes\widehat\xi\right)\tau
\right|
\geq
\frac{r}{R^2}|\tau|.
\]
In particular, the linear map $D\pi(\xi)
\colon
T_\xi\partial\Omega_\ast
\to
T_{\widehat\xi}\mathbb S^{n-1}$ is injective, as we have asserted.

Since $\partial\Omega_\ast$ is of class $C^{1,1}$ and map $\Phi$ is smooth, the radial projection $\pi$ is a $C^{1,1}$ map between $C^{1,1}$ manifolds. This can also be seen directly from formula \eqref{eq:diff-Phi}. Moreover, as $D\pi(\xi)$, $\xi\in \partial\Omega_*$, is an isomorphism, the
$C^{1,1}$ inverse function theorem, applied in local boundary charts, shows that $\pi$ admits a locally $C^{1,1}$ inverse near every point. Because $\pi$ is bijective, these local inverses agree with its set-theoretic inverse
\[
\pi^{-1}(\theta)=\ell(\theta)\theta,
\qquad
\theta\in\mathbb S^{n-1}.
\]
Consequently, $\pi^{-1}\colon
\mathbb S^{n-1}\to\partial\Omega_\ast$ is of class $C^{1,1}$. Since $\ell(\theta)=\pi^{-1}(\theta)\cdot\theta$, it follows that $\ell\in C^{1,1}(\mathbb S^{n-1})$.

Finally, using the definitions of the Minkowski functional $p$ and the exit time functional $\ell$ (see \eqref{eq:Minkowski}, \eqref{eq:exit-time}), together with property positive homogeneity of $p$ \ref{pos-hom}, we find 
\begin{equation}
\label{eq:gauge-radial-representation}
p(z)
=|z|p(z/|z|)=\frac{|z|}{\ell(z/|z|)},\qquad z\neq 0.
\end{equation}
Since $\ell\in C^{1,1}(\mathbb S^{n-1})$ and
$\ell\geq r>0$ (see~\eqref{eq:bounds-l}), formula
\eqref{eq:gauge-radial-representation} shows that $p\in C^{1,1}_{\mathrm{loc}}
\bigl(\mathbb R^n\setminus\{0\}\bigr)$.

Next, let us introduce the function
\[
F\colon\R^n\to [0,\infty],\quad F(z)\vcentcolon = \frac12p(z)^2.
\]
Then \ref{pos-hom} implies that $F$ is positively homogeneous of degree two. To prove that $F\in C^{1,1}(\mathbb R^n)$, note that $p\in C^{1,1}_{\rm loc}(\R^n\setminus\{0\})$ implies that $\nabla F\in C^{0,1}_{\rm loc}(\R^\setminus\{0\})$. Thus, the positive homogeneity of $\nabla F$ yields the existence of a constant $C'>0$ such that
\begin{equation}
\label{eq:gradient-F-est}
|\nabla F(z)|\leq C'|z|.
\end{equation}
Thus $\nabla F$ extends continuously to the origin by 
\begin{equation}
\label{eq:nabla-F-origin}
    \nabla F(0)=0.
\end{equation}
Indeed, since $F(0)=0$ and
\[
0\leq F(z)=\frac12p(z)^2\leq\frac{C^2}{2}|z|^2
\]
(see \ref{vanishing} and \ref{continuity-p}), the function $F$ is differentiable at the origin with
$\nabla F(0)=0$. Moreover, the one-homogeneity of $\nabla F\in C(\R^n)$ and the local Lipschitz continuity
\[
\nabla F\in C_{\mathrm{loc}}^{0,1}
(\mathbb R^n\setminus\{0\})
\]
give $\nabla F\in C^{0,1}(\R^n)$. Indeed, let $x,y\in\mathbb R^n$, and put $M=\max\{|x|,|y|\}$. If $M=0$, then 
\[
    |\nabla F(x)-\nabla F(y)|=0\leq C |x-y|
\]
for any $C>0$ and hence we may assume that $M>0$. If $|x-y|\geq M/2$, then \eqref{eq:gradient-F-est}--\eqref{eq:nabla-F-origin} yield
\[
|\nabla F(x)-\nabla F(y)|
 \leq C'(|x|+|y|)\leq 2C'M
 \leq 4C'|x-y|.
\]
If $|x-y|\leq M/2$, then both $x/M$ and $y/M$ belong to the annulus
\[
    \mathcal{A}\vcentcolon = \{z\in\R^n:\,1/2\leq |z|\leq 1\}.
\]
The lower bound follows from the triangle inequality and the choice $M=\max(|x|,|y|)$. Thus positive homogeneity of $\nabla F$ and $\nabla F\in C^{0,1}_{\rm loc}(\R^n\setminus\{0\})$ yield
\[
\begin{aligned}
|\nabla F(x)-\nabla F(y)|
&=M\left|
\nabla F\left(\frac{x}{M}\right)
-\nabla F\left(\frac{y}{M}\right)
\right| \\
&\leq C_{\mathcal A}|x-y|.
\end{aligned}
\]
Therefore $\nabla F$ is globally Lipschitz and
$F\in C^{1,1}(\mathbb R^n)$.

Thus, it remains to show that $F$ is uniformly convex, that is, it satisfies 
\begin{equation}
\label{eq:uniform-conv-F}
    D^2 F\geq cI\quad\text{a.e. in }\R^n.
\end{equation}

By \eqref{eq:Minkowski-on-bdry}, $F$ satisfies $F=1/2$ on
$\partial \Omega_*$ and hence the tangential components of $\nabla F$ vanish on $\partial \Omega_*$. This shows that $\nabla F|_{\partial\Omega_*}$ points into the normal direction to $\partial \Omega_*$, that is, for all $\xi\in \partial\Omega_*$,
\begin{equation}
\label{eq:nablaF-bdry}
    \nabla F(\xi)=s(\xi)\nu(\xi)\quad\text{for some }s(\xi)\in\R.
\end{equation}
Since $F$ is two-homogeneous, Euler's theorem gives
\begin{equation}
\label{eq:Eulers-id}
    2F(z)=\nabla F(z)\cdot z,\qquad z\neq 0.
\end{equation}
Thus, $F=1/2$ on $\partial\Omega_*$ and \eqref{eq:nablaF-bdry} imply 
\[
    1=s(\xi)(\nu(\xi)\cdot \xi).
\]
Thus, \eqref{eq:uniform-radial-transversality} yields
\begin{equation}
\label{eq:gradient-gauge}
    \nabla F(\xi)=\frac{\nu(\xi)}{\xi\cdot\nu(\xi)},\qquad \xi\in\partial \Omega_*.
\end{equation}
By homogeneity,
\begin{equation}
\label{eq:gradient-gauge-ray}
\nabla F(t\xi)
=t\frac{\nu(\xi)}{\xi\cdot\nu(\xi)},
\qquad t>0,\quad \xi\in\partial \Omega_*.
\end{equation}

Now, let $N\subset\partial \Omega_*$ be the set from Lemma~\ref{lem:second-order-boundary-geometry}. We already know that $\mathcal H^{n-1}(N)=0$ and hence the product
$[1/j,j]\times N$ has zero $n$-dimensional Hausdorff measure. Now, for each $j\in\mathbb N$, the restriction of the radial parametrization $\Psi$,
\[
\Psi|_{[1/j,j]\times\partial \Omega_*}\colon [1/j,j]\times\partial \Omega_*\to \R^n,\quad (t,\xi)\mapsto t\xi,
\] 
is Lipschitz, so the image of $[1/j,j]\times N$ has zero
$n$-dimensional Lebesgue measure \cite[Chapter~2]{evans1991measure}. Taking the union over $j$ shows
that the cone generated by $N$ is Lebesgue-null, that is, 
\begin{equation}
\label{eq:cone-gen-N}
    \mathcal{L}^n(\{t\xi:t>0,\xi\in N\})=0.
\end{equation}
Thus, if we show that $F$ is twice differentiable at every point on the ray generated by a good point $\xi \in \partial\Omega_*\setminus N$, then by \eqref{eq:cone-gen-N} $F$ is twice differentiable almost everywhere in $\R^n$.

To prove differentiability on the good rays, define
\[
A(\xi)\vcentcolon = 
\frac{\nu(\xi)}{\xi\cdot\nu(\xi)},
\qquad \xi\in\partial\Omega_\ast.
\]
Then equation~\eqref{eq:gradient-gauge-ray} becomes
\begin{equation}
\label{eq:gradient-radial-coordinates}
\nabla F(t\xi)=tA(\xi).
\end{equation}
Next fix $\xi\in\partial\Omega_\ast\setminus N$. Since $\nu$ is
tangentially differentiable at $\xi$ and
$\xi\cdot\nu(\xi)\geq r>0$, the map $A$ is tangentially
differentiable at $\xi$. More precisely, for every
$\tau\in T_\xi\partial\Omega_\ast$,
\[
D_{\partial\Omega_\ast}A(\xi)[\tau]
=
\frac{D_\tau\nu(\xi)}{\xi\cdot\nu(\xi)}
-
\frac{\xi\cdot D_\tau\nu(\xi)}
     {(\xi\cdot\nu(\xi))^2}\nu(\xi),
\]
where we have used $\tau\cdot\nu(\xi)=0$. Let $X\colon V\subset\mathbb R^{n-1}
\to\partial\Omega_\ast$ be a local $C^{1,1}$ parametrization with $X(x'_0)=\xi$ (cf.~\eqref{eq:local-rep}), and set
\[
\widetilde\Psi\colon (0,\infty)\times V\to \R^n,\quad\widetilde\Psi(s,x')\vcentcolon = sX(x').
\]
Its differential satisfies
\[
D\widetilde\Psi(t,x'_0)[a,h]
=
a\xi+tDX(x'_0)h.
\]
Since $\mathbb R^n
=
\operatorname{span}\{\xi\}
\oplus T_\xi\partial\Omega_\ast$
this differential is invertible and hence
$\widetilde\Psi$ is a local $C^{1,1}$ diffeomorphism at $(t,x'_0)$. In these coordinates, \eqref{eq:gradient-radial-coordinates} becomes
\[
\nabla F\bigl(\widetilde\Psi(s,x')\bigr)
=
sA(X(x'))
=:\widetilde G(s,x').
\]
The map $\widetilde G$ is differentiable at $(t,x'_0)$ because $X$ is $C^{1,1}$ at $x'_0$, $\nu\circ X$ is differentiable at $x'_0$, and $\xi\cdot\nu(\xi)\geq r>0$. Therefore 
\[
\nabla F
=
\widetilde G\circ\widetilde\Psi^{-1}
\]
is differentiable at $t\xi=\widetilde\Psi(t,x'_0)$. Consequently, $F$ is twice differentiable at $t\xi$, as we wanted to see.

Finally, let $\gamma$ be the twice differentiable curve from Lemma~\ref{lem:second-order-boundary-geometry} passing through a point $\xi\in \partial\Omega_*\setminus N$ into the direction $\tau\in T_\xi \partial\Omega_*$. Since $F=\frac12$ on $\partial\Omega_\ast$, we have $F(\gamma(s))=\frac12$, for small $s$. Since $F$ is twice differentiable at $\xi$, differentiating this identity twice at $s=0$ gives
\[
0
=
D^2F(\xi)[\tau,\tau]
+
\nabla F(\xi)\cdot\gamma''(0).
\]
Using \eqref{eq:boundary-curve-second-order}, we obtain
\begin{equation}
\label{eq:lower-bound-tangential-second}
\begin{split}
D^2F(\xi)[\tau,\tau]
&=
-\frac{\nu(\xi)\cdot\gamma''(0)}
       {\xi\cdot\nu(\xi)}=
\frac{\mathrm{II}_\xi(\tau,\tau)}
     {\xi\cdot\nu(\xi)}\geq 
\frac{2\kappa_\Omega}{R}|\tau|^2,
\qquad
\tau\in T_\xi\partial\Omega_\ast.
\end{split}
\end{equation}
Moreover, differentiating Euler's identity \eqref{eq:Eulers-id} at $\xi\in \partial\Omega_*\setminus N$ gives 
\[
    D^2F(\xi)\xi=\nabla F(\xi).
\]
Hence, \eqref{eq:gradient-gauge} implies
\begin{equation}
\label{eq:second-ord-xi-xi}
D^2F(\xi)[\xi,\xi]=1,
\qquad
D^2F(\xi)[\xi,\tau]=0
\quad
\text{for }\tau\in T_{\xi}\partial \Omega_*.
\end{equation}
Now, let $v\in\mathbb R^n$ be an arbitrary vector. By \eqref{eq:gauge-transversality}, we have $\xi\notin T_\xi \partial\Omega_*$ and thus
\[
\mathbb R^n
=
\operatorname{span}\{\xi\}
\oplus
T_\xi\partial\Omega_\ast.
\]
Consequently, $v\in\R^n$ can be decomposed as
\begin{equation}
\label{eq:radial-tangential-decomposition}
v=a\xi+\tau,
\qquad
a\in\mathbb R,
\quad
\tau\in T_\xi\partial\Omega_\ast.
\end{equation}
More explicitly,
\[
a
=
\frac{v\cdot\nu(\xi)}
{\xi\cdot\nu(\xi)},
\qquad
\tau=v-a\xi.
\]
The transversality estimate
\eqref{eq:gauge-transversality} gives
\[
|a|
\leq
\frac{|v|}{r},
\qquad
|\tau|
\leq
|v|+|a||\xi|
\leq
\left(1+\frac Rr\right)|v|.
\]
Thus the decomposition is uniformly nondegenerate.

Using
\eqref{eq:lower-bound-tangential-second} and
\eqref{eq:second-ord-xi-xi}, we obtain
\begin{align*}
D^2F(\xi)[v,v]
&=
a^2D^2F(\xi)[\xi,\xi]
+2aD^2F(\xi)[\xi,\tau]
+D^2F(\xi)[\tau,\tau]\\
&\geq
a^2+\frac{2\kappa_\Omega}{R}|\tau|^2.
\end{align*}
Set
\[
\lambda_0
\vcentcolon = 
\min\left\{1,\frac{2\kappa_\Omega}{R}\right\}.
\]
Then
\[
D^2F(\xi)[v,v]
\geq
\lambda_0\bigl(a^2+|\tau|^2\bigr).
\]
On the other hand, since $|\xi|\leq R$,
\[
|v|
\leq
|a||\xi|+|\tau|
\leq
R|a|+|\tau|
\leq
\sqrt{R^2+1}\,
\bigl(a^2+|\tau|^2\bigr)^{1/2}.
\]
It follows that
\[
a^2+|\tau|^2
\geq
\frac{|v|^2}{R^2+1}.
\]
Consequently,
\begin{equation}
\label{eq:second-diff-a.e.-bdry-point}
D^2F(\xi)[v,v]
\geq
c|v|^2,
\qquad
c
\vcentcolon = 
\frac{1}{R^2+1}
\min\left\{1,\frac{2\kappa_\Omega}{R}\right\}>0,
\end{equation}
for every $v\in\mathbb R^n$ and every
$\xi\in\partial\Omega_\ast \setminus N$. 

Finally, since $\nabla F$ is positively homogeneous of degree one, we have
\[
\nabla F(tz)=t\nabla F(z),
\qquad t>0.
\]
Now we can differentiate this identity with respect to $z$ at $\xi\in \partial\Omega_*\setminus N$, which yields
\[
D^2F(t\xi)=D^2F(\xi).
\]
Therefore, \eqref{eq:second-diff-a.e.-bdry-point} implies
\[
D^2F(t\xi )\geq cI
\qquad\text{for all }t>0,\,\xi\in\partial\Omega_*\setminus N.
\]
The complement of these good rays is contained in the cone
generated by $N$, together with the origin. By
\eqref{eq:cone-gen-N}, this exceptional set has zero Lebesgue
measure. Consequently,
\[
D^2F\geq cI
\qquad\text{a.e. in }\mathbb R^n.
\]

Finally, let us set
\[
\rho(x)\vcentcolon =2F(x-x_*)-1=p(x-x_\ast)^2-1.
\]
We deduce from the properties of $F$ that $\rho$ belongs to $C^{1,1}(\mathbb R^n)$ and satisfies
\[
\Omega=\{\rho<0\},\qquad
\partial\Omega=\{\rho=0\},\qquad
D^2\rho\geq2cI
\quad\text{a.e. in }\mathbb R^n.
\]
This proves \ref{unif-convex-2}.
\end{proof}

    \bibliography{refs} 

@article{CNS,
  author  = {Caffarelli, Luis and Nirenberg, Louis and Spruck, Joel},
  title   = {The {D}irichlet problem for nonlinear second-order elliptic equations. {I}. {Monge--Amp{\`e}re} equation},
  journal = {Communications on Pure and Applied Mathematics},
  volume  = {37},
  number  = {3},
  year    = {1984},
  pages   = {369--402}
}

@misc{CarsteaGhosh,
  author        = {C{\^a}rstea, C{\u a}t{\u a}lin I. and Ghosh, Tuhin},
  title         = {An inverse source problem for the {Monge--Amp{\`e}re} equation from large boundary data},
  year          = {2026},
  eprint        = {2606.07064},
  archiveprefix = {arXiv},
  url           = {https://arxiv.org/abs/2606.07064},
  note          = {Preprint, arXiv:2606.07064}
}

@book{Figalli,
  author    = {Figalli, Alessio},
  title     = {The {M}onge--{A}mp{\`e}re {E}quation and {I}ts {A}pplications},
  series    = {Zurich Lectures in Advanced Mathematics},
  publisher = {European Mathematical Society},
  address   = {Z{\"u}rich},
  year      = {2017},
  doi       = {10.4171/170},
  url       = {https://doi.org/10.4171/170}
}

@book{Gutierrez,
  author    = {Guti{\'e}rrez, Cristian E.},
  title     = {The {Monge--Amp{\`e}re} Equation},
  edition   = {2},
  series    = {Progress in Nonlinear Differential Equations and Their Applications},
  volume    = {89},
  publisher = {Birkh{\"a}user},
  address   = {Cham},
  year      = {2016}
}

@book{Helgason,
  author    = {Helgason, Sigurdur},
  title     = {The Radon Transform},
  edition   = {2},
  series    = {Progress in Mathematics},
  volume    = {5},
  publisher = {Birkh{\"a}user},
  address   = {Boston},
  year      = {1999}
}

@misc{LiimatainenLin,
  author        = {Liimatainen, Tony and Lin, Yi-Hsuan},
  title         = {An inverse problem for the {Monge--Amp{\`e}re} equation},
  year          = {2025},
  eprint        = {2510.11572},
  archiveprefix = {arXiv},
  url           = {https://arxiv.org/abs/2510.11572},
  note          = {Preprint, arXiv:2510.11572}
}

@book{Rockafellar,
  author    = {Rockafellar, R. Tyrrell},
  title     = {Convex Analysis},
  series    = {Princeton Mathematical Series},
  volume    = {28},
  publisher = {Princeton University Press},
  year      = {1970}
}

@misc{UhlmannZimmermann,
  author        = {Uhlmann, Gunther and Zimmermann, Philipp},
  title         = {Inverse problems for the {Monge--Amp{\`e}re} equation: linearization and nonlinear recovery},
  year          = {2026},
  eprint        = {2606.13939},
  archiveprefix = {arXiv},
  url           = {https://arxiv.org/abs/2606.13939},
  note          = {Preprint; revised manuscript}
}

@book{isakov1990inverse,
  title={Inverse source problems},
  author={Isakov, Victor},
  number={34},
  year={1990},
  publisher={American Mathematical Soc.}
}

@book{aubin2012nonlinear,
  title={Nonlinear analysis on manifolds. {M}onge-{A}mp{\` e}re equations},
  author={Aubin, Thierry},
  volume={252},
  year={2012},
  publisher={Springer Science \& Business Media}
}

@book{le2024analysis,
  title={Analysis of {M}onge--{A}mp{\`e}re {E}quations},
  author={L{\^e}, Nam Q.},
  volume={240},
  year={2024},
  publisher={American Mathematical Society}
}

@book{GT,
  title={Elliptic partial differential equations of second order},
  author={Gilbarg, David and Trudinger, Neil S},
  volume={2},
  number={1},
  year={1998},
  publisher={Springer}
}

@book{evans1991measure,
  title={Measure theory and fine properties of functions},
  author={Evans, Lawrence Craig and Gariepy, Ronald F},
  volume={5},
  year={1991},
  publisher={CRC press}
}

@book{trudinger2015elliptic,
	title={Elliptic partial differential equations of second order},
	author={Gilbarg, David and Trudinger, Neil S},
	year={1983},
	publisher={Second edition, Springer}
}

@book {evans-mt,
	AUTHOR = {Evans, Lawrence C. and Gariepy, Ronald F.},
	TITLE = {Measure theory and fine properties of functions},
	SERIES = {Textbooks in Mathematics},
	EDITION = {Revised},
	PUBLISHER = {CRC Press, Boca Raton, FL},
	YEAR = {2015},
	PAGES = {xiv+299},
	ISBN = {978-1-4822-4238-6},
	MRCLASS = {28-01},
	MRNUMBER = {3409135},
}

@book{Brezis,
	AUTHOR = {Brezis, Haim},
	TITLE = {Functional analysis, {S}obolev spaces and partial differential
	equations},
	SERIES = {Universitext},
	PUBLISHER = {Springer, New York},
	YEAR = {2011},
	PAGES = {xiv+599},
	ISBN = {978-0-387-70913-0},
	MRCLASS = {35-01 (46-01 46E35 46N20 47F05)},
	MRNUMBER = {2759829},
	MRREVIEWER = {Vicen\c{t}iu D. R\u{a}dulescu},
}

@book{federer2014geometric,
  title={Geometric measure theory},
  author={Federer, Herbert},
  year={2014},
  publisher={Springer}
}

@article{AL0,
  author  = {Alexandrov, A. D.},
  title   = {Application of the theorem on domain invariance to
             existence proofs},
  journal = {Izvestiya Akademii Nauk SSSR},
  volume  = {3},
  year    = {1939},
  pages   = {243--256}
}

@article{AL1,
  author  = {Alexandrov, A. D.},
  title   = {Existence and uniqueness of a convex surface with a
             given integral curvature},
  journal = {C. R. (Doklady) Acad. Sci. URSS (N. S.)},
  volume  = {35},
  year    = {1942},
  pages   = {131--134}
}

@article{AL2,
  author  = {Alexandrov, A. D.},
  title   = {Smoothness of the convex surface of bounded
             {G}aussian curvature},
  journal = {C. R. (Doklady) Acad. Sci. URSS (N. S.)},
  volume  = {36},
  year    = {1942},
  pages   = {195--199}
}

@book{AL3,
  author    = {Alexandrov, A. D.},
  title     = {Convex Polyhedra},
  series    = {Springer Monographs in Mathematics},
  publisher = {Springer-Verlag},
  address   = {Berlin},
  year      = {2005},
  note      = {Translated from the 1950 Russian edition by
               N. S. Dairbekov, S. S. Kutateladze, and
               A. B. Sossinsky. With comments and bibliography by
               V. A. Zalgaller and appendices by L. A. Shor and
               Yu. A. Volkov}
}

@article{AL4,
  author  = {Alexandrov, A. D.},
  title   = {Dirichlet's problem for the equation
             {$\det\lVert z_{ij}\rVert
             =\varphi(z_1,\ldots,z_n,z,x_1,\ldots,x_n)$}, {I}},
  journal = {Vestnik Leningrad. Univ. Ser. Mat. Meh. Astr.},
  volume  = {13},
  number  = {1},
  year    = {1958},
  pages   = {5--24},
  note    = {In Russian}
}

@article{rockafellar1966characterization,
  title={Characterization of the subdifferentials of convex functions},
  author={Rockafellar, Ralph},
  journal={Pacific Journal of Mathematics},
  volume={17},
  number={3},
  pages={497--510},
  year={1966},
  publisher={Mathematical Sciences Publishers}
}

    \bibliographystyle{alpha}

\end{document}